\documentclass{amsart}

\usepackage{amssymb}
\usepackage[dvips]{epsfig}
\usepackage{graphicx}
\usepackage{color}

 \usepackage[latin1]{inputenc}
 \usepackage[T1]{fontenc}
 \usepackage[normalem]{ulem}
\usepackage{verbatim}
 \usepackage{graphicx}
\usepackage[latin1]{inputenc}
\usepackage{amsmath}
\usepackage{amssymb}
\usepackage{amsfonts}
\usepackage{amsthm}
\usepackage{bbm}

\newcommand{\dt}{\partial_t}
\renewcommand{\u}{\mathfrak{u}}

\newcommand{\N}{\mathbb{N}}
\newcommand{\R}{\mathbb{R}}

\newcommand{\beno}{\begin{eqnarray*}}
\newcommand{\eeno}{\end{eqnarray*}}
\def\curl{\mathop{\rm curl}\nolimits}
\newcommand{\vv}[1]{\boldsymbol{#1}}

\newtheorem{theorem}{Theorem}[section]
\newtheorem{lemma}[theorem]{Lemma}

\newtheorem{corollary}[theorem]{Corollary}
\theoremstyle{remark}
\newtheorem{remark}{Remark}[section]
\theoremstyle{definition}
\newtheorem{definition}[theorem]{Definition}

\numberwithin{equation}{section}

\begin{document}

\title[Boussinesq systems II]{The Cauchy problem on large time for surface waves type Boussinesq systems II}

\author[J.-C. Saut]{Jean-Claude Saut}
\address{Laboratoire de Math\' ematiques, UMR 8628\\
Universit\' e Paris-Sud et CNRS\\ 91405 Orsay, France}
\email{jean-claude.saut@math.u-psud.fr}

\author[C. Wang]{Chao Wang}
\address{School of Mathematical Sciences\\ Peking University\\ Beijing 100871,China}
\email{wangchao@math.pku.edu.cn}

\author[L. Xu]{Li Xu}
\address{LSEC, Institute of Computational Mathematics\\ Academy of Mathematics and Systems Science, CAS\\
Beijing 100190, China}
\email{ xuliice@lsec.cc.ac.cn}

\date{November 26th, 2015}
\maketitle

\begin{abstract}
This paper is a continuation of a  previous work by two of  the Authors \cite{SX} on long time existence for Boussinesq systems modeling the propagation of long, weakly nonlinear water waves. We provide proofs on examples not considered in \cite{SX} in particular we prove a long time well-posedness result for a delicate "strongly dispersive" Boussinesq system.
\end{abstract}

\section{Introduction}
One aim of this paper is to  complete  the results  obtained in a previous paper {\cite{SX} on the Cauchy theory for some  (a,b,c,d) Boussinesq systems for surface water waves
\begin{equation}
    \label{abcd2}
    \left\lbrace
    \begin{array}{l}
    \eta_t+\nabla \cdot {\vv u}+\epsilon \nabla\cdot(\eta {\vv u})+\mu\lbrack a \nabla\cdot \Delta{\vv u}-b\Delta \eta_t\rbrack=0 \\
    {\vv u}_t+\nabla \eta+\epsilon \frac{1}{2}\nabla |{\vv u}|^2+\mu\lbrack c\nabla \Delta \eta-d\Delta {\vv u}_t\rbrack=0.
\end{array}\right.
    \end{equation}

Here $\mu$ and $\epsilon$ are the small parameters (shallowness and nonlinearity parameters respectively) defined as
$$\mu=\frac{h^2}{\lambda^2}, \quad \epsilon= \frac{\alpha}{h}$$
where  $\alpha$ is a typical amplitude of the wave, $h$ a typical depth and $\lambda$ a typical horizontal wavelength.

In the Boussinesq regime, $\epsilon$ and $\mu$ are supposed to be of same order, $\epsilon\sim\mu\ll1,$ and we will take for simplicity $\epsilon=\mu,$  writing \eqref{abcd2} as
\begin{equation}
    \label{abcd}
    \left\lbrace
    \begin{array}{l}
    \eta_t+\nabla \cdot {\vv u}+\epsilon \lbrack\nabla\cdot(\eta {\vv u})+a \nabla\cdot \Delta{\vv u}-b\Delta \eta_t\rbrack=0 \\
    {\vv u}_t+\nabla \eta+\epsilon\lbrack \frac{1}{2}\nabla |{\vv u}|^2+c\nabla \Delta \eta-d\Delta {\vv u}_t\rbrack=0,
\end{array}\right.
    \end{equation}

The class of systems \eqref{abcd2}, \eqref{abcd} models  water waves on a flat bottom propagating in both directions in the aforementioned  regime (see \cite{BCS1, BCS2, BCL}). We will focus here on the  {\it strongly dispersive} case, corresponding to particular choices of the modeling parameters (a,b,c,d) (see below).


One could also derive similar systems with a non trivial bathymetry (non flat bottom), see \cite{Ch}, and one has then to distinguish between the case when the bottom varies slowly and the case where it is strongly varying. In the former case, \eqref{abcd} has to be slightly modified and becomes
\begin{equation}
    \label{abcdBot}
    \left\lbrace
    \begin{array}{l}
    \eta_t+\nabla \cdot {\vv u}+\epsilon \lbrack\nabla\cdot((\eta-\beta) {\vv u})+a \nabla\cdot \Delta{\vv u}-b\Delta \eta_t\rbrack=0 \\
    {\vv u}_t+\nabla \eta+\epsilon\lbrack \frac{1}{2}\nabla |{\vv u}|^2+c\nabla \Delta \eta-d\Delta {\vv u}_t\rbrack=0,
 \end{array}\right.
    \end{equation}
where $\beta$ is a smooth function on $\R^d, d=1,2,$ bounded together with its derivatives. In this case, the results in \cite{SX} and those of the present paper extend easily.
In the second case one gets  much more complicated systems \cite{Ch}. We refer to \cite{MG} for long time existence results in this case.

\vspace{0.3cm}
Recall (see \cite{BCS1, DD}) that the modeling parameters are constrained by the relation
$$a+b+c+d=\frac{1}{3}-\tau,$$
where $\tau\geq 0$ is the surface tension parameter (Bond number).

Recall also \cite{BCS1} that \eqref{abcd} is linearly well-posed when
$$a\leq 0, c\leq 0, b\geq 0, d\geq 0,$$
and when
$$a=c, b\geq 0, d\geq 0.$$



An important step to justify rigorously \eqref{abcd} as an asymptotic model for water waves is to establish the well-posedness of the Cauchy problem on time scales of order $1/\epsilon,$ with uniform bounds in suitable Sobolev spaces, the error estimate being then  (see \cite {BCL, La1}).
$$||U_{\text{Boussinesq}}-U_{\text{Euler}}||=O(\epsilon^2t)$$
in suitable Sobolev norms.

This step has been established in \cite{SX} (see also \cite{MSZ}) for most of Boussinesq systems with and without surface tension. The idea in \cite{SX} is to find an appropriate symmetrization of the system and  this is not a straightforward task since one cannot obviously use the classical symmetrizer of the underlying  Saint-Venant (shallow water) hyperbolic system. This will be reviewed in the first section of this paper. A  complete proof of cases that were not fully developed in \cite{SX} will be given in  Section 4. \footnote{Due to the large number of cases to be considered, we chosed in \cite{SX} to give complete proofs for a limited number of cases.}

 Introducing surface tension enlarges the range of physically admissible parameters  (a,b,c,d) and so even a local  theory \footnote{That is not taking care of the dependence of the lifespan of the solution with respect to $\epsilon$} for a few linearly well -posed systems is still missing, for instance the cases $b=d=0, a<0, c=0$ and $b=d=0, a=0, c<0$). Both cases will be considered here but  the later leads to serious difficulties and the long time existence for it is the main result of the present paper.

 Note also that the (linearly well-posed)  "exceptional KdV-KdV" case $b=d=0, a=c>0$ which is studied  in \cite{LPS}  leading to well-posedness on time scales of order $1/\sqrt \epsilon $ in Sobolev spaces $H^s(\R^2), s>3/2$ which are larger than the "hyperbolic" one $H^s(\R^2), s>2$ is not covered neither  in \cite{SX} nor in the present paper so that a long time existence is still open in this case \footnote{However, the case $b=d=0, a<0,c<0$ that can only occur with a strong surface tension is covered by the theory in \cite{SX}.}.

An important mathematical issue concerning Boussinesq systems \eqref{abcd} is that despite  they  describe the same dynamics of water waves, their {\it mathematical} properties are rather different, due essentially to their different linear dispersion relations. Of course those dispersion relations all coincide in the long wave limit but there are quite different in the short wave limit. A convenient way to classify the system is according to the order of the Fourier multiplier operator given by the eigenvalues of the linearized operator (see \cite{BCS1}). The order can be $-1,0,1,2$ or $3$. The two last cases are referred to as the {\it strongly dispersive} ones.

After a brief review of our previous results,   we will consider in the third section the ("local") Cauchy problem for two  strongly dispersive Boussinesq systems of {\it Schr\"{o}dinger type}, namely  $b=d=c=0, a<0$ and $a=b=d=0, c<0,$ two situations that are admissible in case of strong surface tension and that have not been considered before. In the first case, it turns out that the local (that is on time scales of order $1/\sqrt \epsilon)$ Cauchy theory  can be obtained by "elementary" energy methods on the original formulation, as in   the purely gravity waves  cases  $a<0, c<0, b=0, d>0$ or $a<0, c<0, b>0, d=0$ considered in \cite{LPS}.  On the other hand, the second case $a=b=d=0, c<0,$ leads to serious difficulties that are explained in this section.

In the fourth  section we provide detailed proofs (not given in \cite{SX}) for the long time well-posedness of the strongly dispersive case  $b=d=c=0, a<0$ and for two systems which can be viewed as {\it weakly dispersive}, namely $ b>0, a=c=d=0$ and $d>0, a=b=c=0.$
We conclude this section by establishing long time existence for the difficult case $a=b=d=0, c<0 $ by a quasilinearization method quite different from the other cases. As was aforementioned this is the main result of the present paper (see Theorems 4.6 and 4.7). We explain first how to get  the needed  a priori estimates, the complete proof being given in the next section.

Finally we show  in Section 6 that the symmetrization method can be used to obtain  long time existence results for a fifth order Boussinesq system and we briefly allude to possible extensions to nonlocal {\it Full dispersion} Boussinesq type systems.

During the completion of the present paper we were informed of the very interesting paper \cite{Bu} where an alternative proof of long time existence for
most of the Boussinesq systems is provided (excluding the "strongly dispersive" ones $b=d=0,$ thus the "difficult case" $a=b=d=0, c<0$). This proof also relaxes the non-cavitation condition on the initial data $\eta_0.$

We were also informed by Vincent Duch\^{e}ne of the article \cite{DIT} which contains in the one-dimensional case (see Appendix A) results related to ours in Subsection 4.4.

\vspace{0.3cm}
\noindent{\bf Notations.} We will denote $|\cdot|_p$ the norm in the Lebesgue space $L^p(\R),\; 1\leq p\leq \infty$ and $\|\cdot\|_s$ the norm in the Sobolev space $H^s(\R^d),\; s\in \R.$ $(\cdot | \cdot)_2$ denotes the scalar product in $L^2.$We will denote $\hat {f}$ or $\mathcal F(f)$ the Fourier transform of a tempered distribution $f.$ For any $s\in \R,$ we define $|D|^s f$ by its Fourier transform $\widehat{|D|^s f}(\xi)=|\xi|^s \hat{f}(\xi).$ We also denote $|D_x|^s f=\mathcal F^{-1}(|\xi_1|\hat {f})$ and $|D_y|^s f=\mathcal F^{-1}(|\xi_2|\hat {f}).$ Finally we will denote $\Lambda=(I-\Delta)^{1/2}$ and $J_\epsilon=(I-\epsilon \Delta)^{1/2}.$

\section{ A review of long time well-posed Boussinesq systems}

As recalled previously, in order to fully justify the Boussinesq systems, one needs to prove the well-posedness of the Cauchy problem on time scales of order  at least $O(1/\epsilon)$ (together with the relevant uniform bounds). This would be achieved of course  if one could obtain the {\it global well-posedness} (also with uniform bounds). This is only known however for a very limited number of Boussinesq systems in one-dimension. A first idea would be  to use appropriate conservation laws, but contrary to the {\it one-directional or quasi one-directional} equations such as the Korteweg- de Vries or the Kadomtsev-Petviashvili equations which are derived in the same regime, the Boussinesq systems do not possess the two invariants ($L^2$ norm and energy) that provide useful a priori bounds.

Nevertheless, when $b=d,$ the Boussinesq systems are endowed with an Hamiltonian structure. More precisely, denoting by $J$ the skew adjoint matrix operator

\begin{displaymath}
J=\begin{pmatrix} 0 & \partial_{x} (I-\epsilon b\Delta)^{-1}& \partial_{y}(I-\epsilon b\Delta)^{-1} \\
                 \partial_{x}(I-\epsilon b\Delta)^{-1} & 0 & 0 \\
                 \partial_{y} (I-\epsilon b\Delta)^{-1}& 0 & 0 \end{pmatrix},
\end{displaymath}
and
\begin{displaymath}
U=\begin{pmatrix} \eta  \\
                {\vv u}
                  \end{pmatrix},
\end{displaymath}
the Boussinesq systems write in this case
\begin{displaymath}
\partial_t U=-J (\text{grad} \, H_{\epsilon})(U),
\end{displaymath}
where $H_{\epsilon}(U)$ is the Hamiltonian given by
\begin{displaymath}
H_{\epsilon}(U)=\frac12 \int_{\mathbb R^2}\big(-c\epsilon|\nabla
\eta|^2-a\epsilon|\nabla {\vv u}|^2+\eta^2+|{\vv u}|^2+\epsilon
\eta|{\vv u}|^2\big)dxdy,
\end{displaymath}
so that $H_{\epsilon}(U)$ is conserved by the flow. This can be used (see \cite{BCS2}) in the one dimensional case where $b=d>0, a\leq 0, c\leq 0$ or $b=d>0, a=0, c<0$ to establish the global well-posedness of the corresponding Boussinesq systems provided   $H_\epsilon(\eta_0, u_0)$ is small enough and the non cavitation condition $\inf_x (1+\epsilon \eta_0(x)) >0$ is satisfied. The proof uses in a crucial way the fact that $b=d>0$ and the Sobolev embedding $H^1(\R)\subset L^{\infty}(\R)$ thus it does not work in two dimensions.

Another one-dimensional situation leading to global well-posedneess is when $a=b=c=0, d>0.$ Then Amick and Schonbeck \cite {A,Sc} use the underlying hyperbolic structure of the shallow-water (Saint-Venant) system to get a priori bounds stemming from an entropy functional. This allows to prove the global well-posedness under the condition $\inf_{x\in \R}(1+\epsilon\eta_0(x))>0$ \footnote{Contrary to what was claimed in \cite{SX}, the results in  \cite{A, Sc} do not need a smallness assumption on the initial data.} but again the extension of this result to the two-dimensional case is unclear. We will prove in this case the large time existence in Section 4.

As far as  long time results are concerned, it has been claimed in \cite{SX}  that the Boussinesq systems \eqref{abcd} are well-posed in a suitable Sobolev setting (with uniform bounds) on time scales of order $1/\epsilon$ in the following cases:
\vspace{0.3cm}
 \begin{itemize}
\item[(1)]  $b>0,d=0, a,c<0$;
\item[(2)]   $b>0 ,d=0, a=0, c<0$;
\item[(3)]  $b=0, d>0, a,c<0$;
\item[(4)] $b\neq d$, $b,d>0, a,c<0$ or $b=0, d>0, a=0, c<0$;
\item[(5)]  $b\neq d, b,d>0, a=0, c<0$;
\item[(6)]  $b=d>0, a,c<0$ or $b>0, d=0, a<0, c=0$;
\item[(7)] $b>0, d=0, a=c=0$ or $b=d>0, a=0, c<0$;
\item[(8)]  $b,d>0, a<0, c=0$
or $b=0, d>0, a=c=0$;
\item[(9)]  $b=0, d>0, a<0, c=0$;
\item[(10)]  $b,d>0, a=c=0$;
\item[(11)] $b=d=0, a,c<0$.
\end{itemize}

Note that the last case can occur only in case of a strong surface tension, as the two following that were not considered in \cite{SX} :
\begin{itemize}
\item[(12)]  $b=d=0, a=0,c<0$;
\item[(13)]  $b=d=0, a<0,c=0$.
\end{itemize}

Actually, the same  scheme  of proof (by symmetrization) is used in \cite{SX} but because of the many different cases to be dealt with (the technical details  cannot be treated in an unified way), we only provided a complete proof in \cite{SX} for  cases (4) ("generic case"), (1)  and (11), which are "strongly dispersive". The other cases can be dealt with by similar symmetrization techniques but the proofs  for some of them need more  explanations that we detail below.

\section{Some strongly dispersive Boussinesq systems}

We will study here the local well-posedness (that is on time scales of order $1/\sqrt \epsilon$) of the Cauchy problem  for two  strongly dispersive  Boussinesq systems having an order two dispersion. They occur only for capillary-gravity waves when the surface tension parameter is greater than $1/3.$ Two (purely gravity waves ) systems having also an order two dispersion corresponding respectively to $a<0, c<0, b=0, d>0$ and $a<0, c<0, d=0, b>0$ have been studied in \cite{LPS} under a curl free condition in the later case. The local well-posedness on time scale of order $1/\sqrt \epsilon$ was proven there while well-posedness on time scales of order $1/\epsilon$ is established in \cite{SX} together with the appropriate uniform bounds. As in \cite {LPS} we will use somehow the dispersive properties of the systems which allows to enlarge the space of resolution but will not provide existence  on  the "long" time scale $1/\epsilon$ which will be considered in Section 4.

Those systems will be referred to as "Schr\"{o}dinger type" since in space dimension two their dispersion relations for large frequencies are reminiscent of the Schr\"{o}dinger one (in one dimension, the analogy is with the Benjamin-Ono equation). This will be made clear when rewriting the systems in an equivalent form after diagonalizing the linear part.

\subsection{A first  Schr\"{o}dinger type system}
We consider the Boussinesq systems when $a<0, b=c=d=0,$ a case which occurs for capillary surface waves with strong enough surface tension, $\tau>1/3$ and which was not considered in \cite{SX}. One can obviously restrict to the case where $a=-1$ and  we consider first the one-dimensional   system
\begin{equation}
    \label{a<0}
    \left\lbrace
    \begin{array}{l}
    \eta_t+u_x+\epsilon(u\eta)_x-\epsilon u_{xxx}=0, \\
    u_t+\eta_x+\epsilon uu_x=0.
    \end{array}\right.
    \end{equation}

Note that this system has the hamiltonian structure
\begin{displaymath}
\partial_t\begin{pmatrix} \eta\\u
\end{pmatrix}+J\text{grad}\;H_\epsilon(\eta,u)=0
\end{displaymath}
where
\begin{displaymath}
J=\begin{pmatrix} 0&\partial_x\\
\partial_x&0
\end{pmatrix}
\end{displaymath}
and
$$H_\epsilon(\eta,u)=\frac{1}{2}\int_\R (\epsilon u_x^2+\eta^2+u^2+\epsilon u^2\eta)dx.$$

Unfortunately the formal conservation of the Hamiltonian cannot be used to get a global $L^2\times H^1$ bound.

As for other order two Boussinesq systems (see \cite{BCS2, LPS}) one can solve the local Cauchy problem for \eqref{a<0} by "elementary" energy methods.

For $U=(\eta, u)^T$ we define

$$||U||_{X^s_\epsilon}=\left(\int_\R (\eta^2+u^2+|D^s_x\eta|^2+|D^s_x u|^2+\epsilon |D^{s+1}_xu|^2)dx\right)^{1/2}.$$

\begin{theorem}\label{a<0elem}
Let $s>1/2$ and  $(\eta_0,u_0)\in X_\epsilon^s.$ There exists $T_\epsilon=O(1/\sqrt \epsilon)$ and a unique solution $(\eta,u)\in   C(\lbrack 0,T_\epsilon\rbrack); X_\epsilon^s)$ of \eqref{a<0} with initial data $(\eta_0,u_0).$

\end{theorem}

\begin{proof}
In order to restrict the technicalities, we consider only the case $s=1$ and we derive only the suitable a priori estimates. The complete proof would use various Kato-Ponce type commutator estimates and an approximation argument. \footnote{We will treat the general situation in the two-dimensional case.}

We take successively the $L^2$ scalar product of the first equation in \eqref{a<0} by $\eta-\eta_{xx}$ and of the second equation by $(I-\epsilon \partial_x^2)(u-u_{xx}).$ After several integrations by parts we obtain by adding the resulting equations :
\begin{equation}\label{e1}
\begin{split}
\frac{1}{2}\frac{d}{dt}\int_\R (\eta^2&+u^2+\eta_x^2+(1+\epsilon)u_x^2+\epsilon u^2_{xx})dx \\
&=-\epsilon\int_\R \lbrack \frac{1}{2}u_x\eta^2+\frac{3}{2}u_x\eta_x^2+\eta \eta_xu_{xx}+\frac{1+\epsilon}{2}u_x^3+\frac{5\epsilon}{2}u_xu_{xx}^2\rbrack dx.
\end{split}
\end{equation}

We now use H\"{o}lder inequality, the standard inequality $|u|_\infty\lesssim |u|_2^{1/2}|u_x|_2^{1/2}$ and that
\beno
|\eta|_\infty,\ |\eta|_2, |\eta_x|_2, |u_x|_2, \sqrt \epsilon|u_x|_\infty, \sqrt \epsilon |u_{xx}|_2 \lesssim  ||U||_{X^1_\epsilon}
\eeno
to obtain from \eqref{e1}  that $||U||_{X_\epsilon^1}\leq C$ on the maximal existence time interval $\lbrack 0,T_\epsilon\rbrack $ of the ODE
\beno
y'\leq C\sqrt \epsilon \;y^2
\eeno
and one readily checks that $T_\epsilon=O(\frac{1}{\sqrt\epsilon}).$

This leads to the existence of a weak solution $U\in L^\infty(0,T_\epsilon;X^1_\epsilon)$  (with an uniform $H^1$ bound).

To prove uniqueness, we set $N=\eta_1-\eta_2,\; V=u_1-u_2$ where $(\eta_1,u_1),\; (\eta_2,u_2)$ are two solutions in $  C(\lbrack 0,T_\epsilon\rbrack); X_\epsilon^s)$. Thus
\begin{equation}
    \label{uniq}
    \left\lbrace
    \begin{array}{l}
    N_t+V_x+\epsilon \lbrack (V\eta_1)_x+(u_2N)_x\rbrack-\epsilon V_{xxx}=0, \\
    V_t+N_x+\epsilon\lbrack  Vu_{1x}+u_2V_x\rbrack=0.
    \end{array}\right.
    \end{equation}

One takes the $L^2$ product scalar of the first equation by $N$ and successively the $L^2$ scalar product of the second equation by $V$ and $-\epsilon V_{xx}.$
Adding the resulting equalities we obtain
\begin{equation}\label{e2}
\begin{split}
&\frac{1}{2}\frac{d}{dt}\int_\R (N^2+V^2+\epsilon V_x^2)dx\\
&=\int_\R\{-\epsilon
\lbrack (V\eta_1)_xN+(u_2N)_xN\rbrack
-\epsilon \lbrack V^2u_{1x}+u_2VV_x\rbrack\\
&\quad+\epsilon^2(VV_{xx}u_{1x}+u_2V_xV_{xx})\}dx ,
\end{split}
\end{equation}
from which we deduce
\begin{equation}\label{e3}
\begin{split}
&\frac{1}{2}\frac{d}{dt}\int_\R (N^2+V^2+\epsilon V_x^2)dx\\
&\leq C\epsilon \lbrack |\eta_1|_\infty|V_x|_2|N|_2+|V|_2^{1/2}|V_x|_2^{1/2}|N|_2
|\eta_{1x}|_2\\
&\quad+|u_{2x}|_\infty|N|^2_2+|u_{1x}|_\infty|V|^2_2+|u_{2x}|_\infty|V|_2^2\\
&\quad+\epsilon((|u_{1x}|_\infty+|u_{2x}|_\infty)|V_x|_2^2)+|V|_2^{1/2}|V_x|_2^{3/2}|u_{1xx}|_2)\rbrack,
\end{split}
\end{equation}
so that
\begin{equation}
\frac{1}{2}\frac{d}{dt}\int_\R (N^2+V^2+\epsilon V_x^2)dx\leq C(|N|^2_2+|V|_2^2+\epsilon |V_x|_2^2)
\end{equation}
and $N=V=0$ by Gronwall's lemma.

\vspace{0.3cm}
It remains to prove the strong continuity in time of the solution with value in $X_\epsilon^1$ and the continuity of the flow, but this results from the Bona-Smith approximation procedure \cite{BS}.
\end{proof}

\vspace{0.5cm}

The local well-posedness of the Cauchy problem \eqref{a<0}  for data of low regularity weighted Sobolev spaces was proven in \cite{KS} by analogy with the DNLS equation.
We indicate now another possible method to obtain the local well-posedness of \eqref{a<0} in a different functional setting by reducing it to a system of Benjamin-Ono type equations.
A natural idea is to transform \eqref{a<0} by diagonalizing the dispersive part. We denote
\begin{displaymath}
\widehat{A}(\xi)=i\xi\begin{pmatrix} 0 & 1+\epsilon |\xi|^2  \\
                 1  & 0  \\

                  \end{pmatrix}.
\end{displaymath}
the Fourier transform of the dispersion matrix with eigenvalues $\pm i\xi(1+\epsilon |\xi|^2)^{1/2}.$
In what follows we will denote $J_\epsilon =(I-\epsilon\partial_x^2)^{1/2}.$

Setting
\begin{displaymath}
U=\begin{pmatrix} \eta\\u
\end{pmatrix}
\end{displaymath}
and
\begin{displaymath}
W=\begin{pmatrix} \zeta \\ v
\end{pmatrix}=P^{-1}U, \quad
P^{-1}=\frac{1}{2}\begin{pmatrix} 1 & J_\epsilon \\
1 & -J_\epsilon
\end{pmatrix},
\end{displaymath}
the linear part of \eqref {a<0} is diagonalized as
$$W_t+\partial_xDW=0,$$
where
\begin{displaymath}
D=\begin{pmatrix} J_\epsilon &0\\
0&-J_\epsilon
\end{pmatrix}.
\end{displaymath}

Since $U=PW$, where
\begin{displaymath}
P=\begin{pmatrix} 1 &1\\
J_\epsilon^{-1}&-J_\epsilon^{-1}
\end{pmatrix}.
\end{displaymath}
one can therefore reduce \eqref{a<0} to the equivalent form
\begin{equation}
    \label{bis}
    \left\lbrace
    \begin{array}{l}
    \zeta_t+J_\epsilon \zeta_x+\frac{\epsilon}{2} N_1(\zeta,v)=0, \\
    v_t-J_\epsilon v_x+\frac{\epsilon}{2} N_2(\zeta,v)=0.
    \end{array}\right.
    \end{equation}
where
$$N_1(\zeta,v)=\partial_x\lbrack (\zeta+v)J_\epsilon^{-1}(\zeta-v)\rbrack+J_\epsilon\lbrack J_\epsilon^{-1}(\zeta-v)J_\epsilon^{-1}(\zeta_x-v_x)\rbrack $$
and
$$ N_2(\zeta,v)= \partial_x\lbrack (\zeta+v)J_\epsilon^{-1}(\zeta-v)\rbrack-J_\epsilon\lbrack J_\epsilon^{-1}(\zeta-v)J_\epsilon^{-1}(\zeta_x-v_x)\rbrack  $$

Since
$$(1+\epsilon \xi^2)^{1/2}-\epsilon^{1/2}|\xi|=\frac{1}{(1+\epsilon\xi^2)^{1/2}+\epsilon^{1/2}|\xi|},$$
\eqref{bis} writes
\begin{equation}
    \label{ter}
    \left\lbrace
    \begin{array}{l}
    \zeta_t+\epsilon^{1/2}\mathcal H\zeta_{xx}+R_\epsilon \zeta+\frac{\epsilon}{2}N_1(\zeta,v)=0, \\
    v_t-\epsilon^{1/2}\mathcal H v_{xx}-R_\epsilon v+\frac{\epsilon}{2}N_2(\zeta,v)=0.
    \end{array}\right.
    \end{equation}
where $R_\epsilon$ is the  (order zero) skew-adjoint  operator with symbol $\frac{i\xi}{(1+\epsilon\xi^2)^{1/2}+\epsilon^{1/2}|\xi|}.$

Note that the nonlinear term are similar but in a sense nicer than the quadratic term $uu_x$ of the Benjamin-Ono equation and thus we should apply for instance the method Ponce \cite{P} used to solve the Cauchy problem for the Benjamin-Ono equation, that is the dispersive estimates on the group $e^{it\partial _x\mathcal H},$ since this method does not used the specific structure of the nonlinear term in the  Benjamin-Ono equation. This would imply local well-posedness for $(\zeta_0, v_0)\in H^s(\R)^2, s\geq 3/2,$ which corresponds to $ (\eta_0,u_0)=(\zeta_0+ v_0, J_\epsilon^{-1}(\zeta_0-v_0))\in H^s(\R)\times H^{s-1}(\R).$ Note the difference  with the functional setting of Theorem \ref{a<0elem}. Similarly, it is likely that the method in \cite{KK} which leads to a local well-posedness theory in $H^s(\R), s>9/8$ for the Benjamin-Ono equation can be applied to \eqref{a<0} leading to a $H^s(\R)\times H^{s-1}(\R), s>9/8$ theory. Also the new method in \cite{MV} could lead to the
  resolution of the Cauchy problem in the energy space $H^{1/2}(\R)$ in the  $(\zeta,v)$ variables. Those methods would however not enlarge the $O(1/\sqrt \epsilon)$ lifespan.\footnote{On the other hand, Tao's method \cite{T} which leads to a $H^1(\R)$ well-posedness theory for the Benjamin-Ono equation uses a gauge transform which strongly relies on the specific structure of the Benjamin-Ono equation and its generalization to \eqref{a<0} is problematic.}












\vspace{0.5cm}
In the two-dimensional case, the system writes
\begin{equation}
    \label{2D}
    \left\lbrace
    \begin{array}{l}
    \eta_t+\nabla\cdot u+\epsilon\nabla\cdot (\eta {\vv u})-\epsilon \nabla\cdot\Delta {\vv u}=0, \\
    {\vv u}_t+\nabla\eta+\frac{\epsilon}{2} \nabla|{\vv u}|^2=0.
    \end{array}\right.
    \end{equation}

\vspace{0.3cm}


\vspace{0.3cm}
This system has also the Hamiltonian structure
\begin{displaymath}
\partial_t\begin{pmatrix} \eta\\{\vv u}
\end{pmatrix}+J\text{grad}\;H_\epsilon(\eta,{\vv u})=0
\end{displaymath}
where
\begin{displaymath}
J=\begin{pmatrix} 0 & \partial_{x} & \partial_{y} \\
                 \partial_{x} & 0 & 0 \\
                 \partial_{y} & 0 & 0 \end{pmatrix}.
\end{displaymath}
and
$$H(\eta,{\vv u})=\frac{1}{2}\int_{\R^2} (\epsilon|\nabla{\vv u}|^2+\eta^2+|{\vv u}|^2+\epsilon|{\vv u}|^2\eta)dxdy.$$

\vspace{0.3cm}

Under a curl free assumption on ${\vv u}$ (which is natural since the Boussinesq systems are derived for potential flows), one can obtain the local well-posedness of \eqref {2D} by "elementary" methods. Actually, when $\text{curl}\; {\vv u}=0,$ \eqref{2D} becomes
\begin{equation}\label{d1}\left\{\begin{aligned}
&\partial_t\eta+\nabla\cdot {\vv u}+\epsilon\nabla\cdot(\eta {\vv u})-\epsilon\nabla\cdot\Delta {\vv u}=0,\\
&\partial_t{\vv u}+\nabla\eta+\epsilon {\vv u}\cdot\nabla {\vv u}=0.
\end{aligned}\right.
\end{equation}

Before going further, we present the following commutator estimates (see Theorems 3 and 6 in \cite{Lannes1}).
\begin{lemma}\label{L1}
Let $t_0>\frac{n}{2}$, $-t_0<r\leq t_0+1$. Then for all $s\geq 0$, $f\in H^{t_0+1}\cap H^{s+r}(\R^n)$ and $u\in H^{s+r-1}(\R^n)$, there holds:
\begin{equation}\label{b10}
|[\Lambda^s, f]u|_{H^r}\leq C(|\nabla f|_{H^{t_0}}|u|_{H^{s+r-1}}+\langle|\nabla f|_{H^{s+r-1}}|u|_{H^{t_0}}\rangle_{s>t_0+1-r}),
\end{equation}
where $a+\langle b\rangle_{s>s_0}$ equals $a$ if $s\leq s_0$ while equals $a+b$ if $s>s_0$.
\end{lemma}

Consequently, taking $t_0=s>1$ and $r=0$ in \eqref{b10}, we have the following corollary.
\begin{corollary}\label{dL1}
For $s>1$, $f\in H^{s+1}(\R^2)$, $g\in H^{s-1}(\R^2)$, then
\begin{equation}\label{d3}
|[\Lambda^s, f]g|_2\lesssim |\nabla f|_{H^s}|g|_{H^{s-1}}.
\end{equation}
\end{corollary}


\vspace{0.3cm}

Going back to \eqref{d1},
similarly to the one-dimensional case, we denote by $U=(\eta,{\vv u})^T,$ and then define
\begin{equation}
\|U\|_{X_{\epsilon}^s}=(|\Lambda^s\eta|_2^2+|\Lambda^s{\vv u}|_2^2+\epsilon|\Lambda^s\nabla {\vv u}|_2^2)^{\frac{1}{2}},
\end{equation}
and we obtain the following theorem.

\begin{theorem}\label{a<0-2D}
Let $s>1$ and $(\eta_0,{\vv u}_0)\in X^s_\epsilon. $ Then there exists $T_\epsilon=O(1/\sqrt \epsilon)$ and a unique solution $(\eta,{\vv u})\in C(\lbrack 0,T_\epsilon\rbrack; X^s_\epsilon)$ of \eqref{d1} with initial data $(\eta_0,{\vv u}_0).$
Moreover,
$$\sup_{t\in \lbrack 0,T_\epsilon\rbrack}\|(\eta(\cdot,t),{\vv u}(\cdot,t))\|_{X^s_\epsilon}< c \|(\eta_0,\vv u_0))\|_{X^s_\epsilon}.$$
\end{theorem}

\begin{proof}

As in the one-dimensional case we will only provide the suitable a priori estimate.

Taking the $L^2$ inner product of the first equation in \eqref{d1} by $\Lambda^{2s}\eta$ and of the second equation by $(1-\epsilon\Delta)\Lambda^{2s}{\vv u}$, and then integrating by parts, it results
\begin{equation}\label{d2}\begin{split}
&\frac{1}{2}\frac{d}{dt}\bigl(|\Lambda^s\eta|_2^2+|\Lambda^s{\vv u}|_2^2+\epsilon|\Lambda^s\nabla {\vv u}|_2^2\bigr)\\
&=-\epsilon(\Lambda^s\nabla\cdot(\eta {\vv u})\,|\,\Lambda^s\eta)_2
-\epsilon(\Lambda^s({\vv u}\cdot\nabla {\vv u})\,|\,(1-\epsilon\Delta)\Lambda^s{\vv u})_2,
\end{split}\end{equation}

Now, we deal with the r.h.s terms in \eqref{d2}. We first get that
\begin{equation}\begin{aligned}
&(\Lambda^s\nabla\cdot(\eta {\vv u})\,|\,\Lambda^s\eta)_2=(\Lambda^s({\vv u}\cdot\nabla\eta)\,|\,\Lambda^s\eta)_2+(\Lambda^s(\eta\nabla\cdot {\vv u})\,|\,\Lambda^s\eta)_2\\
&=([\Lambda^s,{\vv u}]\cdot\nabla\eta)\,|\,\Lambda^s\eta)_2-\frac{1}{2}(\nabla\cdot {\vv u}\Lambda^s\eta\,|\,\Lambda^s\eta)_2+(\Lambda^s(\eta\nabla\cdot {\vv u})\,|\,\Lambda^s\eta)_2,
\end{aligned}\end{equation}
which together with \eqref{d3} implies that
\begin{equation}\label{d7}\begin{aligned}
&|(\Lambda^s\nabla\cdot(\eta {\vv u})\,|\,\Lambda^s\eta)_2|\\
&\lesssim|[\Lambda^s,{\vv u}]\cdot\nabla\eta)|_2|\Lambda^s\eta|_2+|\nabla\cdot {\vv u}|_{\infty}|\Lambda^s\eta|_2^2
+|\Lambda^s(\eta\nabla\cdot{\vv u})|_2|\Lambda^s\eta|_2\\
&\lesssim|\nabla {\vv u}|_{H^s}|\nabla\eta|_{H^{s-1}}|\eta|_{H^s}+|\nabla {\vv u}|_{H^s}|\eta|_{H^s}^2\lesssim|\nabla {\vv u}|_{H^s}|\eta|_{H^s}^2.
\end{aligned}\end{equation}

For the second term on the r.h.s of \eqref{d2}, we have
\begin{equation}\begin{aligned}
&\ \ (\Lambda^s({\vv u}\cdot\nabla {\vv u})\,|\,(1-\epsilon\Delta)\Lambda^s{\vv u})_2=([\Lambda^s,{\vv u}]\cdot\nabla {\vv u}\,|\,\Lambda^s{\vv u})_2\\
&\qquad-\frac{1}{2}(\nabla\cdot {\vv u}\Lambda^s{\vv u}\,|\,\Lambda^s{\vv u})_2+\epsilon\sum_{i=1}^2(\Lambda^s\nabla(u_i\partial_i\vv u)\,|\,\Lambda^s\nabla\vv u)_2\\
&=([\Lambda^s,{\vv u}]\cdot\nabla {\vv u}\,|\,\Lambda^s{\vv u})_2-\frac{1}{2}(\nabla\cdot{\vv u}\Lambda^s{\vv u}\,|\,\Lambda^s{\vv u})_2\\
&\quad+\epsilon\sum_{i=1}^2([\Lambda^s,u_i]\nabla\partial_i{\vv u}\,|\,\Lambda^s\nabla {\vv u})_2-\frac{1}{2}\epsilon(\nabla\cdot{\vv u}\Lambda^s\nabla{\vv u}\,|\,\Lambda^s\nabla{\vv u})_2\\
&\quad+\epsilon\sum_{i=1}^2(\Lambda^s(\nabla u_i\partial_i{\vv u})\,|\,\Lambda^s\nabla{\vv u})_2
\end{aligned}\end{equation}
which along with \eqref{d3} implies that
\begin{equation}\label{d8}
|(\Lambda^s({\vv u}\cdot\nabla {\vv u})\,|\,(1-\epsilon\Delta)\Lambda^s{\vv u})_2|\lesssim|\nabla{\vv u}|_{H^s}\bigl(|{\vv u}|_{H^s}^2+\epsilon|\nabla {\vv u}|_{H^s}^2\bigr).
\end{equation}

Denoting again  $U=(\eta,\vv u)^T,$
we deduce from \eqref{d2}, \eqref{d7} and \eqref{d8} that
\begin{equation}\label{d9}
\frac{d}{dt}\|U(t)\|_{X_{\epsilon}^s}\leq C\sqrt\epsilon\|U(t)\|_{X_{\epsilon}^s}^2,
\end{equation}
from which, we infer that the maximal existence time interval is $[0,T_\epsilon]$ with $T_\epsilon=O(1/{\sqrt\epsilon})$.

As in the one-dimensional case one justifies the a priori estimates by a suitable approximation of the system (for instance by adding $(-\delta \Delta \eta_t, -\delta \Delta{\vv u}_t)^T,\; \delta>0$). Uniqueness is obtained again by a Gronwall type argument and the strong continuity in time and the continuity of the flow map result from the Bona-Smith trick.
\end{proof}

\vspace{0.3cm}
As in the one-dimensional case, one has a better insight on the system by diagonalizing the linear part. The dispersion  matrix is in Fourier variables
 \begin{displaymath}
\widehat{A}(\xi_1,\xi_2)=i\begin{pmatrix} 0 & \xi_1(1+\epsilon |\xi|^2) & \xi_2(1+\epsilon |\xi|^2) \\
                 \xi_1  & 0 & 0 \\
                 \\
                 \xi_2    & 0 & 0 \end{pmatrix}.
\end{displaymath}

\vspace{0.3cm}
The corresponding   eigenvalues are zero and
$$ \lambda_{\pm}=\pm i |\xi|(1+\epsilon|\xi|^2)^{1/2}$$
with corresponding eigenvectors
\begin{displaymath}
E_0=\begin{pmatrix} 0 \\ -\frac{\xi_2}{|\xi|}(1+\epsilon|\xi|^2)^{-1/2} \\ \frac{\xi_1}{|\xi|}(1+\epsilon|\xi|^2)^{-1/2}
\end{pmatrix}, \quad E_1=\begin{pmatrix} 1 \\ \frac{\xi_1}{|\xi|}(1+\epsilon|\xi|^2)^{-1/2} \\ \frac{\xi_2}{|\xi|}(1+\epsilon|\xi|^2)^{-1/2}
\end{pmatrix},
\end{displaymath}
\begin{displaymath}
\text{and}
 \quad E_2=\begin{pmatrix} -1 \\
\frac{\xi_1}{|\xi|}(1+\epsilon|\xi|^2)^{-1/2} \\ \frac{\xi_2}{|\xi|}(1+\epsilon|\xi|^2)^{-1/2}
\end{pmatrix}.
\end{displaymath}

Now we set $J_\epsilon=(I-\epsilon\Delta)^{1/2}$ and $R_1,R_2$ the Fourier multiplier operator with respective symbols $i\xi_1/|\xi|, i\xi_2/|\xi|.$

We also denote
\begin{displaymath}
P=\begin{pmatrix}0 &i &-i\\
-R_2J_\epsilon^{-1} & R_1J_\epsilon^{-1} & R_1J_\epsilon^{-1}\\
R_1J_\epsilon^{-1} & R_2J_\epsilon^{-1} & R_2J_\epsilon^{-1} \end{pmatrix}
\end{displaymath}
and
\begin{displaymath}
P^{-1}=\frac{1}{2} \begin{pmatrix} 0 &2 R_2J_\epsilon & -2 R_1J_\epsilon\\
-i& -R_1J_\epsilon  &   -R_2J_\epsilon \\
i& -R_1J_\epsilon  &   -R_2J_\epsilon
\end{pmatrix}.
\end{displaymath}

Setting  $U=(\eta,{\vv u})^{T}$ and $V=(\zeta,{\bf v})^{T}=P^{-1} U,$
 \eqref{2D} writes as
$$U_t+A U+\epsilon N(U)=0,$$
which is transformed after diagonalizing the linear part,
$$V_t+D V+\epsilon P^{-1}N(PV)=0,$$
or, setting $P^{-1}N(PV)=\tilde{N}(V),$
\begin{equation}\label{NLSlike}
V_t+D V+\epsilon\tilde{N}(V)=0,
\end{equation}
where
\begin{displaymath}
D=\begin{pmatrix} 0 &0&0\\
0&i(-\Delta)^{1/2}J_\epsilon &0\\
0&0&-i(-\Delta)^{1/2}J_\epsilon
\end{pmatrix}
\end{displaymath}

We turn now to the nonlinear part. $N$ is given as a function of $U$ by
\begin{displaymath}
N(U)=\begin{pmatrix} \nabla\cdot(\eta{\vv u})\\
\frac{1}{2}\partial_x(|{\vv u}|^2)\\
\frac{1}{2}\partial_y(|{\vv u}|^2)
\end{pmatrix}=\begin{pmatrix} \nabla\cdot(\eta{\vv u})\\
{\vv u}\cdot\nabla u_1\\
{\vv u}\cdot\nabla u_2
\end{pmatrix},
\end{displaymath}
where we used the condition $\curl{\vv u}=0$ in the second equality.

On the other hand, $P^{-1}N(U)$ is given by
\begin{displaymath}
-\frac{1}{2}\begin{pmatrix} 0\\
i\nabla\cdot(\eta{\vv u})+R_1J_\epsilon({\vv u}\cdot\nabla u_1)+R_2J_\epsilon({\vv u}\cdot\nabla u_2)\\
-i\nabla\cdot(\eta{\vv u})+R_1J_\epsilon({\vv u}\cdot\nabla u_1)+R_2J_\epsilon({\vv u}\cdot\nabla u_2)
\end{pmatrix}.
\end{displaymath}
To obtain the expression of $\tilde{N}(V)$, we should express $(\eta, u_1, u_2)$ as
$$\eta=i(v_1-v_2),\quad  u_1=-R_2J_\epsilon^{-1}\zeta+R_1J_\epsilon^{-1}(v_1+v_2),$$
$$u_2=R_1J_\epsilon^{-1}\zeta+R_2J_\epsilon^{-1}(v_1+v_2),$$
and the nonlinearity is of the same type as in the one-dimensional case.

\begin{remark}
As in the case of the "KdV-KdV" system ($a=c=1/6, b=d=0$) studied in \cite{LPS}, it follows from our analysis that $\zeta=0$ if $\zeta_0$ is
smooth enough, since $\partial_t\zeta=0$. Moreover,
\begin{displaymath}
\zeta=0 \quad \iff \quad R_2u_1=R_1u_2 \quad \iff \text{curl}\; {\vv u}=0.
\end{displaymath}
We observe that this condition makes sense, since our system is
derived from the water waves equations in the irrotational case. Note that ${\vv u}$ is the horizontal velocity at a certain height and it differs from the horizontal velocity at the free surface by an $O(\epsilon^2)$ term. Also, since the equation for ${\vv u}$ writes $\partial_t {\vv u}= \nabla F$, the condition
$\text{curl}\; {\vv u} =0$ is preserved by the evolution.
\end{remark}

\begin{remark}
The linear part in \eqref{NLSlike} is  "Schr\"{o}dinger like"  for large frequencies (the symbol behaves as $\pm i\epsilon^{1/2}|\xi|^2$ as $|\xi|\to +\infty),$ and "wave like" for small frequencies (the symbol behaves as $\pm i|\xi|$ when $|\xi|\to 0).$

The quadratic terms however involves order one operators and this one could a priori think of applying  the results on the Cauchy problem for quasilinear Schr\"{o}dinger type equations (see for instance \cite{KPV}). Those methods however necessitate a high regularity on the data and it is unlikely that they could improve our local result.
\end{remark}

\subsection{A second  Schr\"{o}dinger type system}

We consider here the situation where $a=b=d=0, c<0, $ say $c=-1$ which again may occur only in the case of strong surface tension. It turns out that this system leads to serious difficulties, which are not present in the other Boussinesq systems,  and that we will describe below, even to obtain  the local well-posedness by "elementary" or more sophisticated methods using dispersion. We refer to Section 4 and 5 for
the long time existence issues.

We  consider first the one-dimensional   system
\begin{equation}
    \label{c<0}
    \left\lbrace
    \begin{array}{l}
    \eta_t+u_x+\epsilon(u\eta)_x=0, \\
    u_t+\eta_x+\epsilon uu_x-\epsilon\eta_{xxx}=0.
    \end{array}\right.
    \end{equation}

 The hamiltonian structure is now
\begin{displaymath}
\partial_t\begin{pmatrix} \eta\\u
\end{pmatrix}+J\text{grad}\;H_\epsilon(\eta,u)=0
\end{displaymath}
where
\begin{displaymath}
J=\begin{pmatrix} 0&\partial_x\\
\partial_x&0
\end{pmatrix}
\end{displaymath}
and
$$H_\epsilon(\eta,u)=\frac{1}{2}\int_\R (\epsilon| \eta_x|^2+\eta^2+u^2+\epsilon u^2\eta)dx.$$

Similarly to the case $b=d>0, a=0, c=-1$ considered in \cite{BCS2} and for a related system in \cite{BS2} , one can use the formal conservation of $H_\epsilon$ to derive a
global a priori estimate when $H_\epsilon (\eta_0,u_0)$ is small enough and $\inf_{x\in \R}(1+\epsilon \eta_0(x))>0$. First we derive as in \cite{BCS2,BS} a   $L^\infty$ bound on $\eta $ for solutions $(\eta,u)\in C(\lbrack 0,T\rbrack; H^1(\R))\times C(\lbrack 0,T\rbrack ;L^2(\R)) $ satisfying the above non-cavitation condition. Actually, one writes
\begin{equation}\label{eta}
\begin{split}
\eta^2(x,t)&\leq \int_\R |\eta ||\eta_x|dx=\frac{1}{\sqrt \epsilon}\int_\R \sqrt \epsilon |\eta||\eta_x|dx\leq\frac{1}{2\sqrt \epsilon}\int_R(\eta^2+\epsilon|\eta_x|^2)dx\\
&\leq\frac{1}{\sqrt \epsilon}|H_\epsilon(\eta,u)|=\frac{1}{\sqrt \epsilon}|H_\epsilon(\eta_0,u_0)| \quad t\in \lbrack 0, T\rbrack.
\end{split}
\end{equation}

Using \eqref{eta} and the conservation of $H_\epsilon$ imply a (formal )  $H^1\times L^2$ bound on $(\eta,u)$ provided $H_\epsilon(\eta_0, u_0)$ is small enough, that is
\begin{equation}\label{cond}
H_\epsilon(\eta_0,u_0)<\epsilon^{-3/2}.
\end{equation}

Unfortunately, and contrary to the case $b=d>0, a=0, c<0$ (see \cite{BCS2}), one cannot use the above bounds to get a global well-posedness result, say by a compactness method applied to a regularization of the system. The obstruction is that one cannot pass to the limit of the term $\frac{1}{2}\partial _x(u^2)$ in the second equation only from a $L^2$ bound on $u.$

\vspace{0.3cm}
To obtain an equivalent "diagonal" system, we proceed as in the other "Schr\"{o}dinger type system", setting now
\begin{displaymath}
\widehat{A}(\xi)=i\xi\begin{pmatrix} 0 & 1  \\
                 1+\epsilon|\xi|^2  & 0  \\

                  \end{pmatrix}.
\end{displaymath}
with the same  eigenvalues $\pm i\xi(1+\epsilon |\xi|^2)^{1/2}.$

In the Notation of the previous section, one has still
\begin{displaymath}
D=\begin{pmatrix} J_\epsilon &0\\
0&-J_\epsilon
\end{pmatrix}.
\end{displaymath}
and now
\begin{displaymath}
P=\begin{pmatrix} J_\epsilon^{-1} &-J_\epsilon^{-1}\\
1&1
\end{pmatrix},
\end{displaymath}
\begin{displaymath}
P^{-1}=\frac{1}{2}\begin{pmatrix} J_\epsilon & 1 \\
-J_\epsilon & 1
\end{pmatrix},
\end{displaymath}

Setting again
\begin{displaymath}
U=\begin{pmatrix} \eta\\u
\end{pmatrix}
\end{displaymath}
and
\begin{displaymath}
W=\begin{pmatrix} \zeta \\ v
\end{pmatrix}=P^{-1}U,
\end{displaymath}
one can therefore reduce \eqref{c<0} to the equivalent form
\begin{equation}
    \label{bis2}
    \left\lbrace
    \begin{array}{l}
    \zeta_t+J_\epsilon \zeta_x+\frac{\epsilon}{2} N_1(\zeta,v)=0, \\
    v_t-J_\epsilon v_x+\frac{\epsilon}{2} N_2(\zeta,v)=0.
    \end{array}\right.
    \end{equation}
where
$$N_1(\zeta,v)=\partial_xJ_\epsilon\lbrack (\zeta+v)J_\epsilon^{-1}(\zeta-v)\rbrack+(\zeta+v)(\zeta+v)_x $$
and
$$ N_2(\zeta,v)=- \partial_xJ_\epsilon\lbrack (\zeta+v)J_\epsilon^{-1}(\zeta-v)\rbrack+(\zeta+v)(\zeta+v)_x .$$

We can also write  \eqref{bis2} as
\begin{equation}\label{hard}
     \left\lbrace
    \begin{array}{l}
    \zeta_t+\epsilon^{1/2}\mathcal H\zeta_{xx}+R_\epsilon \zeta+\frac{\epsilon}{2}N_1(\zeta,v)=0, \\
    v_t-\epsilon^{1/2}\mathcal H v_{xx}-R_\epsilon v+\frac{\epsilon}{2}N_2(\zeta,v)=0.
    \end{array}\right.
    \end{equation}
where again $R_\epsilon$ is the  order zero skew-adjoint  operator with symbol $\frac{i\xi}{(1+\epsilon\xi^2)^{1/2}+\epsilon^{1/2}|\xi|}.$

 \vspace{0.3cm}
 Note that the nonlinearity is worse than in the case $a=-1$ and even the local theory does not seem to be straightforward using this formulation.

 \vspace{0.3cm}
We turn now to the two-dimensional case that is
\begin{equation}
    \label{2Dbis}
    \left\lbrace
    \begin{array}{l}
    \eta_t+\nabla\cdot{\vv u}+\epsilon\nabla\cdot (\eta {\vv u})=0, \\
    {\vv u}_t+\nabla\eta+\frac{\epsilon}{2} \nabla|{\vv u}|^2-\epsilon \nabla \Delta \eta=0.
    \end{array}\right.
    \end{equation}

The Hamiltonian structure is now
\begin{displaymath}
\partial_t\begin{pmatrix} \eta\\{\vv u}
\end{pmatrix}+J\text{grad}\;H_\epsilon(\eta,{\vv u})=0
\end{displaymath}
where
\begin{displaymath}
J=\begin{pmatrix} 0 & \partial_{x} & \partial_{y} \\
                 \partial_{x} & 0 & 0 \\
                 \partial_{y} & 0 & 0 \end{pmatrix}.
\end{displaymath}
and
$$H(\eta,{\vv u})=\frac{1}{2}\int_{\R^2} (\epsilon|\nabla \eta|^2+\eta^2+|{\vv u}|^2+\epsilon|{\vv u}|^2\eta)dxdy.$$

\vspace{0.3cm}


As in the one-dimensional case,
one can express \eqref{c<0}  on the equivalent form
\begin{equation}\label{NLSlike2}
V_t+D V+\tilde{N}(V)=0,
\end{equation}
where again $U=(\eta,{\vv u})^T$, $V=(\zeta,{\bf v})^T$,
\begin{displaymath}
D=\begin{pmatrix} 0 &0&0\\
0&i(-\Delta)^{1/2}J_\epsilon &0\\
0&0&-i(-\Delta)^{1/2}J_\epsilon
\end{pmatrix}
\end{displaymath}
and
  $\tilde N(V)$ is expressed as
\begin{displaymath}
\frac{1}{2}\begin{pmatrix} 0\\
-iJ_\epsilon\lbrack \partial_x(u_1\eta)+\partial_y(u_2\eta)\rbrack -\frac{1}{2} \lbrack R_1\partial_x |{\vv u}|^2+R_2\partial_y|{\vv u}|^2\rbrack\\
iJ_\epsilon\lbrack \partial_x(u_1\eta)+\partial_y(u_2\eta)\rbrack -\frac{1}{2} \lbrack R_1\partial_x |{\vv u}|^2+R_2\partial_y|{\vv u}|^2\rbrack\
\end{pmatrix}
\end{displaymath}
with
$$\eta=i J_\epsilon^{-1}(v_1-v_2),\quad u_1=-R_2\zeta+R_1(v_1+v_2),\quad u_2=R_1\zeta+R_2(v_1+v_2).$$

\subsection{Comparison between the two  Schr\"{o}dinger type systems}
The previous considerations display the difficulties of the Cauchy problem in the case $a=b=d=0, c<0.$ We indicate here how to reduce it to the case
$b=d=c=0, a<0$ modulo $O(\epsilon^2)$ terms.

Let us consider for instance the one-dimensional case
\begin{equation}\label{1dc<0}\left\{\begin{aligned}
&\eta_t+u_x+\epsilon(\eta u)_x=0,\\
&u_t+\eta_x+\epsilon uu_x-\epsilon \eta_{xxx}= 0.
\end{aligned}\right.\end{equation}

Setting
$$\tilde{\eta}=(1-\epsilon\partial_x^2)\eta=J_\epsilon^2\eta,$$
\eqref{1dc<0} can be rewritten as follows :
\begin{equation}\label{1dc<0bis}\left\{\begin{aligned}
&\tilde \eta_t+(1-\epsilon\partial_x^2)u_x+\epsilon(\tilde\eta u)_x=\epsilon^2(2\eta_{xx}u_x+3\eta_xu_{xx}+\eta u_{xxx}),\\
&u_t+\tilde\eta_x+\epsilon uu_x= 0,
\end{aligned}\right.\end{equation}
that is
\begin{equation}\label{1dc<0ter}\left\{\begin{aligned}
&\tilde \eta_t+(1-\epsilon\partial_x^2)u_x+\epsilon(\tilde\eta u)_x\\
&\quad=\epsilon^2(2(J_\epsilon^{-2}\tilde\eta_{xx})u_x+3(J_\epsilon^{-2}\tilde\eta_{x})u_{xx}+
(J_\epsilon^{-2}\tilde \eta)u_{xxx}),\\
&u_t+\tilde\eta+\epsilon uu_x= 0,
\end{aligned}\right.\end{equation}

Discarding the $O(\epsilon^2)$ terms, \eqref{1dc<0bis} reduces to
\begin{equation}\label{1dnew}\left\{\begin{aligned}
&\tilde\eta_t+(1-\epsilon\partial_x^2)u_x+\epsilon(\tilde\eta u)_x=0,\\
&u_t+\epsilon\tilde\eta_x+\epsilon u u_x= 0.
\end{aligned}\right.\end{equation}
which is exactly the case $b=c=d=0, a=-1.$

\medskip

Similarly, we can consider the two-dimensional case
\begin{equation}\label{2dc<0}\left\{\begin{aligned}
&\eta_t+\nabla\cdot {\vv u}+\epsilon\nabla\cdot(\eta {\vv u})=0,\\
&{\vv u}_t+\nabla\eta+\frac{\epsilon}{2}\nabla(|{\vv u}|^2)-\epsilon\nabla\Delta\eta={\bf 0}.
\end{aligned}\right.\end{equation}

Setting
$$\tilde{\eta}=(1-\epsilon\Delta)\eta= J_\epsilon^2\eta,$$
\eqref{2dc<0} can be rewritten as follows :
\begin{equation}\label{2dc<0bis}\left\{\begin{aligned}
&\tilde\eta_t+(1-\epsilon\Delta)\nabla\cdot{\vv u}+\epsilon\nabla\cdot(\tilde\eta{\vv u})=\epsilon^2\bigl(\Delta\nabla\cdot(\eta{\vv u})-\nabla\cdot(\Delta\eta{\vv u})\bigr),\\
&{\vv u}_t+\nabla\tilde\eta+\frac{\epsilon}{2}\nabla(|{\vv u}|^2)= 0,
\end{aligned}\right.\end{equation}
that is
\begin{equation}\label{2dc<0ter}\left\{\begin{aligned}
&\tilde\eta_t+(1-\epsilon\Delta)\nabla\cdot{\vv u}+\epsilon\nabla\cdot(\tilde\eta{\vv u})=\epsilon^2\bigl(\Delta\nabla\cdot({\vv u}J_\epsilon^{-2}\tilde\eta)-\nabla\cdot({\vv u}\Delta J_\epsilon^{-2}\tilde\eta\bigr),\\
&{\vv u}_t+\nabla\tilde\eta+\frac{\epsilon}{2}\nabla(|{\vv u}|^2)= 0,
\end{aligned}\right.\end{equation}

Discarding the $O(\epsilon^2)$ terms, \eqref{2dc<0bis} reduces to
\begin{equation}\label{2dnew}\left\{\begin{aligned}
&\tilde\eta_t+(1-\epsilon\Delta)\nabla\cdot{\vv u}+\epsilon\nabla\cdot(\tilde\eta{\vv u})=0,\\
&{\vv u}_t+\nabla\tilde\eta+\frac{\epsilon}{2}\nabla(|{\vv u}|^2)={\bf 0}
\end{aligned}\right.\end{equation}
which is exactly the case $b=c=d=0, a=-1.$

\vspace{0.3cm}
The bad structure of the nonlinear terms in \eqref{1dc<0ter}, \eqref{2dc<0ter} (or \eqref{hard}, \eqref{NLSlike2}) explain why solving the Cauchy problem for systems \eqref{c<0} or \eqref{2Dbis} is so difficult, despite their apparent simplicity. One could notice that we always lose one derivative for $\eta$ or $\vv u$. Thus, to solve the case $a=b=d=0,c<0$, we turn to quasilinearize the system by applying $\partial_t^k$ instead of the usual $\partial_x^{\alpha}$. We shall discuss details in the following section.

\section{Long time existence for some Boussinesq systems}

We first give a complete proof for some systems considered in \cite{SX} (in particular the two-dimensional version of the system considered in \cite{A,Sc}) and apply the same symmetrization techniques to study the long time existence of solutions (in a smaller Sobolev space) of one of the "Schr\"{o}dinger type systems" described in the previous section. We then consider the more delicate case $a=b=d=0, c<0.$

\vspace{0.3cm}
We associate to \eqref{abcd} the initial data
\begin{equation}\label{rme2}
\eta|_{t=0}=\eta_0,\quad {\vv u}|_{t=0}= {\vv u}_0.
\end{equation}

\subsection{The case  $a=c=d=0, b>0$ with condition $ \curl\; {\vv u}=0$ }

Before going further, we state some technical lemmas and definitions.

\begin{definition}\label{def1}
For any $s\in\R$, $k\in\N$, $\epsilon\in(0, 1)$, the Banach space $X^s_{\epsilon^k}(\R^n)$
is defined as
$H^{s+k}(\R^n)$ equipped with the norm:
\beno
|u|_{X^s_{\epsilon^k}}^2=|u|_{H^s}^2+\epsilon^k|u|_{H^{s+k}}^2
\eeno
\end{definition}

\begin{lemma}\label{lem1}
For any $i,k\in\N$ and $0<i<k$,
there holds the following interpolation inequality:
\begin{equation}\label{int1}
\epsilon^{\frac{i}{2}}|f|_{H^{s+i}}\lesssim|f|_{H^s}^{1-\frac{i}{k}}\bigl(\epsilon^{\frac{k}{2}}|f|_{H^{s+k}}\bigr)^{\frac{i}{k}}\lesssim|f|_{X^s_{\epsilon^k}}.
\end{equation}
\end{lemma}

\

\begin{theorem}\label{rmT1}
Let $b>0, a=c=d=0$. $n=1,2$, $s>1+\frac{n}{2}$ .
Assume that $\eta_0\in X^s_{\epsilon^2}(\R^n),{\vv u}_0\in X^s_{\epsilon}(\R^n)$ with $\curl{\vv u}_0=0$ when $n=2$, satisfy the (non-cavitation) condition
\begin{equation}\label{rme5}
1+\epsilon\eta_0\geq H>0,\quad H\in(0,1),
\end{equation}
Then there exists a constant $\tilde{c}_0$  such that for any $\epsilon\leq\epsilon_0=\frac{1-H}{\tilde{c}_0(| \eta_0|_{X^s_{\epsilon^2}}+|{\bf  u}_0|_{X^s_{\epsilon}})}$, there
exists $T>0$ independent of $\epsilon$, such that \eqref{abcd}-\eqref{rme2} has a unique solution $(\eta, {\vv u})^T$ with
$\eta\in C([0,T/\epsilon];X^s_{\epsilon^2}(\R^n))$ and $ {\vv u}\in C([0,T/\epsilon];X^s_{\epsilon}(\R^n))$. Moreover,
\begin{equation}\label{rme6}
\max_{t\in[0,T/\epsilon]} (|\eta|_{X^s_{\epsilon^2}}+| {\vv u}|_{X^s_{\epsilon}})\leq \tilde c (|\eta_0|_{X^s_{\epsilon^2}}+| {\vv u}_0|_{X^s_{\epsilon}}).
\end{equation}
Here  $\tilde c=C(H^{-1})$ and $\tilde c_0=C(H^{-1})$ are nondecreasing functions of their argument.
And in what follows, without confusion, we denote $\tilde c=C(H^{-1})$ a
nondecreasing constant depending on $H^{-1}$.
 Otherwise, we denote $\tilde{c}_i$ (i=0,1,2,...)  constants having the same
properties as $\tilde{c}$.
\end{theorem}

\begin{proof} The proof follows the same method used in \cite{SX}, that is to obtain energy estimates on a suitable  symmetrized  linearized system followed by an iterative scheme. Here we only give the {\it a priori estimates} on the full nonlinear system and in the two-dimensional case.
Since $c=d=0$ and $\curl {\vv u}_0=0$, we deduce from the second equation of \eqref{abcd} that
\begin{equation}\label{rme4}
\curl {\vv u}=0,\quad\text{for}\quad t>0.
\end{equation}
Then using \eqref{rme4},  \eqref{abcd} becomes
\begin{equation}\label{rme3}\left\{\begin{aligned}
&\partial_t\eta+\nabla\cdot {\vv u}+\epsilon\nabla\cdot(\eta {\vv u})-b\epsilon\Delta\partial_t\eta=0,\\
&\partial_t {\vv u}+\nabla\eta+\epsilon {\vv u}\cdot\nabla {\vv u}= {\bf 0}.
\end{aligned}\right.\end{equation}
Denoting by $U=(\eta, {\vv u})$, \eqref{rme3} is rewritten in the condensed form as
\begin{equation}\label{rme7}
(1-b\epsilon \Delta)\partial_t U+M(U,D)U= {\bf 0},
\end{equation}
where
\beno
M(U,D)=
\begin{pmatrix}
\epsilon {\vv u}\cdot\nabla&(1+\epsilon\eta)\partial_1&(1+\epsilon\eta)\partial_2\\
(1-b\epsilon\Delta)\partial_1&(1-b\epsilon\Delta)(\epsilon {\vv u}\cdot\nabla)&0\\
(1-b\epsilon\Delta)\partial_2&0&(1-b\epsilon\Delta)(\epsilon {\vv u}\cdot\nabla)
\end{pmatrix}.\eeno
The symmetrizer of $M(U,D)$ is
\beno
S_U(D)=\begin{pmatrix}
1-b\epsilon\Delta&0&0\\
0&1+\epsilon\eta&0\\
0&0&1+\epsilon\eta
\end{pmatrix}.
\eeno

We define the energy functional associated to \eqref{rme7} as
\begin{equation}\label{rme8}\begin{aligned}
E_s(U)&=((1-b\epsilon\Delta)\Lambda^sU\,|\,S_U(D)\Lambda^sU)_2\\
&=((1-b\epsilon\Delta)\Lambda^s\eta\,|\,(1-b\epsilon\Delta)\Lambda^s\eta)_2+((1-b\epsilon\Delta)\Lambda^s {\vv u}\,|\,(1+\epsilon\eta)\Lambda^s {\vv u})_2.
\end{aligned}\end{equation}

Assume that
\begin{equation}\label{rme9}
1+\epsilon\eta\geq H>0,\quad \epsilon|\eta|_{W^{1,\infty}}\leq \kappa_{H}\quad\text{for}\quad t\in[0,\bar{T}]
\end{equation}
with $\kappa_{H}$ sufficiently small,
and
\begin{equation}\label{rme10}
\max_{0\leq t\leq\bar{T}}E_s(U)\leq C_0,
\end{equation}
for some constant $C_0$.
The assumptions \eqref{rme9} and \eqref{rme10} hold provided that \eqref{rme5} holds and $\epsilon\leq\epsilon_0\ll1$ (one can refer to \cite{SX}).

Under the conditions \eqref{rme9}, it is easy to check that
\begin{equation}\label{rme11}
E_s(U)\sim|\eta|_{X_{\epsilon^2}^s}^2+| {\vv u}|_{X_{\epsilon}^s}^2.
\end{equation}
The proof of \eqref{rme11} is similar to that in \cite{SX} and we omit it.

A standard energy estimate leads to
\begin{equation}\label{rme12}\begin{aligned}
&\frac{d}{dt}E_s(U)=2((1-b\epsilon\Delta)\Lambda^s\partial_tU\,|\,S_U(D)\Lambda^sU)_2\\
&\quad+((1-b\epsilon\Delta)\Lambda^sU\,|\,\partial_tS_U(D)\Lambda^sU)_2
-b\epsilon\bigl([S_U(D),\Delta]\Lambda^sU\,|\,\Lambda^sU_t\bigr)_2\\
&=-2(\Lambda^s\bigl(M(U,D)U\bigr)\,|\,S_U(D)\Lambda^sU)_2+((1-b\epsilon\Delta)\Lambda^s {\vv u}\,|\,\epsilon\partial_t\eta\Lambda^s {\vv u})_2\\
&\quad-b\epsilon^2\bigl([\eta,\Delta]\Lambda^s{\vv u}\,|\,\Lambda^s\partial_t{\vv u}\bigr)_2\\
&\underset{\text{def}}= I+II+III.
\end{aligned}\end{equation}

{\bf Estimate for $I$.} Firstly, one gets
\beno
I=-2([\Lambda^s,M(U,D)]U\,|\,S_U(D)\Lambda^sU)_2-2(M(U,D)\Lambda^sU\,|\,S_U(D)\Lambda^sU)_2\underset{\text{def}}= I_1+I_2.
\eeno
For $I_1$, one has
\beno\begin{aligned}
I_1&=-2([\Lambda^s,\epsilon {\vv u}]\cdot\nabla\eta+[\Lambda^s,\epsilon\eta]\nabla\cdot {\vv u}\,|\,(1-b\epsilon\Delta)\Lambda^s\eta)_2\\
&\quad-2((1-b\epsilon\Delta)\bigl([\Lambda^s,\epsilon {\vv u}]\cdot\nabla {\vv u}\bigr)\,|\,(1+\epsilon\eta)\Lambda^s {\vv u})_2\\
&\underset{\text{def}}= I_{11}+I_{12}.
\end{aligned}\eeno
Thanks to Lemma \ref{L1}, it is easy to get that for $s>2$,
\begin{equation}\label{rme13}\begin{aligned}
|I_{11}|&\lesssim\bigl(|[\Lambda^s,\epsilon {\vv u}]\cdot\nabla\eta|_2+|[\Lambda^s,\epsilon\eta]\nabla\cdot {\vv u}|_2\bigr)|(1-b\epsilon\Delta)\Lambda^s\eta|_2\\
&\lesssim\epsilon| {\vv u}|_{H^s}|\eta|_{H^s}(|\eta|_{H^s}+\epsilon|\eta|_{H^{s+2}})\lesssim\epsilon| {\vv u}|_{X^s_{\epsilon}}|\eta|_{X^s_{\epsilon^2}}^2.
\end{aligned}\end{equation}
For $I_{12}$, integrating by parts, there holds
\beno
I_{12}=-2([\Lambda^s,\epsilon {\vv u}]\cdot\nabla {\vv u}\,|\,(1+\epsilon\eta)\Lambda^s {\vv u})_2-2b\epsilon(\nabla\bigl([\Lambda^s,\epsilon {\vv u}]\cdot\nabla {\vv u}\bigr)\,|\,\nabla \bigl((1+\epsilon\eta)\Lambda^s {\vv u}\bigr))_2
\eeno
which along with \eqref{rme9} and Lemma \ref{L1} implies that
\begin{equation}\label{rme14}\begin{aligned}
|I_{12}|&\lesssim(1+\epsilon|\eta|_\infty)\epsilon| {\vv u}|_{H^s}^3+\epsilon^2|{\bf  u}|_{X^s_\epsilon}^2\bigl((1+\epsilon|\eta|_\infty)|\nabla\Lambda^s {\vv u}|_2+\epsilon|\nabla\eta|_\infty|\Lambda^s {\vv u}|_2\bigr)\\
&\lesssim\epsilon| {\vv u}|_{X^s_\epsilon}^3.
\end{aligned}\end{equation}

Then we get by \eqref{rme13} and \eqref{rme14} that
\begin{equation}\label{rme15}
|I_1|\lesssim\epsilon| {\vv u}|_{X^s_\epsilon}(|\eta|_{X^s_{\epsilon^2}}^2+| {\vv u}|_{X^s_\epsilon}^2).
\end{equation}

For $I_2$, due to the expressions of $M(U,D)$ and $S_U(D)$, we get that
\beno\begin{aligned}
I_2=&-2(\epsilon {\vv u}\cdot\nabla\Lambda^s\eta\,|\,(1-b\epsilon\Delta)\Lambda^s\eta)_2-2((1-b\epsilon\Delta)(\epsilon {\vv u}\cdot\nabla\Lambda^s {\vv u})\,|\,(1+\epsilon\eta)\Lambda^s {\vv u})_2\\
&-2\{((1+\epsilon\eta)\nabla\cdot\Lambda^s {\vv u}\,|\,(1-b\epsilon\Delta)\Lambda^s\eta)_2+((1-b\epsilon\Delta)\nabla\Lambda^s\eta\,|\,(1+\epsilon\eta)\Lambda^s {\vv u})_2\}\\
\underset{\text{def}}=&I_{21}+I_{22}+I_{23}.
\end{aligned}\eeno
Integrating by parts, one gets that
\beno\begin{aligned}
&I_{21}=\epsilon(\nabla\cdot {\vv u}\Lambda^s\eta\,|\,\Lambda^s\eta)_2+b\epsilon^2(\nabla\cdot {\vv u}\nabla\Lambda^s\eta\,|\,\nabla\Lambda^s\eta)_2-2b\epsilon^2\sum_{i=1}^2(\partial_i{\vv u}\cdot\nabla\Lambda^s\eta\,|\,\partial_i\Lambda^s\eta)_2\\
&I_{22}=\epsilon(\nabla\cdot\bigl((1+\epsilon\eta) {\vv u}\bigr)\Lambda^s {\vv u}\,|\,\Lambda^s{\bf  u})_2
+2b\epsilon^3(( {\vv u}\cdot\nabla)\Lambda^s {\vv u}\,|\,\sum_{i=1}^2\partial_i(\partial_i\eta\Lambda^s {\vv u}))_2\\
&\qquad-2b\epsilon^2\sum_{i=1}^2(\partial_i {\vv u}\cdot\nabla\Lambda^s {\vv u}\,|\,(1+\epsilon\eta)\partial_i\Lambda^s {\vv u})_2
+b\epsilon^2(\nabla\cdot\bigl((1+\epsilon\eta) {\vv u}\bigr)\nabla\Lambda^s {\vv u}\,|\,\nabla\Lambda^s {\vv u})_2,\\
&I_{23}=2\epsilon(\nabla\eta\cdot\Lambda^s{\vv u}\,|\,(1-b\epsilon\Delta)\Lambda^s\eta)_2.
\end{aligned}\eeno
Then thanks to \eqref{rme9}, \eqref{rme10}, \eqref{rme11} and \eqref{int1}, there holds
\begin{equation}\label{rme16}
|I_2|\lesssim\epsilon| {\vv u}|_{X^s_\epsilon}\bigl(|\eta|_{X^s_{\epsilon^2}}^2+| {\vv u}|_{X^s_\epsilon}^2\bigr).
\end{equation}

Thanks to \eqref{rme15} and \eqref{rme16}, we obtain
\begin{equation}\label{rme17}
|I|\lesssim\epsilon| {\vv u}|_{X^s_\epsilon}\bigl(|\eta|_{X^s_{\epsilon^2}}^2+| {\vv u}|_{X^s_\epsilon}^2\bigr).
\end{equation}

{\bf Estimate for $II$.} Integrating by parts, we have
\beno
II=\epsilon(\Lambda^s {\vv u}\,|\,\partial_t\eta\Lambda^s {\vv u})_2+b\epsilon^2(\nabla\Lambda^s {\vv u}\,|\,\nabla(\partial_t\eta\Lambda^s {\vv u}))_2
\eeno
which along with \eqref{int1} implies that
\begin{equation}\label{rme18}
|II|\lesssim\epsilon|\partial_t\eta|_{X^{s-1}_{\epsilon^2}}| {\vv u}|_{X^s_\epsilon}^2.
\end{equation}

{\bf Estimate for $III$.} Thanks to Lemma \ref{L1}, we get that
\begin{equation}\label{rme18a}
|III|\lesssim\epsilon^2|\nabla\eta|_{H^{s-1}}|{\vv u}|_{H^{s+1}}|\Lambda^s\partial_t{\vv u}|_2
\lesssim\epsilon|\eta|_{X^s_{\epsilon^2}}|{\vv u}|_{X^s_\epsilon}|\partial_t{\vv u}|_{X^{s-1}_\epsilon}.
\end{equation}

Combining \eqref{rme12}, \eqref{rme17}, \eqref{rme18} and \eqref{rme18a}, we obtain that
\begin{equation}\label{rme19}
\frac{d}{dt}E_s(U)\lesssim\epsilon\bigl(|\eta|_{X^s_{\epsilon^2}}+| {\vv u}|_{X^s_\epsilon}\bigr)\bigl(|\eta|_{X^s_{\epsilon^2}}^2+| {\vv u}|_{X^s_\epsilon}^2
+|\partial_t\eta |_{X^{s-1}_{\epsilon^2}}^2+|\partial_t{\vv u}|_{X^{s-1}_\epsilon}^2\bigr).
\end{equation}

Thanks to the equations of \eqref{rme3}, one gets by using \eqref{rme9}  and \eqref{rme10} that
\begin{equation}\label{rme20}\begin{aligned}
|\partial_t\eta|_{X^{s-1}_{\epsilon^2}}+|\partial_t\vv u|_{X^{s-1}_{\epsilon}}
&\lesssim(1+\epsilon|\eta|_\infty)| {\vv u}|_{X^s_\epsilon}+\epsilon| {\vv u}|_\infty|\eta|_{H^s}
+|\eta|_{X^s_{\epsilon^2}}+| {\vv u}|_{X^s_\epsilon}^2\\
&\lesssim|\eta|_{X^s_{\epsilon^2}}+| {\vv u}|_{X^s_\epsilon},
\end{aligned}\end{equation}
which along with \eqref{rme19} implies that
\begin{equation}\label{rme21}
\frac{d}{dt}E_s(U)\lesssim\epsilon\bigl(|\eta|_{X^s_{\epsilon^2}}+|{\bf  u}|_{X^s_\epsilon}\bigr)\bigl(|\eta|_{X^s_{\epsilon^2}}^2+| {\vv u}|_{X^s_\epsilon}^2\bigr).
\end{equation}
Then due to \eqref{rme11}, there holds
\beno
\frac{d}{dt}\bigl(E_s(U)\bigr)^{\frac{1}{2}}\leq C_1\epsilon E_s(U),
\eeno
which gives rise to
\begin{equation}\label{rme22a}
\bigl(E_s(U)\bigr)^{\frac{1}{2}}\leq\frac{\bigl(E_s(U_0)\bigr)^{\frac{1}{2}}}{1-C_1\epsilon t\bigl(E_s(U_0)\bigr)^{\frac{1}{2}}}\leq 2\bigl(E_s(U_0)\bigr)^{\frac{1}{2}},
\end{equation}
for any $t\leq\frac{1}{2C_1\bigl(E_s(U_0)\bigr)^{\frac{1}{2}}}\frac{T}{\epsilon}$ with $T=\frac{1}{2C_1\bigl(E_s(U_0)\bigr)^{\frac{1}{2}}}$.
This completes the proof of Theorem \ref{rmT1}.
\end{proof}

\subsection{Case $d>0, a=b=c=0$. }
\begin{theorem}\label{rmT2}
Let $d>0, a=b=c=0$. $n=1,2$, $s>1+\frac{n}{2}$ .
Assume that $\eta_0\in X^s_{\epsilon}(\R^n),\vv u_0\in X^s_{\epsilon^2}(\R^n)$  satisfy the (non-cavitation) condition
\begin{equation}\label{rme24}
1+\epsilon\eta_0\geq H>0,\quad H\in(0,1),
\end{equation}
Then there exists a constant $\tilde{c}_0$  such that for any $\epsilon\leq\epsilon_0=\frac{1-H}{\tilde{c}_0(| \eta_0|_{X^s_{\epsilon}}+|\vv u_0|_{X^s_{\epsilon^2}})}$, there
exists $T>0$ independent of $\epsilon$, such that \eqref{abcd}-\eqref{rme2} has a unique solution $(\eta,\vv u)^T$ with
$\eta\in C([0,T/\epsilon];X^s_{\epsilon}(\R^n))$ and $\vv u\in C([0,T/\epsilon];X^s_{\epsilon^2}(\R^n))$. Moreover,
\begin{equation}\label{rme25}
\max_{t\in[0,T/\epsilon]} (|\eta|_{X^s_{\epsilon}}+|\vv u|_{X^s_{\epsilon^2}})\leq \tilde c (|\eta_0|_{X^s_{\epsilon}}+|\vv u_0|_{X^s_{\epsilon^2}}).
\end{equation}
\end{theorem}

\begin{remark}
As was previously mentioned, one gets global well-posedness in the one-dimensional case (\cite{A, Sc}) in a different functional setting though but the method of proof in \cite{A,Sc} does not seem to adapt to the two-dimensional case since it relies strongly on properties of the {\it one-dimensional}  hyperbolic Saint-Venant (shallow water) system.
\end{remark}
\begin{proof}
The proof also follows the same method used in \cite{SX}. Here we only give the {\it a priori estimates}. For $d>0, a=b=c=0$, we rewrite \eqref{abcd} in two-dimensional space as follows:
\begin{equation}\label{rme26}\left\{\begin{aligned}
&\partial_t\eta+\nabla\cdot\vv u+\epsilon\nabla\cdot(\eta\vv u)=0,\\
&\partial_t\vv u+\nabla\eta+\frac{\epsilon}{2}\nabla(|\vv u|^2)-d\epsilon\Delta\partial_t\vv u=\vv 0,
\end{aligned}\right.\end{equation}
Denoting by $U=(\eta,\vv u)$, \eqref{rme26} is equivalent to the following condensed system
\begin{equation}\label{rme27}
(1-d\epsilon\Delta)\partial_tU+M(U,D)U=0,
\end{equation}
where
\beno
M(U,D)=\begin{pmatrix}
\epsilon(1-d\epsilon\Delta)(\vv u\cdot\nabla)&(1-d\epsilon\Delta)\bigl((1+\epsilon\eta)\partial_1\bigr)&(1-d\epsilon\Delta)\bigl((1+\epsilon\eta)\partial_2\bigr)\\
\partial_1&\epsilon u_1\partial_1&\epsilon u_2\partial_1\\
\partial_2&\epsilon u_1\partial_2&\epsilon u_2\partial_2
\end{pmatrix}.\eeno
The symmetrizer $S_U(D)$ for $M(U,D)$ is defined by
\beno\begin{pmatrix}
1&\epsilon u_1&\epsilon u_2\\
\epsilon u_1&(1+\epsilon\eta)(1-d\epsilon\Delta)&0\\
\epsilon u_2&0&(1+\epsilon\eta)(1-d\epsilon\Delta)
\end{pmatrix}
+\begin{pmatrix}
0&0&0\\
0&d\epsilon^3u_1u_1\Delta&d\epsilon^3u_1u_2\Delta\\
0&d\epsilon^3u_1u_2\Delta&d\epsilon^3u_2u_2\Delta
\end{pmatrix}.\eeno

We define the energy functional associated to \eqref{rme27} as
\begin{equation}\label{rme28}
E_s(U)=((1-d\epsilon\Delta)\Lambda^sU\,|\,S_U(D)\Lambda^sU)_2
\end{equation}
Assume that
\begin{equation}\label{rme29}
1+\epsilon\eta\geq H>0,\quad \epsilon|\vv u|_{W^{1,\infty}}\leq \kappa_{H}\quad\text{for}\quad t\in[0,\bar{T}]
\end{equation}
with $\kappa_{H}$ sufficiently small, and
\begin{equation}\label{rme30}
\max_{0\leq t\leq\bar{T}}E_s(U)\leq C_0,
\end{equation}
for some constants $C_0$.
The assumptions \eqref{rme29} and \eqref{rme30} also hold provided that \eqref{rme24} holds and $\epsilon\leq\epsilon_0\ll1$ (one can refer to \cite{SX}).
Under the assumption \eqref{rme29}, it is not difficult to check that
\begin{equation}\label{rme31}
E_s(U)\sim|\eta|_{X^s_\epsilon}^2+|\vv u|_{X^s_{\epsilon^2}}^2.
\end{equation}

As usual, a standard computation shows that
\begin{equation}\label{rme32}\begin{aligned}
&\frac{d}{dt}E_s(U)=-(\Lambda^s\bigl(M(U,D)U\bigr)\,|\,\bigl(S_U(D)+S_U(D)^*\bigr)\Lambda^sU)_2\\
&\qquad-d\epsilon([S_U(D)^*,\Delta]\Lambda^sU\,|\,\Lambda^s\partial_tU)_2
+((1-d\epsilon\Delta)\Lambda^sU\,|\,\partial_tS_U(D)\Lambda^sU)_2\\
& \underset{\text{def}} = I+II+III,
\end{aligned}\end{equation}
where $S_U(D)^*$ is the adjoint matrix of $S_U(D)$.

{\bf Estimate for $I$.} One has that
\beno\begin{aligned}
I=&-([\Lambda^s,M(U,D)]U\,|\,\bigl(S_U(D)+S_U(D)^*\bigr)\Lambda^sU)_2\\
&-(\Lambda^sU\,|\,\bigl(S_U(D)+S_U(D)^*\bigr)\bigl(M(U,D)\Lambda^sU\bigr))_2\\
\underset{\text{def}} =&I_1+I_2
\end{aligned}\eeno

{\it Estimate for $I_1$.} Using the expressions of $M(U,D)$ and $S_U(D)$, one gets that
\beno\begin{aligned}
&([\Lambda^s,M(U,D)]U\,|\,S_U(D)\Lambda^sU)_2=([\Lambda^s,\epsilon(1-d\epsilon\Delta)(\vv u\cdot\nabla)]\eta\,|\,\Lambda^s\eta+\epsilon\vv u\cdot\Lambda^s\vv u)_2\\
&\quad+([\Lambda^s,\epsilon(1-d\epsilon\Delta)(\eta\nabla)]\cdot\vv u\,|\,\Lambda^s\eta+\epsilon\vv u\cdot\Lambda^s\vv u)_2\\
&\quad+\sum_{i,j=1}^2([\Lambda^s,\epsilon u_j]\partial_iu_j\,|\,\epsilon u_i\Lambda^s\eta+(1+\epsilon\eta)(1-d\epsilon\Delta)\Lambda^s u_i
+d\epsilon^3u_i\vv u\cdot\Delta\Lambda^s\vv u)_2\\
&\underset{\text{def}} = I_{11}+I_{12}+I_{13}.
\end{aligned}\eeno
Integrating by parts, there hold
\beno\begin{aligned}
&I_{11}=\epsilon([\Lambda^s,\vv u]\cdot\nabla\eta\,|\,\Lambda^s\eta+\epsilon\vv u\cdot\Lambda^s\vv u)_2+d\epsilon^2(\nabla\bigl([\Lambda^s,\vv u]\cdot\nabla\eta\bigr)\,|\,\nabla\bigl(\Lambda^s\eta+\epsilon\vv u\cdot\Lambda^s\vv u\bigr))_2\\
&I_{12}=\epsilon([\Lambda^s,\eta]\nabla\cdot\vv u\,|\,\Lambda^s\eta+\epsilon\vv u\cdot\Lambda^s\vv u)_2
+d\epsilon^2(\nabla\bigl([\Lambda^s,\eta]\nabla\cdot\vv u\bigr)\,|\,\nabla\bigl(\Lambda^s\eta+\epsilon\vv u\cdot\Lambda^s\vv u\bigr))_2,
\end{aligned}\eeno
which along with \eqref{rme29}, \eqref{rme30}, \eqref{b10} and \eqref{int1} imply that
\begin{equation}\label{rme33}\begin{aligned}
&|I_{11}|+|I_{12}|\lesssim\epsilon|\vv u|_{H^s}|\eta|_{H^s}\bigl(|\eta|_{H^s}+\epsilon|\vv u|_{H^s}^2\bigr)\\
&\qquad
+\epsilon^2\bigl(|\nabla\vv u|_{H^{s-1}}|\nabla\eta|_{H^s}+|\nabla\vv u|_{H^s}|\nabla\eta|_{H^{s-1}}\bigr)\bigl(|\eta|_{H^{s+1}}+\epsilon|\vv u|_{H^s}|\vv u|_{H^{s+1}}\bigr)\\
&\lesssim\epsilon|\vv u|_{X^s_{\epsilon^2}}\bigl(|\eta|_{X^s_\epsilon}^2+|\vv u|_{X^s_{\epsilon^2}}^2\bigr).
\end{aligned}\end{equation}
Thanks to \eqref{b10}, \eqref{int1}, \eqref{rme30} and \eqref{rme31}, there holds
\beno
|I_{13}|\lesssim\epsilon|\vv u|_{H^s}^2\bigl(\epsilon|\vv u|_{H^s}|\eta|_{H^s}+|\vv u|_{X^s_{\epsilon^2}}+d\epsilon^3|\vv u|_{X^s_{\epsilon^2}}^3\bigr)
\lesssim\epsilon|\vv u|_{X^s_{\epsilon^2}}^2\bigl(|\eta|_{X^s_\epsilon}+|\vv u|_{X^s_{\epsilon^2}}\bigr),
\eeno
which along with \eqref{rme33} shows that
\beno
|([\Lambda^s,M(U,D)]U\,|\,S_U(D)\Lambda^sU)_2|\lesssim\epsilon|\vv u|_{X^s_{\epsilon^2}}\bigl(|\eta|_{X^s_\epsilon}^2+|\vv u|_{X^s_{\epsilon^2}}^2\bigr).
\eeno
The same estimate holds for term $([\Lambda^s,M(U,D)]U\,|\,S_U(D)^*\Lambda^sU)_2$. Then we obtain that
\begin{equation}\label{rme34}
|I_1|\lesssim\epsilon|\vv u|_{X^s_{\epsilon^2}}\bigl(|\eta|_{X^s_\epsilon}^2+|\vv u|_{X^s_{\epsilon^2}}^2\bigr).
\end{equation}

{\it Estimate for $I_2$.} For $I_2$, we first calculate $S_U(D)(M(U,D))=A(U,D)=(a_{ij})$ as follows
\beno\begin{aligned}
&a_{11}=\epsilon(1-d\epsilon\Delta)(\vv u\cdot\nabla)+\epsilon\vv u\cdot\nabla=2\epsilon\vv u\cdot\nabla-d\epsilon^2\Delta(\vv u\cdot\nabla),\\
&a_{12}=(1-d\epsilon\Delta)((1+\epsilon\eta)\partial_1)+\epsilon^2u_1\vv u\cdot\nabla,\\
&a_{13}=(1-d\epsilon\Delta)((1+\epsilon\eta)\partial_2)+\epsilon^2u_2\vv u\cdot\nabla,\\
&a_{21}=(1+\epsilon\eta)(1-d\epsilon\Delta)\partial_1+\epsilon^2u_1\vv u\cdot-d\epsilon^3u_1[\Delta,\vv u]\cdot\nabla,\\
&a_{22}=\epsilon u_1(1-d\epsilon\Delta)((1+\epsilon\eta)\partial_1)+\epsilon(1+\epsilon\eta)(1-d\epsilon\Delta)(u_1\partial_1)+d\epsilon^4u_1\vv u\cdot\Delta(u_1\nabla),\\
&a_{23}=\epsilon u_1(1-d\epsilon\Delta)((1+\epsilon\eta)\partial_2)+\epsilon(1+\epsilon\eta)(1-d\epsilon\Delta)(u_2\partial_1)+d\epsilon^4u_1\vv u\cdot\Delta(u_2\nabla),\\
&a_{31}=(1+\epsilon\eta)(1-d\epsilon\Delta)\partial_2+\epsilon^2u_2(1-d\epsilon\Delta)(\vv u\cdot\nabla)+d\epsilon^3u_2\vv u\cdot\nabla\Delta,\\
&a_{32}=\epsilon u_2(1-d\epsilon\Delta)((1+\epsilon\eta)\partial_1)+\epsilon(1+\epsilon\eta)(1-d\epsilon\Delta)(u_1\partial_2)+d\epsilon^4u_2\vv u\cdot\Delta(u_1\nabla),\\
&a_{33}=\epsilon u_2(1-d\epsilon\Delta)((1+\epsilon\eta)\partial_2)+\epsilon(1+\epsilon\eta)(1-d\epsilon\Delta)(u_2\partial_2)+d\epsilon^4u_2\vv u\cdot\Delta(u_2\nabla).
\end{aligned}\eeno

Now , we calculate $(S_U(D)(M(U,D)\Lambda^sU)\,|\,\Lambda^sU)_2=(A(U,D)\Lambda^sU\,|\,\Lambda^sU)_2$.

For $a_{11}$, one has
\beno\begin{aligned}
&(a_{11}\Lambda^s\eta\,|\,\Lambda^s\eta)_2=2\epsilon(\vv u\cdot\nabla\Lambda^s\eta\,|\,\Lambda^s\eta)_2-d\epsilon^2(\Delta(\vv u\cdot\nabla\Lambda^s\eta)\,|\,\Lambda^s\eta)_2\\
&=-\epsilon(\nabla\cdot\vv u\Lambda^s\eta\,|\,\Lambda^s\eta)_2
-\frac{1}{2}d\epsilon^2(\nabla\cdot\vv u\nabla\Lambda^s\eta\,|\,\nabla\Lambda^s\eta)_2
+d\epsilon^2\sum_{i=1}^2(\partial_i\vv u\cdot\nabla\Lambda^s\eta\,|\,\partial_i\Lambda^s\eta)_2,
\end{aligned}\eeno
which shows that
\begin{equation}\label{rme35}
|(a_{11}\Lambda^s\eta\,|\,\Lambda^s\eta)_2|\lesssim\epsilon|\vv u|_{H^s}|\eta|_{X^s_\epsilon}^2.
\end{equation}

For $a_{22}$, one gets
\beno\begin{aligned}
&(a_{22}\Lambda^su_1\,|\,\Lambda^su_1)_2=\epsilon\{(u_1(1-d\epsilon\Delta)\bigl((1+\epsilon\eta)\partial_1\Lambda^su_1\bigr)\,|\,\Lambda^su_1)_2\\
&\quad
+((1+\epsilon\eta)(1-d\epsilon\Delta)(u_1\partial_1\Lambda^su_1)\,|\,\Lambda^su_1)_2\}+d\epsilon^4(u_1\vv u\cdot\Delta(u_1\nabla\Lambda^su_1)\,|\,\Lambda^su_1)_2\\
&=-\epsilon(\Lambda^su_1\,|\,\epsilon\partial_1\eta(1-d\epsilon\Delta)(u_1\Lambda^su_1)
+(1+\epsilon\eta)(1-d\epsilon\Delta)(\partial_1u_1\Lambda^su_1))_2\\
&\quad-d\epsilon^4\sum_{i=1}^2(\partial_i(u_1\nabla\Lambda^su_1)\,|\,\partial_i(u_1\vv u\Lambda^su_1))_2
\end{aligned}\eeno
which along with \eqref{rme29},\eqref{rme30} and \eqref{int1} gives rise to
\begin{equation}\label{rme36}
|(a_{22}\Lambda^su_1\,|\,\Lambda^su_1)_2|\lesssim\epsilon|\vv u|_{H^s}|\vv u|_{X^s_{\epsilon^2}}^2.
\end{equation}
The same estimate holds for term $(a_{33}\Lambda^su_2\,|\,\Lambda^su_2)_2$.

For $a_{12}$ and $a_{21}$, one calculates that
\beno\begin{aligned}
&(a_{12}\Lambda^su_1\,|\,\Lambda^s\eta)_2+(a_{21}\Lambda^s\eta\,|\,\Lambda^su_1)_2\\
&=\{((1-d\epsilon\Delta)\bigl((1+\epsilon\eta)\partial_1\Lambda^su_1\bigr)\,|\,\Lambda^s\eta)_2
+((1+\epsilon\eta)(1-d\epsilon\Delta)\partial_1\Lambda^s\eta\,|\,\Lambda^su_1)_2\}\\
&\quad+\{\epsilon^2(u_1\vv u\cdot\nabla\Lambda^su_1\,|\,\Lambda^s\eta)_2+\epsilon^2(u_1\vv u\cdot\nabla\Lambda^s\eta\,|\,\Lambda^su_1)_2
-d\epsilon^3(u_1[\Delta,\vv u]\cdot\nabla\Lambda^s\eta\,|\,\Lambda^su_1)_2\}\\
&=-\epsilon((1-d\epsilon\Delta)(\partial_1\eta\Lambda^su_1)\,|\,\Lambda^s\eta)_2-\epsilon^2(\nabla\cdot(u_1\vv u)\Lambda^su_1\,|\,\Lambda^s\eta)_2\\
&\quad+\epsilon^3(\nabla\Lambda^s\eta\,|\,[\Delta,\vv u](u_1\Lambda^su_1))_2,
\end{aligned}\eeno
which along with \eqref{rme29}, \eqref{rme30} and \eqref{int1} implies
\begin{equation}\label{rme37}\begin{aligned}
&|(a_{12}\Lambda^su_1\,|\,\Lambda^s\eta)_2+(a_{21}\Lambda^s\eta\,|\,\Lambda^su_1)_2|\\
&\lesssim\epsilon|\eta|_{X^s_\epsilon}^2|\vv u|_{X^s_{\epsilon^2}}+\epsilon^2|\vv u|_{X^s_{\epsilon^2}}^3|\eta|_{X^s_\epsilon}\\
&\lesssim\epsilon|\vv u|_{X^s_{\epsilon^2}}
\bigl(|\eta|_{X^s_\epsilon}^2+|\vv u|_{X^s_{\epsilon^2}}^2\bigr).
\end{aligned}\end{equation}
The same estimate holds for $(a_{13}\Lambda^su_2\,|\,\Lambda^s\eta)_2+(a_{31}\Lambda^s\eta\,|\,\Lambda^su_2)_2$.

At last, for $a_{23}$ and $a_{32}$, one estimates that
\beno\begin{aligned}
&(a_{23}\Lambda^su_2\,|\,\Lambda^su_1)_2+(a_{32}\Lambda^su_1\,|\,\Lambda^su_2)_2\\
&=\epsilon\{(u_1(1-d\epsilon\Delta)\bigl((1+\epsilon\eta)\partial_2\Lambda^su_2\bigr)\,|\,\Lambda^su_1)_2
+((1+\epsilon\eta)(1-d\epsilon\Delta)(u_1\partial_2\Lambda^su_1)\,|\,\Lambda^su_2)_2\}\\
&\quad+\epsilon\{((1+\epsilon\eta)(1-d\epsilon\Delta)(u_2\partial_1\Lambda^su_2)\,|\,\Lambda^su_1)_2
+(u_2(1-d\epsilon\Delta)\bigl((1+\epsilon\eta)\partial_1\Lambda^su_1\bigr)\,|\,\Lambda^su_2)_2\}\\
&\quad+d\epsilon^4\{(u_1\vv u\cdot\Delta(u_2\nabla\Lambda^su_2)\,|\,\Lambda^su_1)_2+(u_2\vv u\cdot\Delta(u_1\nabla\Lambda^su_1)\,|\,\Lambda^su_2)_2\}\\
&=-\epsilon(\epsilon\partial_2\eta(1-d\epsilon\Delta)(u_1\Lambda^su_1)+(1+\epsilon\eta)(1-d\epsilon\Delta)(\partial_2u_1\Lambda^su_1)\,|\,\Lambda^su_2)_2\\
&\quad-\epsilon(\epsilon\partial_1\eta(1-d\epsilon\Delta)(u_2\Lambda^su_2)+(1+\epsilon\eta)(1-d\epsilon\Delta)(\partial_1u_2\Lambda^su_2)\,|\,\Lambda^su_1)_2\\
&\quad-d\epsilon^4\{(\nabla u_2\cdot\Delta(u_1\vv u\Lambda^su_1)+u_2\Delta\bigl(\nabla\cdot(u_1\vv u)\Lambda^su_1\bigr)\,|\,\Lambda^su_2)_2\\
&\quad-(2u_2\sum_{i=1}^2\partial_i\vv u\cdot\partial_i(u_1\nabla\Lambda^su_1)+u_2\Delta\vv u\cdot u_1\nabla\Lambda^su_1\,|\,\Lambda^su_2)_2\},
\end{aligned}\eeno
which together with \eqref{rme29}, \eqref{rme30} and \eqref{int1} leads to
\begin{equation}\label{rme38}\begin{aligned}
&|(a_{23}\Lambda^su_2\,|\,\Lambda^su_1)_2+(a_{32}\Lambda^su_1\,|\,\Lambda^su_2)_2|\\
&\lesssim\epsilon(1+\epsilon|\eta|_{X^s_\epsilon})|\vv u|_{X^s_{\epsilon^2}}^3+\epsilon^3|\vv u|_{X^s_{\epsilon^2}}^5\lesssim\epsilon|\vv u|_{X^s_{\epsilon^2}}^3.
\end{aligned}\end{equation}

Thanks to \eqref{rme35}, \eqref{rme36}, \eqref{rme37} and \eqref{rme38}, we obtain that
\beno
|(S_U(D)(M(U,D)\Lambda^sU)\,|\,\Lambda^sU)_2|\lesssim\epsilon\bigl(|\eta|_{X^s_\epsilon}+|\vv u|_{X^s_{\epsilon^2}}\bigr)^3,
\eeno
provided that there hold \eqref{rme29} and \eqref{rme30}. The same estimate holds for $(S_U(D)^*(M(U,D)\Lambda^sU)\,|\,\Lambda^sU)_2$. Then we obtain
that
\begin{equation}\label{rme39}
|I_2|\lesssim\epsilon\bigl(|\eta|_{X^s_\epsilon}+|\vv u|_{X^s_{\epsilon^2}}\bigr)^3.
\end{equation}

Due to \eqref{rme34} and \eqref{rme39}, we get that
\begin{equation}\label{rme40}
|I|\lesssim\epsilon\bigl(|\eta|_{X^s_\epsilon}+|\vv u|_{X^s_{\epsilon^2}}\bigr)^3.
\end{equation}

{\bf Estimate for $II$.} Using the expression of $S_U(D)$, one obtains that
\beno\begin{aligned}
&II=-d\epsilon\Bigl(\epsilon([\vv u,\Delta]\cdot\Lambda^s\vv u\,|\,\Lambda^s\partial_t\eta)_2+\epsilon([\vv u,\Delta]\Lambda^s\eta\,|\,\Lambda^s\partial_t\vv u)_2\\
&\quad+\epsilon((1-d\epsilon\Delta)\bigl([\eta,\Delta]\Lambda^s\vv u\bigr)\,|\,\Lambda^s\partial_t\vv u)_2+d\epsilon^3\sum_{i,j=1}^2(\Delta\bigl([u_iu_j,\Delta]\Lambda^su_j\bigr)\,|\,\Lambda^s\partial_tu_i)_2\Bigr).
\end{aligned}\eeno
Along the same line as previous work, by virtue of \eqref{b10}, \eqref{int1}, \eqref{rme29}, \eqref{rme30} and integrating by parts, we finally get that
\begin{equation}\label{rme41}
|II|\lesssim\epsilon\bigl(|\eta|_{X^s_\epsilon}^2+|\vv u|_{X^s_{\epsilon^2}}^2\bigr)\bigl(|\partial_t\eta|_{X^{s-1}_\epsilon}+|\partial_t\vv u|_{X^{s-1}_{\epsilon^2}}\bigr).
\end{equation}

{\bf Estimate for $III$.}  Using the expression of $S_U(D)$ again, one gets that
\beno\begin{aligned}
&III=\epsilon((1-d\epsilon\Delta)\Lambda^s\eta\,|\,\partial_t\vv u\cdot\Lambda^s\vv u)_2
+\epsilon((1-d\epsilon\Delta)\Lambda^s\vv u\,|\,\partial_t\vv u\Lambda^s\eta)_2\\
&\qquad
+\epsilon((1-d\epsilon\Delta)\Lambda^s\vv u\,|\,\partial_t\eta(1-d\epsilon\Delta)\Lambda^s\vv u)_2
+d\epsilon^3\sum_{j=1}^2((1-d\epsilon\Delta)\Lambda^su_j\,|\,\partial_t(u_j\vv u)\cdot\Delta\Lambda^s\vv u)_2.
\end{aligned}\eeno
Note that
\beno
((1-d\epsilon\Delta)\Lambda^s\eta\,|\,\partial_t\vv u\cdot\Lambda^s\vv u)_2=(\Lambda^s\eta\,|\,\partial_t\vv u\cdot\Lambda^s\vv u)_2
+d\epsilon(\nabla\Lambda^s\eta\,|\,\nabla(\partial_t\vv u\cdot\Lambda^s\vv u))_2
\eeno
Then \eqref{rme29}, \eqref{rme30} and \eqref{int1} leads to
\begin{equation}\label{rme42}
|III|\lesssim\epsilon\bigl(|\eta|_{X^s_\epsilon}^2+|\vv u|_{X^s_{\epsilon^2}}^2\bigr)\bigl(|\partial_t\eta|_{X^{s-1}_{\epsilon}}+|\partial_t\vv u|_{X^{s-1}_{\epsilon^2}}\bigr)
\end{equation}

Combining \eqref{rme32}, \eqref{rme40}, \eqref{rme41} and \eqref{rme42}, we obtain that
\begin{equation}\label{rme43}
\frac{d}{dt}E_s(U)\lesssim\epsilon\bigl(|\eta|_{X^s_\epsilon}+|\vv u|_{X^s_{\epsilon^2}}\bigr)
\bigl(|\eta|_{X^s_\epsilon}^2+|\vv u|_{X^s_{\epsilon^2}}^2+|\partial_t\eta|_{X^{s-1}_\epsilon}^2+|\partial_t\vv u|_{X^{s-1}_{\epsilon^2}}^2\bigr).
\end{equation}

Thanks to \eqref{rme26}, we get by using \eqref{rme29} and \eqref{rme30} that
\begin{equation}\label{rme44}
|\partial_t\eta|_{X^{s-1}_\epsilon}+|\partial_t\vv u|_{X^{s-1}_{\epsilon^2}}\lesssim|\eta|_{X^s_\epsilon}+|\vv u|_{X^s_{\epsilon^2}},
\end{equation}
which along with \eqref{rme43} implies
\begin{equation}\label{rme45}
\frac{d}{dt}E_s(U)\lesssim\epsilon\bigl(|\eta|_{X^s_\epsilon}+|\vv u|_{X^s_{\epsilon^2}}\bigr)
\bigl(|\eta|_{X^s_\epsilon}^2+|\vv u|_{X^s_{\epsilon^2}}^2\bigr).
\end{equation}

Then due to \eqref{rme31}, there holds
\beno
\frac{d}{dt}\bigl(E_s(U)\bigr)^{\frac{1}{2}}\leq C_1\epsilon E_s(U).
\eeno
Similarly as the proof to Theorem \ref{rmT1}, there exists $T=\frac{1}{2C_1\bigl(E_s(U_0)\bigr)^{\frac{1}{2}}}$ such that \eqref{rme25} holds.
This completes the proof of Theorem \ref{rmT2}.
\end{proof}

We now turn to the "Schr\"{o}dinger like" Boussinesq systems.

\subsection{The case  $b=d=c=0, a<0$}

This case can be treated by following the lines developed in \cite{SX}. For the sake of completeness we provide some details now.
\begin{theorem}\label{rmT3}
Let $b=c=d=0, a=-1$, $n=1,2$, $s>2+\frac{n}{2}$ .
Assume that $\eta_0\in H^s(\R^n),\vv u_0\in X^s_{\epsilon}(\R^n)$  satisfy the (non-cavitation) condition
\begin{equation}\label{arme24}
1+\epsilon\eta_0\geq H>0,\quad H\in(0,1),
\end{equation}
Then there exists a constant $\tilde{c}_0$  such that for any $\epsilon\leq\epsilon_0=\frac{1-H}{\tilde{c}_0(| \eta_0|_{H^s}+|\vv u_0|_{X^s_{\epsilon}})}$, there
exists $T>0$ independent of $\epsilon$, such that \eqref{abcd}-\eqref{rme2} has a unique solution $(\eta,\vv u)^T$ with
$\eta\in C([0,T/\epsilon];H^s(\R^n))$ and $\vv u\in C([0,T/\epsilon];X^s_{\epsilon}(\R^n))$. Moreover,
\begin{equation}\label{arme25}
\max_{t\in[0,T/\epsilon]} (|\eta|_{H^s}+|\vv u|_{X^s_{\epsilon}})\leq \tilde c (|\eta_0|_{H^s}+|\vv u_0|_{X^s_{\epsilon}}).
\end{equation}
\end{theorem}
\begin{proof}
We only sketch the proof of the two-dimensional case. For $b=c=d=0, a=-1$, we firstly rewrite the two-dimensional version of \eqref{abcd} in the following condensed system
\begin{equation}\label{aa1}
\partial_tU+M(U,D)U=0,
\end{equation}
where $U=(\eta,\vv u)^T$, and
\beno
M(U,D)=\begin{pmatrix}
\epsilon\vv u\cdot\nabla&(1+\epsilon\eta-\epsilon\Delta)\partial_1&(1+\epsilon\eta-\epsilon\Delta)\partial_2\\
\partial_1&\epsilon u_1\partial_1&\epsilon u_2\partial_1\\
\partial_2&\epsilon u_1\partial_2&\epsilon u_2\partial_2
\end{pmatrix}.\eeno
The symmetrizer $S_U(D)$ for $M(U,D)$ is defined by
\beno
S_U(D)=\begin{pmatrix}
1&\epsilon u_1&\epsilon u_2\\
\epsilon u_1&1+\epsilon\eta-\epsilon\Delta&0\\
\epsilon u_2&0&1+\epsilon\eta-\epsilon\Delta.
\end{pmatrix}
\eeno

We define the energy functional associated to \eqref{aa1} as
\begin{equation}\label{arme28}
E_s(U)=(\Lambda^sU\,|\,S_U(D)\Lambda^sU)_2
\end{equation}
Assume that
\begin{equation}\label{arme29}
1+\epsilon\eta\geq H>0,\quad \epsilon|\eta|_{W^{1,\infty}}+\epsilon|\vv u|_{W^{1,\infty}}\leq \kappa_{H}\quad\text{for}\quad t\in[0,\bar{T}]
\end{equation}
with $\kappa_{H}$ sufficiently small, and
\begin{equation}\label{arme30}
\max_{0\leq t\leq\bar{T}}E_s(U)\leq C_0,
\end{equation}
for some constants $C_0$.
The assumptions \eqref{arme29} and \eqref{arme30} also hold provided that \eqref{arme24} holds and $\epsilon\leq\epsilon_0\ll1$ (one can refer to \cite{SX}).
Under the assumption \eqref{arme29}, it is not difficult to check that
\begin{equation}\label{arme31}
E_s(U)\sim|\eta|_{H^s}^2+|\vv u|_{X^s_{\epsilon}}^2.
\end{equation}

As usual, we get by a standard energy estimate that
\begin{equation}\label{arme12}\begin{aligned}
&\frac{d}{dt}E_s(U)=2(\Lambda^s\partial_tU\,|\,S_U(D)\Lambda^sU)_2+(\Lambda^sU\,|\,\partial_tS_U(D)\Lambda^sU)_2\\
&=-2(\Lambda^s\bigl(M(U,D)U\bigr)\,|\,S_U(D)\Lambda^sU)_2+(\Lambda^sU\,|\,\partial_tS_U(D)\Lambda^sU)_2\\
&\underset{\text{def}}= I+II.
\end{aligned}\end{equation}

{\bf Estimate for $II$.} Using the expression of $S_U(D)$ yields that
\beno
II=2\epsilon(\Lambda^s\eta\,|\,\partial_t\vv u\cdot\Lambda^s\vv u)_2
+\epsilon(\Lambda^s\vv u\,|\,\partial_t\eta\Lambda^s\vv u)_2,
\eeno
which implies that for $s>3$,
\begin{equation}\label{aa2}
|II|\lesssim\epsilon|\partial_t\vv u|_{H^{s-1}}|\eta|_{H^s}|\vv u|_{H^s}+\epsilon|\partial_t\eta|_{H^{s-2}}|\vv u|_{H^s}^2.
\end{equation}

{\bf Estimate for $I$.} We first have that
\beno
I=-2([\Lambda^s,M(U,D)]U\,|\,S_U(D)\Lambda^sU)_2-2(M(U,D)\Lambda^sU\,|\,S_U(D)\Lambda^sU)_2\underset{\text{def}}= I_1+I_2.
\eeno

{\it Estimate for $I_1$.} For $I_1$, we get that
\beno\begin{aligned}
I_1&=-2\epsilon([\Lambda^s,\vv u]\cdot\nabla\eta+[\Lambda^s,\eta]\nabla\cdot\vv u\,|\,\Lambda^s\eta+\epsilon\vv u\cdot\Lambda^s\vv u)_2\\
&\quad-2\epsilon\sum_{j=1}^2([\Lambda^s,\vv u]\cdot\partial_j\vv u\,|\,u_j\Lambda^s\eta+(1+\epsilon\eta-\epsilon\Delta)\Lambda^s u_j)_2\\
&\underset{\text{def}}= I_{11}+I_{12}.
\end{aligned}\eeno
Thanks to Lemma \ref{L1} and \eqref{arme29}, it is easy to get that for $s>2$,
\begin{equation}\label{arme13}\begin{aligned}
|I_{11}|&\lesssim\epsilon\bigl(|[\Lambda^s,\vv u]\cdot\nabla\eta|_2+|[\Lambda^s,\eta]\nabla\cdot\vv u|_2\bigr)\bigl(|\Lambda^s\eta|_2+|\epsilon\vv u\cdot\Lambda^s\vv u|_2\bigr)\\
&\lesssim\epsilon| {\vv u}|_{H^s}|\eta|_{H^s}(|\eta|_{H^s}+\epsilon|\vv u|_{H^s}^2)\\
&\lesssim\epsilon| {\vv u}|_{H^s}(|\eta|_{H^s}^2+|\vv u|_{H^s}^2).
\end{aligned}\end{equation}
For $I_{12}$, integrating by parts, there holds
\beno\begin{aligned}
I_{12}=&-2\epsilon\sum_{j=1}^2([\Lambda^s,\vv u]\cdot\partial_j\vv u\,|\,\epsilon u_j\Lambda^s\eta+(1+\epsilon\eta)\Lambda^s u_j)_2\\
&\quad-2\epsilon^2\sum_{j=1}^2(\nabla\bigl([\Lambda^s,\vv u]\cdot\partial_j\vv u\bigr)\,|\,\nabla\Lambda^s u_j)_2
\end{aligned}\eeno
which along with \eqref{arme29} and Lemma \ref{L1} implies that
\begin{equation}\label{arme14}\begin{aligned}
|I_{12}|&\lesssim\epsilon|\vv u|_{H^s}^2\bigl(\epsilon|\vv u|_\infty|\eta|_{H^s}+(1+\epsilon|\eta|_\infty)|\vv u|_{H^s}\bigr)
+\epsilon^2|\vv u|_{H^s}|\vv u|_{H^{s+1}}^2\\
&\lesssim\epsilon|{\vv u}|_{H^s}|\vv u|_{X^s_\epsilon}^2.
\end{aligned}\end{equation}

Thanks to \eqref{arme13} and \eqref{arme14}, we get that
\begin{equation}\label{aa3}
|I_1|\lesssim\epsilon|{\vv u}|_{H^s}\bigl(|\eta|_{H^s}^2+|\vv u|_{X^s_\epsilon}^2\bigr).
\end{equation}

{\it Estimate for $I_2$.} For $I_2$, using the expressions of $M(U,D)$ and $S_U(D)$, we obtain that
\begin{equation}\label{aa4}\begin{aligned}
I_2=&-4\epsilon(\vv u\cdot\nabla\Lambda^s\eta\,|\,\Lambda^s\eta)_2-2\{((1+\epsilon\eta-\epsilon\Delta)\nabla\cdot\Lambda^s\vv u\,|\,\Lambda^s\eta)_2\\
&\quad+(\nabla\Lambda^s\eta\,|\,(1+\epsilon\eta-\epsilon\Delta)\Lambda^s\vv u)_2\}\\
&-2\epsilon^2\{(\vv u\cdot\nabla\Lambda^s\eta\,|\,\vv u\cdot\Lambda^s\vv u)_2+\sum_{i=1}^2(\vv u\cdot\partial_i\Lambda^s\vv u\,|\,u_i\Lambda^s\eta)_2\}\\
&-2\epsilon\{((1+\epsilon\eta-\epsilon\Delta)\nabla\cdot\Lambda^s\vv u\,|\,\vv u\cdot\Lambda^s\vv u)_2\\
&\quad+\sum_{i=1}^2(\vv u\cdot\partial_i\Lambda^s\vv u\,|\,(1+\epsilon\eta-\epsilon\Delta)\Lambda^su_i)_2\}\\
&\underset{\text{def}}=I_{21}+I_{22}+I_{23}+I_{24}.
\end{aligned}\end{equation}
Integrating by parts, we get that
\beno\begin{aligned}
&I_{21}=2\epsilon(\nabla\cdot\vv u\Lambda^s\eta\,|\,\Lambda^s\eta)_2,\quad I_{22}=2\epsilon(\nabla\eta\cdot\Lambda^s\vv u\,|\,\Lambda^s\eta)_2,\\
&
I_{23}=2\epsilon^2(\nabla\cdot\vv u\Lambda^s\eta\,|\,\vv u\cdot\Lambda^s\vv u)_2+2\epsilon^2(\Lambda^s\vv u\,|\,(\vv u\cdot\nabla)\vv u\Lambda^s\eta)_2,\\
&I_{24}=2\epsilon^2(\nabla\eta\cdot\Lambda^s\vv u\,|\,\vv u\cdot\Lambda^s\vv u)_2
+2\epsilon\sum_{i=1}^2((1+\epsilon\eta)\Lambda^su_i\,|\,\partial_i\vv u\cdot\Lambda^s\vv u)_2\\
&\qquad\quad+2\epsilon^2\sum_{i=1}^2(\nabla\Lambda^su_i\,|\,\nabla(\partial_i\vv u\cdot\Lambda^s\vv u))_2,
\end{aligned}\eeno
which along with \eqref{arme29} and  \eqref{arme30} implies that
\begin{equation}\label{aa5}
|I_2|\lesssim\epsilon|{\vv u}|_{H^s}\bigl(|\eta|_{H^s}^2+|\vv u|_{X^s_\epsilon}^2\bigr).
\end{equation}

Thanks to \eqref{aa3} and \eqref{aa5}, we obtain that
\begin{equation}\label{aa6}
|I|\lesssim\epsilon|{\vv u}|_{H^s}\bigl(|\eta|_{H^s}^2+|\vv u|_{X^s_\epsilon}^2\bigr).
\end{equation}

\vspace{0.3cm}

Due to \eqref{abcd} with $b=c=d=0,a=-1$, we deduce by using \eqref{arme29} and \eqref{arme30} that
\begin{equation}\label{aa7}
|\partial_t\eta|_{H^{s-2}}+|\partial_t\vv u|_{H^{s-1}}\lesssim|\eta|_{H^s}+|\vv u|_{X^s_\epsilon}.
\end{equation}

\vspace{0.3cm}

Combining \eqref{arme12}, \eqref{aa2}, \eqref{aa6} and \eqref{aa7}, we finally get that
\begin{equation}\label{arme45}
\frac{d}{dt}E_s(U)\lesssim\epsilon\bigl(|\eta|_{H^s}+|\vv u|_{X^s_{\epsilon}}\bigr)
\bigl(|\eta|_{H^s}^2+|\vv u|_{X^s_{\epsilon}}^2\bigr).
\end{equation}

Due to \eqref{arme31}, there holds
\begin{equation}\label{arme46}
\frac{d}{dt}\bigl(E_s(U)\bigr)^{\frac{1}{2}}\leq C_1\epsilon E_s(U).
\end{equation}

Then following the same line as the proofs of Theorems \ref{rmT1} and \ref{rmT2}, one obtains that there exists $T>0$ independent of $\epsilon$ such that \eqref{abcd}-\eqref{rme2} has a unique solution on time interval $[0,T/\epsilon]$. Moreover,  \eqref{arme25} holds and Theorem \ref{rmT3}
is proved.
\end{proof}

\begin{remark}
The modelling of internal waves at the interface of a two-fluid system with different densities and in the presence of a rigid top leads, in an appropriate regime, to Boussinesq systems that are similar to those studied in the previous sections (see \cite{BLS}, section 3.1.3) and for which one can obtain the same results as in \cite{SX} or in the present paper. The same regime for  a two-fluid system  but with a free upper surface has been considered in \cite{Dudu}, section 2.3.2. One gets a system of four equations for which the methods of \cite{SX} and of the present paper are likely to work, including the case of a slowing varying bottom.  We also refer to \cite {Dudu2} for further investigations on those extended Boussinesq systems, in particular for a construction of symmetrizable ones (modulo $\epsilon^2$ terms).
\end{remark}

\subsection{The difficult case $a=b=d=0, c<0$.} The method to solve the long time existence for this case is quite different from the other cases we dealt in the previous subsections and in the paper \cite{SX}. We will now  quasilinearize the system  by applying time together with space derivatives. The key point here is that we improve the regularity in space by improving the regularity in time (applying space derivatives to the system would cause a loss of  derivatives).

We first state the long time existence result in the one dimensional case :

\begin{theorem}\label{long time existence for case c=-1}
Let $a=b=d=0, c=-1$.
Assume that $\eta_0\in X^2_{\epsilon^3}(\R),u_0\in X^2_{\epsilon^2}(\R)$  satisfy the (non-cavitation) condition
\begin{equation}\label{lte1}
1+\epsilon\eta_0\geq H>0,\quad H\in(0,1),
\end{equation}
Then there exists a constant $\tilde{c}_0$  such that for any $\epsilon\leq\epsilon_0=\frac{1}{\tilde{c}_0(| \eta_0|_{X^2_{\epsilon^3}}+| u_0|_{X^2_{\epsilon^2}})}$, there
exists $T>0$ independent of $\epsilon$, such that \eqref{c<0}-\eqref{rme2} has a unique solution $(\eta,u)$ with
$\eta\in C([0,T/\epsilon];X^2_{\epsilon^3}(\R))$ and $u\in C([0,T/\epsilon];X^2_{\epsilon^2}(\R))$. Moreover,
\begin{equation}\label{lte2}
\begin{aligned}
&\sup_{t\in[0,T/\epsilon]}\bigl(|\eta|_{X^2_{\epsilon^3}}^2+|\eta_t|_{X^1_{\epsilon^2}}^2+|\eta_{tt}|_{X^0_{\epsilon}}^2
+|u|_{X^2_{\epsilon^2}}^2+|u_t|_{X^1_{\epsilon}}^2+|u_{tt}|_2^2\bigr)\\
&\leq C\bigl(|\eta_0|_{X^2_{\epsilon^3}}^2+|u_0|_{X^2_{\epsilon^2}}^2\bigr).
\end{aligned}\end{equation}
\end{theorem}

\begin{remark} System \eqref{c<0} can be viewed as the Saint-Venant (shallow water) system with surface tension and corresponds to system (A.1) in \cite{DIT} with $\mu=0, \delta=1.$ Thus the previous theorem can be compared to Theorem A.3 in \cite{DIT}.
\end{remark}

\begin{remark}
It will be clear in the  following proof  that the regularity we choose for the initial data is the lowest possible one. One could also impose higher regularity on the initial data such as $\eta_0\in X^{2+k}_{\epsilon^{3+k}}(\R),u_0\in X^{2+k}_{\epsilon^{2+k}}(\R)$ for $k\in\N$. In that case, one has to apply $\partial_t$ to  \eqref{abcd} for $k+2$ times. For simplicity, we only consider the case $k=0$.
\end{remark}
\begin{proof}
We shall divide the proof into several steps. We first  indicate how to obtain the a priori energy estimates. We shall use the a priori estimates to prove the existence of the solutions, in Section 5.

{\bf Step 1. Reduction of the system.}
Since $a=b=d=0, c=-1$, we rewrite  \eqref{abcd} in the form \eqref{c<0}. Setting
\beno
v\underset{\text{def}}=(1+\epsilon\eta)u,
\eeno
the first equation of \eqref{c<0} becomes
\beno
\eta_t+v_x=0.
\eeno
Elementary calculations and the  use of \eqref{c<0}   yield the evolution equation for $v :$
\beno
v_t+(1+\epsilon\eta)\eta_x-\epsilon(1+\epsilon\eta)\eta_{xxx}+\epsilon\Bigl(\frac{v^2}{1+\epsilon\eta}\Bigr)_x=0.
\eeno
Indeed, we have
\beno\begin{aligned}
\partial_tv&=(1+\epsilon\eta)u_t+\epsilon u\eta_t\\
&=-(1+\epsilon\eta)(\eta_x-\epsilon\eta_{xxx})-\epsilon(1+\epsilon\eta)uu_x-\epsilon u v_x\\
&=-(1+\epsilon\eta)(\eta_x-\epsilon\eta_{xxx})-\epsilon \Bigl(\frac{v^2}{1+\epsilon\eta}\Bigr)_x.
\end{aligned}\eeno

Then \eqref{c<0} is rewritten in terms of $(\eta,v)$ as follows
\begin{equation}\label{reduction form 1}\left\{\begin{aligned}
&\eta_t+v_x=0,\\
&\frac{1}{1+\epsilon\eta}v_t+\eta_x-\epsilon\eta_{xxx}+\frac{\epsilon}{1+\epsilon\eta}\Bigl(\frac{v^2}{1+\epsilon\eta}\Bigr)_x=0.
\end{aligned}\right.\end{equation}

We shall derive  energy estimates for this system.

\medskip

{\bf Step 2. Quasilinearization of \eqref{reduction form 1}.} In this step, we shall  quasilinearize the system \eqref{reduction form 1}
by applying to it  $\partial_t$ and $\partial_t^2$. Applying $\partial_t$ to the first equation of \eqref{reduction form 1} leads to
\beno\begin{aligned}
\partial_t^2\eta&=-\partial_x v_t=\bigl((1+\epsilon\eta)\eta_x\bigr)_x-\epsilon\bigl((1+\epsilon\eta)\eta_{xxx}\bigr)_x+\epsilon\Bigl(\frac{v^2}{1+\epsilon\eta}\Bigr)_{xx}\\
&=\bigl((1+\epsilon\eta)\eta_x\bigr)_x-\epsilon\bigl((1+\epsilon\eta)\eta_{xxx}\bigr)_x+\frac{2\epsilon v}{1+\epsilon\eta}v_{xx}\\
&\qquad+2\epsilon\frac{v_x^2}{1+\epsilon\eta}-\frac{2\epsilon^2}{(1+\epsilon\eta)^2}vv_x\eta_x
-\epsilon^2\Bigl(\frac{\eta_xv^2}{(1+\epsilon\eta)^2}\Bigr)_x.
\end{aligned}\eeno
One notices that the last term $\frac{2\epsilon v}{1+\epsilon\eta}v_{xx}$ in the second line of the above equality is the higher order term. Since by \eqref{reduction form 1} $v_x=-\eta_t$, we rewrite this term as
\beno
\frac{2\epsilon v}{1+\epsilon\eta}v_{xx}=-\frac{2\epsilon v}{1+\epsilon\eta}\partial_x\eta_t.
\eeno
Then we obtain
\begin{equation}\label{lte3}
\eta_{tt}-\bigl((1+\epsilon\eta)\eta_x\bigr)_x+\epsilon\bigl((1+\epsilon\eta)\eta_{xxx}\bigr)_x+\frac{2\epsilon v}{1+\epsilon\eta}\partial_x\eta_t=f,
\end{equation}
with
\begin{equation}\label{reduction 2}
f\underset{\text{def}}=2\epsilon\frac{v_x^2}{1+\epsilon\eta}-\frac{2\epsilon^2}{(1+\epsilon\eta)^2}vv_x\eta_x
-\epsilon^2\Bigl(\frac{\eta_xv^2}{(1+\epsilon\eta)^2}\Bigr)_x.
\end{equation}

Applying $\partial_t$ to the second equation of \eqref{reduction form 1} one obtains
\beno\begin{aligned}
\partial_t^2v&=-(1+\epsilon\eta)\partial_x\eta_t+\epsilon(1+\epsilon\eta)\partial_x^3\eta_t-\epsilon\eta_t(\eta_x-\epsilon\eta_{xxx})
-\epsilon\Bigl(\frac{v^2}{1+\epsilon\eta}\Bigr)_{xt}\\
&=(1+\epsilon\eta)v_{xx}-\epsilon(1+\epsilon\eta)v_{xxxx}-\epsilon\eta_t(\eta_x-\epsilon\eta_{xxx})\\
&\qquad-2\epsilon v\partial_x\bigl(\frac{v_t}{1+\epsilon\eta}\Bigr)- \frac{2\epsilon v_xv_t}{1+\epsilon\eta}+\epsilon^2\Bigl(\frac{v^2\eta_t}{(1+\epsilon\eta)^2}\Bigr)_x.
\end{aligned}\eeno
Then we get
\begin{equation}\label{lte4}
\frac{1}{1+\epsilon\eta}v_{tt}-v_{xx}+\epsilon v_{xxxx}+\frac{2\epsilon v}{1+\epsilon\eta}\partial_x\bigl(\frac{v_t}{1+\epsilon\eta}\Bigr)=g,
\end{equation}
with
\begin{equation}\label{reduction 3}
g\underset{\text{def}}=-\frac{\epsilon\eta_t}{1+\epsilon\eta}(\eta_x-\epsilon\eta_{xxx})- \frac{2\epsilon v_xv_t}{(1+\epsilon\eta)^2}+\frac{\epsilon^2}{1+\epsilon\eta}\Bigl(\frac{v^2\eta_t}{(1+\epsilon\eta)^2}\Bigr)_x.
\end{equation}

Combining \eqref{lte3} and \eqref{lte4}, we obtain
\begin{equation}\label{quasilinear 1}\left\{\begin{aligned}
&\eta_{tt}-\bigl((1+\epsilon\eta)\eta_x\bigr)_x+\epsilon\bigl((1+\epsilon\eta)\eta_{xxx}\bigr)_x+\frac{2\epsilon v}{1+\epsilon\eta}\partial_x\eta_t=f,\\
&\frac{1}{1+\epsilon\eta}v_{tt}-v_{xx}+\epsilon v_{xxxx}+\frac{2\epsilon v}{1+\epsilon\eta}\partial_x\bigl(\frac{v_t}{1+\epsilon\eta}\bigr)=g,
\end{aligned}\right.\end{equation}
with $(f,g)$ being defined in \eqref{reduction 2} and \eqref{reduction 3}.

\smallskip

We remark here that \eqref{quasilinear 1} is a diagonalization of \eqref{reduction form 1} and that the principal linear part for both equations of \eqref{quasilinear 1} is the dispersive wave equation
\beno
(\partial_t^2-\partial_x^2+\epsilon\partial_x^4)\Psi.
\eeno
The source terms $(f,g)$ are of lower order. One can then derive the $L^2$ energy estimate for \eqref{quasilinear 1}.

However, if we want to derive  higher order energy estimates, it is not successful to apply $\partial_x$ to the second equation of \eqref{quasilinear 1} since when $\partial_x$ acts on  the term $\frac{1}{1+\epsilon\eta}v_{tt}$, it will turn out an uncontrolled term $-\frac{\epsilon\eta_x}{(1+\epsilon\eta)^2}v_{tt}$.
One has  to apply instead $\partial_t$ to \eqref{quasilinear 1}. In  other words, we shall improve the regularity of the unknowns by applying $\partial_t^k$ (not $\partial_x^{\alpha}$) to \eqref{reduction form 1}.

\medskip

Denoting by $\eta'=\partial_t\eta$ and $v'=\partial_t v$, applying $\partial_t$ to \eqref{quasilinear 1}, it transpires  that $(\eta',v')$ satisfies the following system
\begin{equation}\label{quasilinear 2}\left\{\begin{aligned}
&\eta'_{tt}-\bigl((1+\epsilon\eta)\eta'_x\bigr)_x+\epsilon\bigl((1+\epsilon\eta)\eta'_{xxx}\bigr)_x+\frac{2\epsilon v}{1+\epsilon\eta}\partial_x\eta'_t=f',\\
&\frac{1}{1+\epsilon\eta}v'_{tt}-v'_{xx}+\epsilon v'_{xxxx}+\frac{2\epsilon v}{1+\epsilon\eta}\partial_x\bigl(\frac{v'_t}{1+\epsilon\eta}\Bigr)=g',
\end{aligned}\right.\end{equation}
where
\begin{equation}\label{reduction 4}\begin{aligned}
&f'\underset{\text{def}}=\partial_t f+\epsilon(\eta_t\eta_x)_x-\epsilon^2(\eta_t\eta_{xxx})_x-2\epsilon\Bigl(\frac{v}{1+\epsilon\eta}\Bigr)_t\eta_{tx},\\
&g'\underset{\text{def}}=\partial_t g+\frac{\epsilon\eta_t}{(1+\epsilon\eta)^2}v_{tt}+\frac{2\epsilon^2 v}{1+\epsilon\eta}\Bigl(\frac{\eta_tv_t}{(1+\epsilon\eta)^2}\Bigr)_x-2\epsilon\Bigl(\frac{v}{1+\epsilon\eta}\Bigr)_t
\Bigl(\frac{v_t}{1+\epsilon\eta}\Bigr)_x.
\end{aligned}\end{equation}
The principal part of \eqref{quasilinear 2} is the same as that of  \eqref{quasilinear 1}.

\medskip

{\bf Step 3. Energy estimates for the quasilinear system \eqref{reduction form 1}-\eqref{quasilinear 1}-\eqref{quasilinear 2}.}
We shall derive  energy estimates for \eqref{reduction form 1}, \eqref{quasilinear 1} and \eqref{quasilinear 2}  under the assumptions
\begin{equation}\label{ansatz 1}
1+\epsilon\eta\geq H>0,
\end{equation}
and
\begin{equation}\label{ansatz}
|\eta(\cdot,t)|_{W^{1,\infty}}+|v(\cdot,t)|_{W^{1,\infty}}+|\eta(\cdot,t)_t|_{\infty}+|v(\cdot,t)_t|_{\infty}\leq c,\quad \text{for}\; t\in [0,T],
\end{equation}
where the constant $c$ is independent of $\epsilon$ but depends on the initial data. We remark that \eqref{ansatz 1} and \eqref{ansatz} are  consequences of the assumption \eqref{lte1} and the {\it a priori} estimate \eqref{total energy estimate} for $(\eta,v)$.

\smallskip

{\it Step 3.1. Estimates for \eqref{reduction form 1}.} We notice that the symmetrizer for the linear part of \eqref{reduction form 1} is the matrix
 $\text{diag}(1-\epsilon\partial_x^2,1)$. Then taking the $L^2$ inner product of  \eqref{reduction form 1} by $\bigl((1-\epsilon\partial_x^2)\eta,v\bigr)^T$ leads to
 \begin{equation}\label{lte5}
 \frac{1}{2}\frac{d}{dt}E_0(t)
 =-\frac{\epsilon}{2}\bigl(\frac{\eta_tv}{(1+\epsilon\eta)^2}\,|\,v\bigr)_2
 -\epsilon\bigl(\Bigl(\frac{v^2}{1+\epsilon\eta}\Bigr)_x\,|\,\frac{v}{1+\epsilon\eta}\bigr)_2,
 \end{equation}
  where
   \beno
   E_0(t)\underset{\text{def}}=|\eta|_2^2+\epsilon|\eta_x|_2^2+(\frac{v}{1+\epsilon\eta}\,|\,v)_2.
    \eeno
Thanks to \eqref{ansatz 1} and \eqref{ansatz}, we have
  \begin{equation}\label{lte6}
E_0(t)\sim|\eta|_2^2+\epsilon|\eta_x|_2^2+|v|_{L^2}^2 
  \end{equation}
 By  \eqref{ansatz 1}, the first term on the r.h.s of \eqref{lte5} is estimated as
 \beno
 |-\frac{\epsilon}{2}\bigl(\frac{\eta_tv}{(1+\epsilon\eta)^2}\,|\,v\bigr)_2|\lesssim\epsilon|v|_{\infty}|\eta_t|_2|v|_2
 \lesssim\epsilon|v|_{\infty}(|\eta_t|_2^2+|v|_2^2),
 \eeno
 while by integration by parts and \eqref{ansatz 1}, the second term on the r.h.s of \eqref{lte5} is estimated as
 \beno\begin{aligned}
 &|-\epsilon\bigl(\Bigl(\frac{v^2}{1+\epsilon\eta}\Bigr)_x\,|\,\frac{v}{1+\epsilon\eta}\bigr)_2|
 =\frac{\epsilon}{2}|\bigl(v_x\,|\,|\frac{v}{1+\epsilon\eta}|^2\bigr)_2|\\
 &\lesssim\epsilon|v|_{\infty}|v|_2|v_x|_2
 \lesssim\epsilon|v|_{\infty}(|v|_2^2+|v_x|_2^2).
\end{aligned}\eeno
 Then we obtain
 \begin{equation}\label{lowest estimate}
 \frac{1}{2}\frac{d}{dt}E_0(t)\lesssim\epsilon|v|_{\infty}\bigl(|\eta_t|_2^2+|v|_2^2+|v_x|_2^2\bigr).
 \end{equation}

 \smallskip

 {\it Step 3.2. Estimates for \eqref{quasilinear 1}.}  Taking the  $L^2$ scalar product of the first equation of \eqref{quasilinear 1} by $(1-\epsilon\partial_x^2)\eta_t$, we obtain
  \beno\begin{aligned}
&(\eta_{tt}\,|\,(1-\epsilon\partial_x^2)\eta_t)_2-(\bigl((1+\epsilon\eta)\eta_x\bigr)_x\,|\,(1-\epsilon\partial_x^2)\eta_t)_2\\
&\quad+\epsilon(\bigl((1+\epsilon\eta)\eta_{xxx}\bigr)_x\,|\,(1-\epsilon\partial_x^2)\eta_t)_2+(\frac{2\epsilon v}{1+\epsilon\eta}\partial_x\eta_t\,|\,(1-\epsilon\partial_x^2)\eta_t)_2=(f\,|\,(1-\epsilon\partial_x^2)\eta_t)_2.
 \end{aligned}\eeno
 Integration by parts gives
 \beno
 (\eta_{tt}\,|\,(1-\epsilon\partial_x^2)\eta_t)_2=\frac{1}{2}\frac{d}{dt}\bigl(|\eta_t|_2^2+\epsilon|\eta_{tx}|_2^2\bigr),
 \eeno
 \beno\begin{aligned}
 &-(\bigl((1+\epsilon\eta)\eta_x\bigr)_x\,|\,(1-\epsilon\partial_x^2)\eta_t)_2
 =\frac{1}{2}\frac{d}{dt}\bigl(((1+\epsilon\eta)\eta_x\,|\,\eta_x)_2+\epsilon((1+\epsilon\eta)\eta_{xx}\,|\,\eta_{xx})_2\bigr)\\
 &\qquad-\frac{\epsilon}{2}(\eta_t\eta_x\,|\,\eta_x)_2-\frac{\epsilon^2}{2}(\eta_t\eta_{xx}\,|\,\eta_{xx})_2
 -\epsilon^2(\partial_x(\eta_x\eta_x)\,|\,\eta_{tx})_2,
 \end{aligned}\eeno
 \beno\begin{aligned}
 &\epsilon(\bigl((1+\epsilon\eta)\eta_{xxx}\bigr)_x\,|\,(1-\epsilon\partial_x^2)\eta_t)_2
 =\frac{\epsilon}{2}\frac{d}{dt}\bigl(((1+\epsilon\eta)\eta_{xx}\,|\,\eta_{xx})_2+\epsilon((1+\epsilon\eta)\eta_{xxx}\,|\,\eta_{xxx})_2\bigr)\\
 &\qquad-\frac{\epsilon^2}{2}(\eta_t\eta_{xx}\,|\,\eta_{xx})_2-\frac{\epsilon^3}{2}(\eta_t\eta_{xxx}\,|\,\eta_{xxx})_2
 +\epsilon^2(\eta_x\eta_{xx}\,|\,\eta_{tx})_2.
 \end{aligned}\eeno
Then we obtain that
  \begin{equation}\label{lte7}\begin{aligned}
 &\frac{1}{2}\frac{d}{dt}E_{11}(t)+(\frac{2\epsilon v}{1+\epsilon\eta}\partial_x\eta_t\,|\,(1-\epsilon\partial_x^2)\eta_t)_2\\
 &=\frac{\epsilon}{2}(\eta_t\eta_x\,|\,\eta_x)_2+\epsilon^2(\eta_t\eta_{xx}\,|\,\eta_{xx})_2
 +\epsilon^2(\eta_{xx}\eta_x\,|\,\eta_{tx})_2\\
 &\quad+\frac{\epsilon^3}{2}(\eta_t\eta_{xxx}\,|\,\eta_{xxx})_2
 +(f\,|\,(1-\epsilon\partial_x^2)\eta_t)_2
 \end{aligned}\end{equation}
where
\beno\begin{aligned}
 E_{11}(t)&\underset{\text{def}}=|\eta_t|_2^2+\epsilon|\eta_{tx}|_2^2
 +((1+\epsilon\eta)\eta_x\,|\,\eta_x)_2+2\epsilon((1+\epsilon\eta)\eta_{xx}\,|\,\eta_{xx})_2\\
 &\quad+\epsilon^2((1+\epsilon\eta)\eta_{xxx}\,|\,\eta_{xxx})_2.
\end{aligned}\eeno

By \eqref{ansatz 1} and \eqref{ansatz}, we have
 \begin{equation}\label{lte8}
 E_{11}(t)\sim|\eta_t|_2^2+\epsilon|\eta_{tx}|_2^2+|\eta_x|_2^2+\epsilon|\eta_{xx}|_2^2+\epsilon^2|\eta_{xxx}|_2^2.
 \end{equation}

Now, we estimate the second term on the l.h.s of \eqref{lte7}. Integrating by parts, we have
\beno\begin{aligned}
&-(\frac{2\epsilon v}{1+\epsilon\eta}\partial_x\eta_t\,|\,(1-\epsilon\partial_x^2)\eta_t)_2=-\epsilon (\frac{v}{1+\epsilon\eta}\,|\,\partial_x(|\eta_t|^2)-\epsilon\partial_x(|\eta_{tx}|^2))_2\\
&
=\epsilon (\partial_x\Bigl(\frac{v}{1+\epsilon\eta}\Bigr)\,|\,|\eta_t|^2-\epsilon|\eta_{tx}|^2)_2,
\end{aligned}\eeno
which along with \eqref{ansatz 1} and \eqref{ansatz} implies that
\begin{equation}\label{lte9}
|(\frac{2\epsilon v}{1+\epsilon\eta}\partial_x\eta_t\,|\,(1-\epsilon\partial_x^2)\eta_t)_2|
\lesssim\epsilon\bigl(|\eta_x|_\infty+|v_x|_\infty\bigr)\bigl(|\eta_t|_2^2+\epsilon|\eta_{tx}|_2^2\bigr).
\end{equation}

Due to \eqref{lte7} and \eqref{lte9}, we get
 \begin{equation}\label{lte10}\begin{aligned}
& \frac{1}{2}\frac{d}{dt}E_{11}(t)\lesssim\epsilon\bigl(|\eta_t|_\infty+|\eta_x|_\infty+|v_x|_\infty\bigr)
\bigl(|\eta_x|_2^2+\epsilon|\eta_{xx}|_2^2\\
&\quad+\epsilon^2|\eta_{xxx}|_2^2+|\eta_t|_2^2+\epsilon|\eta_{tx}|_2^2\bigr)
 +|f|_2|\eta_t|_2+\epsilon|f_x|_2|\eta_{tx}|_2.
 \end{aligned}\end{equation}

Taking the $L^2$ scalar product of the second equation of \eqref{quasilinear 1} by $v_t$ yields
 \begin{equation}\label{lte11}
 \frac{1}{2}\frac{d}{dt}E_{12}(t)+(\frac{2\epsilon v}{1+\epsilon\eta}\partial_x\Bigl(\frac{v_t}{1+\epsilon\eta}\Bigr)\,|\,v_t)_2
 =-\frac{1}{2}(\partial_t\bigl(\frac{1}{1+\epsilon\eta}\bigr)v_t\,|\,v_t)_2
 +(g\,|\,v_t)_2
 \end{equation}
where
\beno
 E_{12}(t)\underset{\text{def}}=(\frac{v_t}{1+\epsilon\eta}\,|\,v_t)_2+|v_x|_2^2+\epsilon|v_{xx}|_2^2.
\eeno
Thanks to \eqref{ansatz 1}, we have
 \begin{equation}\label{lte12}
 E_{12}(t)\sim|v_t|_2^2+|v_x|_2^2+\epsilon|v_{xx}|_2^2.
 \end{equation}
Similarly  as for  \eqref{lte9}, integration by parts on the second term on the l.h.s of \eqref{lte11} leads to
\begin{equation}\label{lte13}
|(\frac{2\epsilon v}{1+\epsilon\eta}\partial_x\Bigl(\frac{v_t}{1+\epsilon\eta}\Bigr)\,|\,v_t)_2|=\epsilon|(\partial_xv\,|\,\bigl|\frac{v_t}{1+\epsilon\eta}\bigr|^2)_2|
\lesssim\epsilon|v_x|_\infty|v_t|_2^2.
\end{equation}
We can also bound the first term on the r.h.s of \eqref{lte11} as follows
\beno
|-\frac{1}{2}(\partial_t\bigl(\frac{1}{1+\epsilon\eta}\bigr)v_t\,|\,v_t)_2|\lesssim\epsilon|\eta_t|_\infty|v_t|_2^2,
\eeno
which along with \eqref{lte11} and \eqref{lte13} gives rise to
\begin{equation}\label{lte14}
\frac{1}{2}\frac{d}{dt}E_{12}(t)\lesssim\epsilon\bigl(|v_x|_\infty+|\eta_t|_\infty\bigr)|v_t|_2^2+|g|_2|v_t|_2.
\end{equation}

\smallskip

Thanks to the expressions \eqref{reduction 2} and \eqref{reduction 3}, using the assumptions \eqref{ansatz 1} and \eqref{ansatz}, we estimate the source terms $|f|_2+\epsilon^{\frac{1}{2}}|f_x|_2$ and $|g|_2$ as follows
\begin{equation}\label{lte15}\begin{aligned}
&|f|_2+\epsilon^{\frac{1}{2}}|f_x|_2+|g|_2\lesssim\epsilon\bigl(|\eta_x|_\infty+|v|_\infty+|v_x|_\infty+|\eta_t|_\infty\bigr)\\
&\quad\times\bigl(|\eta_x|_2+\epsilon^{\frac{1}{2}}|\eta_{xx}|_2+\epsilon|\eta_{xxx}|_2+|v_t|_2+|v_x|_2+\epsilon|v_{xx}|_2\bigr).
\end{aligned}\end{equation}

Now, we define $E_1(t)\underset{\text{def}}=E_{11}(t)+E_{12}(t)$. Then \eqref{lte8} and \eqref{lte12} yields
\begin{equation}\label{lte16}\begin{aligned}
E_1(t)&\sim|\eta_x|_{X^0_{\epsilon^2}}^2+|\eta_t|_{X^0_\epsilon}^2
+|v_x|_{X^0_\epsilon}^2+|v_t|_2^2.
\end{aligned}\end{equation}
where $|\cdot|_{X^s_{\epsilon^k}}^2=|\cdot|_{H^s}^2+\epsilon^k|\cdot|_{H^{s+k}}^2$.

Combining estimates \eqref{lte10}, \eqref{lte14} and \eqref{lte15},  using \eqref{lte16}, we obtain
\begin{equation}\label{first order estimate}\begin{aligned}
\frac{1}{2}&\frac{d}{dt}E_1(t)\lesssim
\epsilon\bigl(|\eta_t|_\infty+|\eta_x|_\infty+|v|_\infty+|v_x|_\infty\bigr)E_1(t), \quad t\in [0,T].
\end{aligned}\end{equation}

\smallskip

 {\it Step 3.3. Estimates for \eqref{quasilinear 2}.} Since \eqref{quasilinear 2} has the same form as \eqref{quasilinear 1}, we have a similar estimate as \eqref{lte14} for the second equation of \eqref{quasilinear 2}, that is,
\begin{equation}\label{lte17}
\frac{1}{2}\frac{d}{dt}E_{22}(t)\lesssim\epsilon\bigl(|v_x|_\infty+|\eta_t|_\infty\bigr)|v_t'|_2^2+|g'|_2|v_t'|_2,
\end{equation}
where
\begin{equation}\label{lte18}\begin{aligned}
 E_{22}(t)&\underset{\text{def}}=(\frac{v_t'}{1+\epsilon\eta}\,|\,v_t')_2+|v_x'|_2^2+\epsilon|v_{xx}'|_2^2\\
 &\sim|v_t'|_2^2+|v_x'|_2^2+\epsilon|v_{xx}'|_2^2\sim|v_t'|_2^2+|v'_x|_{X^0_{\epsilon}}^2.
\end{aligned}\end{equation}

Taking the $L^2$ scalar product of the first equation of \eqref{quasilinear 2} with $(1-\epsilon\partial_x^2)\eta_t'$,  we obtain (see the similar derivation of \eqref{lte7}) that
  \begin{equation}\label{lte19}\begin{aligned}
 &\frac{1}{2}\frac{d}{dt}E_{21}(t)+(\frac{2\epsilon v}{1+\epsilon\eta}\partial_x\eta_t'\,|\,(1-\epsilon\partial_x^2)\eta_t')_2\\
 &=\frac{\epsilon}{2}(\eta_t\eta_x'\,|\,\eta_x')_2+\epsilon^2(\eta_t\eta_{xx}'\,|\,\eta_{xx}')_2
 +\epsilon^2(\eta_{xx}\eta_x')\,|\,\eta_{tx}')_2\\
 &\quad
 +\frac{\epsilon^3}{2}(\eta_t\eta_{xxx}'\,|\,\eta_{xxx}')_2
 +(f'\,|\,(1-\epsilon\partial_x^2)\eta_t')_2
 \end{aligned}\end{equation}
where
\beno\begin{aligned}
 E_{21}(t)&\underset{\text{def}}=|\eta_t'|_2^2+\epsilon|\eta_{tx}'|_2^2
 +((1+\epsilon\eta)\eta_x'\,|\,\eta_x')_2+2\epsilon((1+\epsilon\eta)\eta_{xx}'\,|\,\eta_{xx}')_2\\
 &\quad+\epsilon^2((1+\epsilon\eta)\eta_{xxx}'\,|\,\eta_{xxx}')_2.
\end{aligned}\eeno

Thanks to \eqref{lte1} and \eqref{ansatz}, we have
 \begin{equation}\label{lte20}
 E_{21}(t)\sim|\eta_t'|_2^2+\epsilon|\eta_{tx}'|_2^2+|\eta_x'|_2^2+\epsilon|\eta_{xx}'|_2^2+\epsilon^2|\eta_{xxx}'|_2^2
 \sim|\eta_t'|_{X^0_\epsilon}^2+|\eta'_x|_{X^0_{\epsilon^2}}^2.
 \end{equation}

%

Similarly to the derivation of \eqref{lte10}, we obtain that
 \begin{equation}\label{lte21}\begin{aligned}
 \frac{1}{2}\frac{d}{dt}E_{21}(t)\lesssim&\epsilon\bigl(|\eta_t|_\infty+|\eta_x|_\infty+\epsilon^{\frac{1}{2}}|\eta_{xx}|_\infty+|v_x|_\infty\bigr)
E_{21}(t)\\
&\quad
 +|f'|_2|\eta_t'|_2+\epsilon|f_x'|_2|\eta_{tx}'|_2.
 \end{aligned}\end{equation}

 \smallskip

In order to get the final estimate on system \eqref{quasilinear 2}, we have to estimate the source terms $|f'|_2+\epsilon^{\frac{1}{2}}|f_x'|_2$
and $|g'|_2$. Thanks to the expressions of $f'$ and $g'$ in \eqref{reduction 4} and the expressions of $f$ and $g$ in \eqref{reduction 2} and \eqref{reduction 3}, using \eqref{ansatz 1} and \eqref{ansatz}, after  tedious but elementary calculations, we obtain that
\begin{equation}\label{lte22}\begin{aligned}
&|f'|_2+\epsilon^{\frac{1}{2}}|f_x'|_2+|g'|_2\lesssim\epsilon\bigl(|\eta_x|_\infty+\epsilon^{\frac{1}{2}}|\eta_{xx}|_\infty
+\epsilon|\eta_{xxx}|_\infty+|\eta_t|_\infty\\
&\quad+\epsilon^{\frac{1}{2}}|\eta_{tx}|_\infty+|v_x|_\infty+\epsilon^{\frac{1}{2}}|v_{xx}|_\infty+|v_t|_\infty\bigr)
\bigl(|\eta_{xx}|_2+\epsilon^{\frac{1}{2}}|\eta_{xxx}|_2\\
&\quad+\epsilon|\eta_{xxxx}|_2+\epsilon^{\frac{3}{2}}|\eta_{xxxxx}|_2+|\eta_t|_2+|\eta_{tx}|_2+\epsilon^{\frac{1}{2}}|\eta_{txx}|_2\\
&\quad
 +\epsilon|\eta_{txxx}|_2+|v_{xx}|_2+|v_t|_2+|v_{tx}|_2+\epsilon^{\frac{1}{2}}|v_{txx}|_2+|v_{tt}|_2\bigr),
\end{aligned}\end{equation}
where we used $\eta'=\eta_t$ and $v'=v_t$.

Now, we define $E_2(t)\underset{\text{def}}=E_{21}(t)+E_{22}(t)$. Then \eqref{lte18} and \eqref{lte20} yields
\begin{equation}\label{lte23}
E_2(t)\sim|\eta_t'|_{X^0_\epsilon}^2+|\eta'_x|_{X^0_{\epsilon^2}}^2+|v_t'|_2^2+|v'_x|_{X^0_{\epsilon}}^2.
\end{equation}

 Thanks to \eqref{lte17}, \eqref{lte21} and \eqref{lte22}, using the interpolation inequality \eqref{int1}, we obtain that
\begin{equation}\label{top order estimate}\begin{aligned}
 &\frac{1}{2}\frac{d}{dt}E_2(t)\lesssim\epsilon\bigl(|\eta_x|_\infty+\epsilon^{\frac{1}{2}}|\eta_{xx}|_\infty
+\epsilon|\eta_{xxx}|_\infty+|\eta_t|_\infty\\
&\quad+\epsilon^{\frac{1}{2}}|\eta_{tx}|_\infty+|v_x|_\infty+\epsilon^{\frac{1}{2}}|v_{xx}|_\infty+|v_t|_\infty\bigr)\\
&\quad\times\bigl(|\eta|_{X^2_{\epsilon^3}}^2+|\eta_t|_{X^1_{\epsilon^2}}^2+|\eta_{tt}|_{X^0_{\epsilon}}^2
+|v|_{H^2}^2+|v_t|_{X^1_{\epsilon}}^2+|v_{tt}|_2^2\bigr),
\end{aligned}\end{equation}
where we replaced $\eta',\ v'$ by $\eta_t, \ v_t$ respectively in the bound.

{\bf Step 4. The final estimate  on \eqref{reduction form 1}.} Before closing the {\it a priori estimates}, we first define the energy functional associated to the quasilinear system \eqref{reduction form 1}-\eqref{quasilinear 1}-\eqref{quasilinear 2} as
\begin{equation}\label{functional 1}
E(t)\underset{\text{def}}=E_0(t)+E_1(t)+E_2(t).
\end{equation}
Notice that $\eta'=\eta_t$ and $v'=v_t$. Then \eqref{lte6}, \eqref{lte16} and \eqref{lte23} yield that
\begin{equation}\label{functional 2}\begin{aligned}
E(t)\sim|\eta|_{X^1_{\epsilon^2}}^2+|\eta_t|_{X^1_{\epsilon^2}}^2+|\eta_{tt}|_{X^0_{\epsilon}}^2
+|v|_{X^1_{\epsilon}}^2+|v_t|_{X^1_{\epsilon}}^2+|v_{tt}|_2^2.
\end{aligned}\end{equation}

In order to close the energy estimate, we also need to define the total energy functional for \eqref{reduction form 1} as follows:
\begin{equation}\label{functional 3}\begin{aligned}
\mathcal{E}(t)\underset{\text{def}}=|\eta|_{X^2_{\epsilon^3}}^2+|\eta_t|_{X^1_{\epsilon^2}}^2+|\eta_{tt}|_{X^0_{\epsilon}}^2
+|v|_{X^2_{\epsilon^2}}^2+|v_t|_{X^1_{\epsilon}}^2+|v_{tt}|_2^2.
\end{aligned}\end{equation}

With the definitions \eqref{functional 1} and \eqref{functional 3}, using the interpolation inequality \eqref{int1} and the Sobolev inequality $|\cdot|_{L^\infty(\R)}\lesssim|\cdot|_{H^1(\R)}$, the energy estimates \eqref{lowest estimate}, \eqref{first order estimate} and \eqref{top order estimate} give rise to
\begin{equation}\label{energy estimate 1}
\frac{1}{2}\frac{d}{dt}E(t)\lesssim\epsilon\mathcal{E}(t)^{\frac{3}{2}}.
\end{equation}

To finish the proof, we have to show that
\begin{equation}\label{sim}
\mathcal{E}(t)\sim E(t).
\end{equation}
Indeed, thanks to \eqref{functional 2} and \eqref{functional 3}, we have
\beno
\mathcal{E}(t)\sim E(t)+|\eta|_{X^2_{\epsilon^3}}^2+|v|_{X^2_{\epsilon^2}}^2.
\eeno
Then we only need to show that
\beno
|\eta|_{X^2_{\epsilon^3}}^2+|v|_{X^2_{\epsilon^2}}^2\lesssim E(t).
\eeno
That is to say, we shall recover the regularity in space through the  regularity in time.
More precisely, \eqref{reduction form 1} yields
\begin{equation}\label{transfer}
v_x=-\eta_t,\quad (1-\epsilon\partial_x^2)\eta_x=-\frac{v_t}{1+\epsilon\eta}-\frac{\epsilon}{1+\epsilon\eta}\Bigl(\frac{v^2}{1+\epsilon\eta}\Bigr)_x.
\end{equation}

The first equation of \eqref{transfer} shows
\begin{equation}\label{improve 1}
|v|_{X_{\epsilon^2}^2}^2=|v|_{H^2}^2+\epsilon^2|v|_{H^{2+2}}^2\lesssim |v|_{H^1}^2+|\eta_{tx}|_2^2+\epsilon^2|\eta_{txxx}|_2^2
\lesssim E(t),
\end{equation}
where we used \eqref{ansatz}.
 While the second equation of \eqref{transfer}, \eqref{ansatz} and \eqref{improve 1} imply
\beno\begin{aligned}
&|\eta|_{X_{\epsilon^3}^2}^2\sim|\eta|_{H^1}^2+|(1-\epsilon\partial_x^2)\eta_{xx}|_2^2+\epsilon|(1-\epsilon\partial_x^2)\eta_{xxx}|_2^2\\
&\lesssim |\eta|_{X^1_{\epsilon^2}}^2+|v|_{X^1_{\epsilon^2}}^2+|v_t|_{X^1_{\epsilon}}^2
\lesssim E(t),
\end{aligned}
\eeno
which achieves the proof of  \eqref{sim}. Due to \eqref{sim} and \eqref{energy estimate 1}, we have
\begin{equation}\label{energy estimate 2}
\frac{1}{2}\frac{d}{dt}E(t)\lesssim\epsilon E(t)^{\frac{3}{2}}.
\end{equation}

{\bf Step 5. Initial data for the quasilinear system and final estimate.}
In this step, we have to derive the regularity of the initial data to the quasilinear system through the system \eqref{reduction form 1} and the regularity of the initial data  $(\eta_0,v_0)$. The first equation of \eqref{reduction form 1} shows that
\beno
|\eta'|_{t=0}|_{X^1_{\epsilon^2}}=|\eta_t|_{t=0}|_{X^1_{\epsilon^2}}=|\partial_x v_0|_{X^1_{\epsilon^2}}\lesssim|v_0|_{X^2_{\epsilon^2}},
\eeno
while the second equation of \eqref{reduction form 1} shows that
\beno\begin{aligned}
&\quad|v'|_{t=0}|_{X^1_{\epsilon}}=|v_t|_{t=0}|_{X^1_{\epsilon}}\\
&\lesssim|(1+\epsilon\eta_0)(1-\epsilon\partial_x^2)\partial_x\eta_0|_{X^1_{\epsilon}}
+\epsilon|\Bigl(\frac{v_0^2}{1+\epsilon\eta_0}\Bigr)_x|_{X^1_{\epsilon}}\\
&\lesssim|\eta_0|_{X^2_{\epsilon^3}}+|v_0|_{X^2_{\epsilon^2}},
\end{aligned}\eeno
where we assume that $|\eta_0|_{X^2_{\epsilon^3}}+|v_0|_{X^2_{\epsilon^2}}\leq C$ and $\epsilon\leq\epsilon_0$ with $\epsilon_0$ small enough.

Thanks to \eqref{quasilinear 2}, we can also infer that
\beno\begin{aligned}
&|\eta'_t|_{t=0}|_{X^0_{\epsilon}}+|v'_t|_{t=0}|_2=|\eta_{tt}|_{t=0}|_{X^0_{\epsilon}}+|v_{tt}|_{t=0}|_2\\
&\lesssim|\eta_0|_{X^2_{\epsilon^3}}+|v_0|_{X^2_{\epsilon^2}}
\end{aligned}\eeno
provided that $|\eta_0|_{X^2_{\epsilon^3}}+|v_0|_{X^2_{\epsilon^2}}\leq C$ and $\epsilon\leq\epsilon_0$ with $\epsilon_0$ small enough.

Thus, we have
\begin{equation}\label{regularity for initial data}
E(0)\sim\mathcal{E}(0)\lesssim|\eta_0|_{X^2_{\epsilon^3}}^2+|v_0|_{X^2_{\epsilon^2}}^2.
\end{equation}

\medskip

{\bf Step 6. Existence and uniqueness.} The estimates  \eqref{energy estimate 2} and \eqref{regularity for initial data} are crucial to prove  the existence of  $T>0$ independent of $\epsilon$ such that \eqref{reduction form 1} has a unique solution $(\eta,v)$ on a  time interval $[0,T/\epsilon]$ with initial data $(\eta_0,v_0)\in X^2_{\epsilon^3}\times X^2_{\epsilon^2}$  satisfying moreover by \eqref{energy estimate 2} and \eqref{sim} the estimate
\begin{equation}\label{total energy estimate}
\sup_{t\in [0,T/\epsilon]}\mathcal{E}(t)\lesssim|\eta_0|_{X^2_{\epsilon^3}}^2+|v_0|_{X^2_{\epsilon^2}}^2.
\end{equation}

We shall precise the existence proof in the following Section 5.

\medskip

Notice that $v=(1+\epsilon\eta)u$. Then we have  obtained the long time estimate of solutions to  the original Boussinesq system \eqref{abcd}-\eqref{rme2} with $a=b=d=0, c=-1$ together to  the energy estimate \eqref{lte2}.
\end{proof}

Now we state the long time existence result for the two-dimensional case.
\begin{theorem}\label{long time existence for case c=-1 2D}
Let $a=b=d=0, c=-1$.
Assume that $\eta_0\in X^3_{\epsilon^4}(\R^2),\vv u_0\in X^3_{\epsilon^3}(\R^2)$  satisfy the curl free condition $\curl\vv u_0=0$ and the (non-cavitation) condition
\begin{equation}\label{lte1a}
1+\epsilon\eta_0\geq H>0,\quad H\in(0,1).
\end{equation}
Then there exists a constant $\tilde{c}_0$  such that for any $\epsilon\leq\epsilon_0=\frac{1}{\tilde{c}_0(| \eta_0|_{X^3_{\epsilon^4}}+| u_0|_{X^3_{\epsilon^3}})}$, there
exists $T>0$ independent of $\epsilon$, such that \eqref{abcd}-\eqref{rme2} has a unique solution $(\eta,\vv u)^T$ with
$\eta\in C([0,T/\epsilon];X^3_{\epsilon^4}(\R^2))$ and $\vv u\in C([0,T/\epsilon];X^3_{\epsilon^3}(\R^2))$. Moreover,
\begin{equation}\label{lte2a}
\begin{aligned}
&\sup_{t\in[0,T/\epsilon]}\bigl(|\eta|_{X^3_{\epsilon^4}}^2+|\eta_t|_{X^2_{\epsilon^3}}^2+|\eta_{tt}|_{X^1_{\epsilon^2}}^2
+|\eta_{ttt}|_{X^0_{\epsilon}}^2
+|\vv u|_{X^3_{\epsilon^3}}^2\\
&\qquad+|\vv u_t|_{X^2_{\epsilon^2}}^2+|\vv u_{tt}|_{X^1_{\epsilon}}^2+|\vv u_{ttt}|_2^2\bigr)
\leq C\bigl(|\eta_0|_{X^3_{\epsilon^4}}^2+|\vv u_0|_{X^3_{\epsilon^3}}^2\bigr).
\end{aligned}\end{equation}
\end{theorem}

\begin{remark}
By simplicity, we assume that $\curl \vv u_0=0$. Actually, the equation of $\vv u$ shows that $\partial_t\curl\vv u(t,\cdot)=0$ so that
$\curl\vv u$ is preserved as time evolves.  In fact, as pointed out to us by Vincent Duch\^{e}ne, considering the term $ \nabla(|{\bf u}|^2)$ is not physically relevant outside the irrotational case. When $\curl {\bf u}\neq 0,$ one should use instead the term  ${\bf u}\cdot \nabla {\bf u},$ but then the corresponding system is to our knowledge not rigorously justified (see \cite{CL} for Green-Naghdi type systems).
\end{remark}

\begin{proof}
Since the proof is similar to that of Theorem \ref{long time existence for case c=-1}, we only sketch  it. We also divide the proof into several steps. Again we only indicate how to obtain the {\it a priori} estimates. The existence proof which is similar to the one-dimensional case is postponed to the following Section 5.

\smallskip

{\bf Step 1. Reduction of the system.}
Since $a=b=d=0, c=-1$, we first rewrite  \eqref{abcd} in the form \eqref{2Dbis}. Setting
\beno
\vv v\underset{\text{def}}=(1+\epsilon\eta)\vv u,
\eeno
the first equation of \eqref{2Dbis} becomes
\beno
\eta_t+\nabla\cdot\vv v=0.
\eeno
To get the evolution equation for $\vv v$, we first get from the second equation of \eqref{2Dbis} that
\beno
\partial_t\curl\vv u=0,
\eeno
which along with the assumption that $\curl\vv u_0=0$ implies that $\curl\vv u=0$. Then  $\nabla(|\vv u|^2)=2\vv u\cdot\nabla\vv u$ and the second equation of \eqref{2Dbis} becomes to
\beno
\partial_t\vv u+\nabla\eta-\epsilon\nabla\Delta\eta+\epsilon\vv u\cdot\nabla\vv u=\vv 0.
\eeno
Similarly as one-dimensional case, elementary calculations and the use of the above equation yield the evolution equation for $\vv v :$
\beno
\vv v_t+(1+\epsilon\eta)\nabla\eta-\epsilon(1+\epsilon\eta)\nabla\Delta\eta+\epsilon\nabla\cdot\Bigl(\frac{\vv v}{1+\epsilon\eta}\otimes\vv v\Bigr)=\vv 0,
\eeno
where $\Bigl(\nabla\cdot(\vv u\otimes\vv v)\Bigr)^i\underset{\text{def}}=\partial_j(u^iv^j)$.
Then \eqref{2Dbis} is rewritten in terms of $(\eta,\vv v)$ as follows
\begin{equation}\label{reduction form 1 for 2D}\left\{\begin{aligned}
&\eta_t+\nabla\cdot\vv v=0,\\
&\frac{1}{1+\epsilon\eta}\vv v_t+\nabla\eta-\epsilon\nabla\Delta\eta+\frac{\epsilon}{1+\epsilon\eta}\nabla\cdot\Bigl(\frac{\vv v}{1+\epsilon\eta}\otimes\vv v\Bigr)=\vv 0.
\end{aligned}\right.\end{equation}

We shall derive  energy estimates for this system.

\medskip

{\bf Step 2. Quasilinearization of \eqref{reduction form 1 for 2D}.} In this step, we shall  quasilinearize the system \eqref{reduction form 1 for 2D}
by applying to it  $\partial_t$, $\partial_t^2$ and $\partial_t^3$. Applying $\partial_t$ to the first equation of \eqref{reduction form 1 for 2D} leads to
\beno\begin{aligned}
\partial_t^2\eta&=-\nabla\cdot\vv v_t=\nabla\cdot\bigl((1+\epsilon\eta)\nabla\eta\bigr)-\epsilon\nabla\cdot\bigl((1+\epsilon\eta)\nabla\Delta\eta\bigr)
+\epsilon\nabla\cdot\Bigl[\nabla\cdot\Bigl(\frac{\vv v}{1+\epsilon\eta}\otimes\vv v\Bigr)\Bigr]\\
&=\nabla\cdot\bigl((1+\epsilon\eta)\nabla\eta\bigr)-\epsilon\nabla\cdot\bigl((1+\epsilon\eta)\nabla\Delta\eta\bigr)+\frac{2\epsilon\vv v}{1+\epsilon\eta}\cdot\nabla (\nabla\cdot\vv v)+\frac{\epsilon|\nabla\cdot\vv v|^2}{1+\epsilon\eta}\\
&\qquad+2\epsilon\vv v\cdot\nabla\bigl(\frac{1}{1+\epsilon\eta}\bigr)(\nabla\cdot\vv v) +\epsilon\vv v\cdot\nabla\bigl[\vv v\cdot\nabla\bigl(\frac{1}{1+\epsilon\eta}\bigr)\bigr]+\sum_{i,j=1,2}\partial_j\bigl(\frac{v^i}{1+\epsilon\eta}\bigr)\partial_iv^j.
\end{aligned}\eeno
One notices that the third term $\frac{2\epsilon\vv v}{1+\epsilon\eta}\cdot\nabla (\nabla\cdot\vv v)$ in the second line of the above equality is the higher order term. Since by \eqref{reduction form 1 for 2D} $\nabla\cdot\vv v=-\eta_t$, we rewrite this term as
\beno
\frac{2\epsilon\vv v}{1+\epsilon\eta}\cdot\nabla (\nabla\cdot\vv v)=-\frac{2\epsilon\vv v}{1+\epsilon\eta}\cdot\nabla\eta_t.
\eeno
Then we obtain
\begin{equation}\label{lte3a}
\eta_{tt}-\nabla\cdot\bigl((1+\epsilon\eta)\nabla\eta\bigr)+\epsilon\nabla\cdot\bigl((1+\epsilon\eta)\nabla\Delta\eta\bigr)
+\frac{2\epsilon\vv v}{1+\epsilon\eta}\cdot\nabla\eta_t=f,
\end{equation}
with
\begin{equation}\label{reduction 2a}\begin{aligned}
f\underset{\text{def}}=&\frac{\epsilon|\nabla\cdot\vv v|^2}{1+\epsilon\eta}+2\epsilon\vv v\cdot\nabla\bigl(\frac{1}{1+\epsilon\eta}\bigr)(\nabla\cdot\vv v)\\
&\quad+\epsilon\vv v\cdot\nabla\bigl[\vv v\cdot\nabla\bigl(\frac{1}{1+\epsilon\eta}\bigr)\bigr]+\epsilon\sum_{i,j=1,2}\partial_j\bigl(\frac{v^i}{1+\epsilon\eta}\bigr)\partial_iv^j.
\end{aligned}\end{equation}

Applying $\partial_t$ to the second equation of \eqref{reduction form 1 for 2D} one obtains
\beno\begin{aligned}
\partial_t^2\vv v&=-(1+\epsilon\eta)\nabla\eta_t+\epsilon(1+\epsilon\eta)\nabla\Delta\eta_t-\epsilon\eta_t(\nabla\eta-\epsilon\nabla\Delta\eta)
-\epsilon\nabla\cdot\Bigl(\frac{\vv v}{1+\epsilon\eta}\otimes\vv v\Bigr)_t.
\end{aligned}\eeno

Notice that $\eta_t=-\nabla\cdot\vv v$,
we obtain the reduced equation for $\vv v$
\begin{equation}\label{lte4a}\begin{aligned}
&\frac{1}{1+\epsilon\eta}\vv v_{tt}-\nabla(\nabla\cdot\vv v)+\epsilon \nabla\Delta(\nabla\cdot\vv v)+(\frac{\epsilon\vv v}{1+\epsilon\eta}\cdot\nabla)\Bigl(\frac{\vv v_t}{1+\epsilon\eta}\Bigr)\\
&\qquad+\frac{\epsilon\vv v}{(1+\epsilon\eta)^2}(\nabla\cdot\vv v_t)=\vv g,
\end{aligned}\end{equation}
with
\begin{equation}\label{reduction 3a}\begin{aligned}
\vv g\underset{\text{def}}=&-\frac{\epsilon\eta_t}{1+\epsilon\eta}(\nabla\eta-\epsilon\nabla\Delta\eta)+(\frac{\epsilon^2\vv v}{1+\epsilon\eta}\cdot\nabla)\Bigl(\frac{\vv v\eta_t}{(1+\epsilon\eta)^2}\Bigr)\\
&
-\frac{\epsilon}{1+\epsilon\eta}\partial_t\Bigl(\frac{\vv v}{1+\epsilon\eta}\Bigr)(\nabla\cdot\vv v).
\end{aligned}\end{equation}

Combining \eqref{lte3a} and \eqref{lte4a}, we obtain
\begin{equation}\label{quasilinear 1a}\left\{\begin{aligned}
&\eta_{tt}-\nabla\cdot\bigl((1+\epsilon\eta)\nabla\eta\bigr)+\epsilon\nabla\cdot\bigl((1+\epsilon\eta)\nabla\Delta\eta\bigr)
+\frac{2\epsilon\vv v}{1+\epsilon\eta}\cdot\nabla\eta_t=f,\\
&\frac{1}{1+\epsilon\eta}\vv v_{tt}-\nabla(\nabla\cdot\vv v)+\epsilon \nabla\Delta(\nabla\cdot\vv v)\\
&\qquad+(\frac{\epsilon\vv v}{1+\epsilon\eta}\cdot\nabla)\Bigl(\frac{\vv v_t}{1+\epsilon\eta}\Bigr)+\frac{\epsilon\vv v}{(1+\epsilon\eta)^2}(\nabla\cdot\vv v_t)=\vv g,
\end{aligned}\right.\end{equation}
with $(f,\vv g)$ being defined in \eqref{reduction 2a} and \eqref{reduction 3a}.

\smallskip

We  also remark here that \eqref{quasilinear 1a} is a diagonalization of \eqref{reduction form 1 for 2D} and that the principal linear part for both equations of \eqref{quasilinear 1a} is the dispersive wave
\beno
(\partial_t^2-\Delta+\epsilon\Delta^2)\Psi.
\eeno
The source terms $(f,\vv g)$ are of lower order. One can then derive the $L^2$ energy estimate for \eqref{quasilinear 1a}.

\medskip

Denoting by $\eta'=\partial_t\eta$ and $\vv v'=\partial_t\vv v$, applying $\partial_t$ to \eqref{quasilinear 1a}, it transpires  that $(\eta',\vv v')$ satisfies the following system
\begin{equation}\label{quasilinear 2a}\left\{\begin{aligned}
&\eta_{tt}'-\nabla\cdot\bigl((1+\epsilon\eta)\nabla\eta'\bigr)+\epsilon\nabla\cdot\bigl((1+\epsilon\eta)\nabla\Delta\eta'\bigr)
+\frac{2\epsilon\vv v}{1+\epsilon\eta}\cdot\nabla\eta_t'=f',\\
&\frac{1}{1+\epsilon\eta}\vv v_{tt}'-\nabla(\nabla\cdot\vv v')+\epsilon \nabla\Delta(\nabla\cdot\vv v')\\
&\qquad+(\frac{\epsilon\vv v}{1+\epsilon\eta}\cdot\nabla)\Bigl(\frac{\vv v_t'}{1+\epsilon\eta}\Bigr)+\frac{\epsilon\vv v}{(1+\epsilon\eta)^2}(\nabla\cdot\vv v_t')=\vv g',
\end{aligned}\right.\end{equation}
where
\begin{equation}\label{reduction 4a}\begin{aligned}
&f'\underset{\text{def}}=\partial_t f+\epsilon\nabla\cdot(\eta_t\nabla\eta)-\epsilon^2\nabla\cdot(\eta_t\nabla\Delta\eta)-2\epsilon\Bigl(\frac{\vv v}{1+\epsilon\eta}\Bigr)_t\cdot\nabla\eta_t\\
&\vv g'\underset{\text{def}}=\partial_t\vv g+\frac{\epsilon\eta_t}{(1+\epsilon\eta)^2}\vv v'_t-
\epsilon(\partial_t\Bigl(\frac{\vv v}{1+\epsilon\eta}\Bigr)\cdot\nabla)\Bigl(\frac{\vv v_t}{1+\epsilon\eta}\Bigr)\\
&\qquad+\epsilon^2\Bigl(\frac{\vv v}{1+\epsilon\eta}\cdot\nabla\Bigr)\Bigl(\frac{\eta_t\vv v_t}{(1+\epsilon\eta)^2}\Bigr)-\epsilon\partial_t\Bigl(\frac{\vv v}{(1+\epsilon\eta)^2}\Bigr)(\nabla\cdot\vv v_t).
\end{aligned}\end{equation}
The principal part of \eqref{quasilinear 2a} is the same as that of  \eqref{quasilinear 1a}.

Similarly, denoting by $\eta''=\partial_t^2\eta=\partial_t\eta'$ and $\vv v''=\partial_t^2\vv v=\partial_t\vv v'$, applying $\partial_t$ to \eqref{quasilinear 2a}, it transpires  that $(\eta'',\vv v'')$ satisfies the following system
\begin{equation}\label{quasilinear 2aa}\left\{\begin{aligned}
&\eta_{tt}''-\nabla\cdot\bigl((1+\epsilon\eta)\nabla\eta''\bigr)+\epsilon\nabla\cdot\bigl((1+\epsilon\eta)\nabla\Delta\eta''\bigr)
+\frac{2\epsilon\vv v}{1+\epsilon\eta}\cdot\nabla\eta_t''=f'',\\
&\frac{1}{1+\epsilon\eta}\vv v_{tt}''-\nabla(\nabla\cdot\vv v'')+\epsilon \nabla\Delta(\nabla\cdot\vv v'')\\
&\qquad+\epsilon(\frac{\vv v}{1+\epsilon\eta}\cdot\nabla)\Bigl(\frac{\vv v_t''}{1+\epsilon\eta}\Bigr)+\frac{\epsilon\vv v}{(1+\epsilon\eta)^2}(\nabla\cdot\vv v_t'')=\vv g'',
\end{aligned}\right.\end{equation}
where
\begin{equation}\label{reduction 4aa}\begin{aligned}
&f''\underset{\text{def}}=\partial_t f'+\epsilon\nabla\cdot(\eta_t\nabla\eta')-\epsilon^2\nabla\cdot(\eta_t\nabla\Delta\eta')-2\epsilon\Bigl(\frac{\vv v}{1+\epsilon\eta}\Bigr)_t\cdot\nabla\eta_t'\\
&\vv g''\underset{\text{def}}=\partial_t\vv g'+\frac{\epsilon\eta_t}{(1+\epsilon\eta)^2}\vv v'_{tt}-
\epsilon(\partial_t\Bigl(\frac{\vv v}{1+\epsilon\eta}\Bigr)\cdot\nabla)\Bigl(\frac{\vv v_t'}{1+\epsilon\eta}\Bigr)\\
&\qquad+\epsilon^2\Bigl(\frac{\vv v}{1+\epsilon\eta}\cdot\nabla\Bigr)\Bigl(\frac{\eta_t\vv v'_t}{(1+\epsilon\eta)^2}\Bigr)-\epsilon\partial_t\Bigl(\frac{\vv v}{(1+\epsilon\eta)^2}\Bigr)(\nabla\cdot\vv v_t').
\end{aligned}\end{equation}
The principal part of \eqref{quasilinear 2aa} is also the same as that of  \eqref{quasilinear 1a}.

\medskip

{\bf Step 3. Energy estimates for the quasilinear system \eqref{reduction form 1 for 2D}-\eqref{quasilinear 1a}-\eqref{quasilinear 2a}-\eqref{quasilinear 2aa}.}
We shall derive  energy estimates for \eqref{reduction form 1 for 2D}, \eqref{quasilinear 1a}, \eqref{quasilinear 2a}, \eqref{quasilinear 2aa}  under the assumptions
\begin{equation}\label{ansatz a1}
1+\epsilon\eta\geq H>0,
\end{equation}
 and for all $t\in [0,T]$,
\begin{equation}\label{ansatz a}
|\eta(\cdot,t)|_{W^{1,\infty}}+|\vv v(\cdot,t)|_{W^{1,\infty}}+|\eta(\cdot,t)_t|_{\infty}+|\vv v(\cdot,t)_t|_{\infty}+|\eta(\cdot,t)|_{X^3_{\epsilon^4}}\leq c.
\end{equation}
where $|\cdot|_{X^s_{\epsilon^k}}^2=|\cdot|_{H^s}^2+\epsilon^k|\cdot|_{H^{s+k}}^2$ and the constant $c$ independent of $\epsilon$ depends on the initial data. We remark that \eqref{ansatz a1} and \eqref{ansatz a} are consequence of the assumption \eqref{lte1a} and the {\it a priori} estimate \eqref{total energy estimate for 2D} for $(\eta,\vv v)$.

\smallskip

{\it Step 3.1. Estimates for \eqref{reduction form 1 for 2D}.} Similarly as \eqref{reduction form 1}, taking the $L^2$ inner product of  \eqref{reduction form 1 for 2D} by $\bigl((1-\epsilon\Delta)\eta,\vv v\bigr)^T$ leads to
 \begin{equation}\label{lte5a}
 \frac{1}{2}\frac{d}{dt}E_0(t)
 =-\frac{\epsilon}{2}\bigl(\frac{\eta_t\vv v}{(1+\epsilon\eta)^2}\,|\,\vv v\bigr)_2
 -\epsilon\bigl(\nabla\cdot\Bigl(\frac{\vv v}{1+\epsilon\eta}\otimes\vv v\Bigr)\,|\,\frac{\vv v}{1+\epsilon\eta}\bigr)_2,
 \end{equation}
  where
   \beno
   E_0(t)\underset{\text{def}}=|\eta|_2^2+\epsilon|\nabla\eta|_2^2+(\frac{\vv v}{1+\epsilon\eta}\,|\,\vv v)_2.
    \eeno
Thanks to \eqref{ansatz a1} and \eqref{ansatz a}, we have
  \begin{equation}\label{lte6a}
E_0(t)\sim|\eta|_2^2+\epsilon|\nabla\eta|_2^2+|\vv v|_{L^2}^2.
  \end{equation}
 By \eqref{ansatz a1}, the first term on the r.h.s of \eqref{lte5a} is estimated as
 \beno
 |-\frac{\epsilon}{2}\bigl(\frac{\eta_t\vv v}{(1+\epsilon\eta)^2}\,|\,\vv v\bigr)_2|\lesssim\epsilon|\vv v|_{\infty}|\eta_t|_2|\vv v|_2
 \lesssim\epsilon|\vv v|_{\infty}(|\eta_t|_2^2+|\vv v|_2^2),
 \eeno
 while by integration by parts and \eqref{ansatz a1}, the second term on the r.h.s of \eqref{lte5a} is estimated as
 \beno\begin{aligned}
 &| -\epsilon\bigl(\nabla\cdot\Bigl(\frac{\vv v}{1+\epsilon\eta}\otimes\vv v\Bigr)\,|\,\frac{\vv v}{1+\epsilon\eta}\bigr)_2|
 =\frac{\epsilon}{2}|\bigl(\nabla\cdot\vv v\,|\,|\frac{\vv v}{1+\epsilon\eta}|^2\bigr)_2|\\
 &\lesssim\epsilon|\vv v|_{\infty}|\vv v|_2|\nabla\vv v|_2
 \lesssim\epsilon|\vv v|_{\infty}(|\vv v|_2^2+|\nabla\vv v|_2^2).
\end{aligned}\eeno
 Then we obtain
 \begin{equation}\label{lowest estimate for 2D}
 \frac{1}{2}\frac{d}{dt}E_0(t)\lesssim\epsilon|\vv v|_{\infty}\bigl(|\eta_t|_2^2+|\vv v|_2^2+|\nabla\vv v|_2^2\bigr).
 \end{equation}

 \smallskip

 {\it Step 3.2. Estimates for \eqref{quasilinear 2aa}.}  Taking the  $L^2$ scalar product of the first equation of \eqref{quasilinear 2aa} with $(1-\epsilon\Delta)\eta_t''$, we obtain
  \beno\begin{aligned}
&(\eta_{tt}''\,|\,(1-\epsilon\Delta)\eta_t'')_2-(\nabla\cdot\bigl((1+\epsilon\eta)\nabla\eta''\bigr)\,|\,(1-\epsilon\Delta)\eta_t'')_2\\
&\quad+\epsilon(\nabla\cdot\bigl((1+\epsilon\eta)\nabla\Delta\eta''\bigr)\,|\,(1-\epsilon\Delta)\eta_t'')_2+(\frac{2\epsilon \vv v}{1+\epsilon\eta}\cdot\nabla\eta_t''\,|\,(1-\epsilon\Delta)\eta_t'')_2\\
&=(f''\,|\,(1-\epsilon\Delta)\eta_t'')_2.
 \end{aligned}\eeno
Similar to the derivation of \eqref{lte7}, using integration by parts, we obtain
  \begin{equation}\label{lte7a}\begin{aligned}
 &\frac{1}{2}\frac{d}{dt}E_{31}(t)+(\frac{2\epsilon \vv v}{1+\epsilon\eta}\cdot\nabla\eta_t''\,|\,(1-\epsilon\Delta)\eta_t'')_2\\
 &=\frac{\epsilon}{2}(\eta_t\nabla\eta''\,|\,\nabla\eta'')_2+\epsilon^2(\eta_t\Delta\eta''\,|\,\Delta\eta'')_2
 +\epsilon^2(\nabla(\nabla\eta\cdot\nabla\eta'')\,|\,\nabla\eta_t'')_2\\
 &\quad-\epsilon^2(\nabla\eta\Delta\eta''\,|\,\nabla\eta_t'')_2+\frac{\epsilon^3}{2}(\eta_t\nabla\Delta\eta''\,|\,\nabla\Delta\eta'')_2
 +(f''\,|\,(1-\epsilon\Delta)\eta_t'')_2
 \end{aligned}\end{equation}
where
\beno\begin{aligned}
 E_{31}(t)&\underset{\text{def}}=|\eta_t''|_2^2+\epsilon|\nabla\eta_t''|_2^2
 +((1+\epsilon\eta)\nabla\eta''\,|\,\nabla\eta'')_2+2\epsilon((1+\epsilon\eta)\Delta\eta''\,|\,\Delta\eta'')_2\\
 &\quad+\epsilon^2((1+\epsilon\eta)\nabla\Delta\eta''\,|\,\nabla\Delta\eta'')_2.
\end{aligned}\eeno

By \eqref{ansatz a1} and \eqref{ansatz a}, we have
 \begin{equation}\label{lte8a}
 E_{31}(t)\sim|\eta_t''|_2^2+\epsilon|\nabla\eta_t''|_2^2+|\nabla\eta''|_2^2+\epsilon|\nabla^2\eta''|_2^2+\epsilon^2|\nabla^3\eta''|_2^2.
 \end{equation}

Now, we estimate the second term on the l.h.s of \eqref{lte7a}. Integrating by parts, we have
\begin{equation}\label{lte9a}\begin{aligned}
&(\frac{2\epsilon \vv v}{1+\epsilon\eta}\cdot\nabla\eta_t''\,|\,(1-\epsilon\Delta)\eta_t'')_2=-\epsilon (\nabla\cdot\Bigl(\frac{\vv v}{1+\epsilon\eta}\Bigr)\eta_t''\,|\,\eta_t'')_2\\
&\quad+2\epsilon^2(\bigl(\nabla\eta_t''\cdot\nabla\bigr)\Bigl(\frac{\vv v}{1+\epsilon\eta}\Bigr)\,|\,\nabla\eta_t'')_2-\epsilon^2(\nabla\cdot\Bigl(\frac{\vv v}{1+\epsilon\eta}\Bigr)\nabla\eta_t''\,|\,\nabla\eta_t'')_2,
\end{aligned}\end{equation}
which along with \eqref{ansatz a1}, \eqref{ansatz a}
and \eqref{lte7a} implies that
 \begin{equation}\label{lte10a}\begin{aligned}
& \frac{1}{2}\frac{d}{dt}E_{31}(t)\lesssim\epsilon\bigl(|\eta_t|_\infty+|\nabla\eta|_\infty+|\nabla\vv v|_\infty+\epsilon^{\frac{1}{2}}|\nabla^2\eta|_\infty\bigr)
\bigl(|\nabla\eta''|_2^2+\epsilon|\nabla^2\eta''|_2^2\\
&\quad+\epsilon^2|\nabla^3\eta''|_2^2+|\eta_t''|_2^2+\epsilon|\nabla\eta_t''|_2^2\bigr)
 +|f''|_2|\eta_t''|_2+\epsilon|\nabla f''|_2|\nabla\eta_t''|_2.
 \end{aligned}\end{equation}

Taking the $L^2$ inner product of the second of \eqref{quasilinear 2aa} with $\vv v_t''$ yields
 \begin{equation}\label{lte11a}\begin{aligned}
 &\frac{1}{2}\frac{d}{dt}E_{32}(t)+\epsilon\bigl((\vv v\cdot\nabla)\Bigl(\frac{\vv v_t''}{1+\epsilon\eta}\Bigr)\,|\,\frac{\vv v_t''}{1+\epsilon\eta}\bigr)_2
 +\bigl(\frac{\epsilon\vv v}{(1+\epsilon\eta)^2}(\nabla\cdot\vv v_t'')\,|\,\vv v_t''\bigr)_2\\
 &\quad
 =-\frac{1}{2}(\partial_t\bigl(\frac{1}{1+\epsilon\eta}\bigr)\vv v_t''\,|\,\vv v_t'')_2
 +(\vv g''\,|\,\vv v_t'')_2,
\end{aligned}\end{equation}
where
\beno
 E_{32}(t)\underset{\text{def}}=(\frac{\vv v_t''}{1+\epsilon\eta}\,|\,\vv v_t'')_2+|\nabla\cdot\vv v''|_2^2+\epsilon|\nabla(\nabla\cdot\vv v'')|_2^2.
\eeno

Since $\curl(\frac{\vv v''}{1+\epsilon\eta})=0$, we have
\beno
\curl\vv v''=(1+\epsilon\eta)\bigl({v''}^1\partial_2(\frac{1}{1+\epsilon\eta})-{v''}^2\partial_1(\frac{1}{1+\epsilon\eta})\bigr)
=\frac{\epsilon({v''}^2\partial_1\eta-{v''}^1\partial_2\eta)}{1+\epsilon\eta},
\eeno
which along with \eqref{ansatz a1} and \eqref{ansatz a} shows that
\begin{equation}\label{curl}\begin{aligned}
&|\curl\vv v''|_2\lesssim\epsilon|\vv v''|_2,\\
&|\nabla(\curl\vv v'')|_2\lesssim\epsilon(|\vv v''|_2+|\nabla\vv v''|_2)|\eta|_{H^3}
\lesssim\epsilon(|\vv v''|_2+|\nabla\vv v''|_2),
\end{aligned}\end{equation}
where for the second inequality, we used the fact that $|\nabla\eta|_\infty\lesssim|\nabla\eta|_{H^2}$ and the following estimate
\beno
|\vv v''\nabla^2\eta|_2\lesssim|\vv v''|_4|\nabla^2\eta|_4\stackrel{\text{Sobolev}}{\lesssim}|\vv v''|_{\dot{H}^{\frac{1}{2}}}|\nabla^2\eta|_{\dot{H}^{\frac{1}{2}}}\stackrel{\text{interpolation}}{\lesssim}(|\vv v''|_2+|\nabla\vv v''|_2)|\nabla^2\eta|_{H^1}.
\eeno
Then by virtue of div-curl lemma, we obtain
 \begin{equation}\label{lte12a}\begin{aligned}
E_{32}(t)\sim|\vv v_t''|_2^2+|\nabla\vv v''|_2^2+\epsilon|\nabla^2\vv v''|_2^2+O(\epsilon|\vv v''|_2^2).
\end{aligned}\end{equation}
Similar to \eqref{lte9a}, integration by parts on the second term on the l.h.s of \eqref{lte11a} leads to
\begin{equation}\label{lte13a}
\epsilon\bigl((\vv v\cdot\nabla)\Bigl(\frac{\vv v_t''}{1+\epsilon\eta}\Bigr)\,|\,\frac{\vv v_t''}{1+\epsilon\eta}\bigr)_2=-\frac{\epsilon}{2}\bigl(\nabla\cdot\vv v\frac{\vv v_t''}{1+\epsilon\eta}\,|\,\frac{\vv v_t''}{1+\epsilon\eta}\bigr)_2.
\end{equation}

For the third term on the l.h.s of \eqref{lte11a}, by integration by parts, we have
\beno\begin{aligned}
&\bigl(\frac{\epsilon\vv v}{(1+\epsilon\eta)^2}(\nabla\cdot\vv v_t'')\,|\,\vv v_t''\bigr)_2
=\epsilon\sum_{i,j=1,2}\bigl(\frac{v^i}{(1+\epsilon\eta)^2}(\partial_j{v_t''}^j)\,|\,{v_t''}^i\bigr)_2\\
&=-\epsilon\sum_{i,j=1,2}\bigl(\partial_j\bigl(\frac{v^i}{(1+\epsilon\eta)^2}\bigr){v_t''}^j\,|\,{v_t''}^i\bigr)_2
-\epsilon\sum_{i,j=1,2}\bigl(\frac{v^i}{(1+\epsilon\eta)^2}{v_t''}^j\,|\,\partial_j{v_t''}^i\bigr)_2.
\end{aligned}\eeno
To estimate the second term on the right hand side of the above equality, we first obtain by using $\curl(\frac{\vv v''}{1+\epsilon\eta})=0$ that
\beno\begin{aligned}
\partial_jv''^i_t&
=\Bigl(\partial_i\bigl(\frac{v''^j}{1+\epsilon\eta}\bigr)\cdot(1+\epsilon\eta)+\frac{\epsilon v''^i\partial_j\eta}{1+\epsilon\eta}\Bigr)_t=\partial_iv''^j_t+\epsilon\Bigl(\frac{v''^i\partial_j\eta-v''^j\partial_i\eta}{1+\epsilon\eta}\Bigr)_t.
\end{aligned}\eeno
By integration by parts, we have
\beno\begin{aligned}
&\ \ -\epsilon\bigl(\frac{v^i}{(1+\epsilon\eta)^2}{v_t''}^j\,|\,\partial_j{v_t''}^i\bigr)_2\\
&
=\frac{\epsilon}{2}\bigl(\partial_i\bigl(\frac{v^i}{(1+\epsilon\eta)^2}\bigr){v_t''}^j\,|\,{v_t''}^j\bigr)_2
-\epsilon^2\bigl(\frac{v^i}{(1+\epsilon\eta)^2}{v_t''}^j\,|\,\Bigl(\frac{v''^i\partial_j\eta-v''^j\partial_i\eta}{1+\epsilon\eta}\Bigr)_t\bigr)_2.
\end{aligned}\eeno
Then we obtain
\beno\begin{aligned}
&\bigl(\frac{\epsilon\vv v}{(1+\epsilon\eta)^2}(\nabla\cdot\vv v_t'')\,|\,\vv v_t''\bigr)_2
=\sum_{i,j=1,2}\Bigl(-\epsilon\bigl(\partial_j\bigl(\frac{v^i}{(1+\epsilon\eta)^2}\bigr){v_t''}^j\,|\,{v_t''}^i\bigr)_2\\
&\quad
+\frac{\epsilon}{2}\bigl(\partial_i\bigl(\frac{v^i}{(1+\epsilon\eta)^2}\bigr){v_t''}^j\,|\,{v_t''}^j\bigr)_2
-\epsilon^2\bigl(\frac{v^i}{(1+\epsilon\eta)^2}{v_t''}^j\,|\,\Bigl(\frac{v''^i\partial_j\eta-v''^j\partial_i\eta}{1+\epsilon\eta}\Bigr)_t\bigr)_2\Bigr),
\end{aligned}\eeno
which along with \eqref{ansatz a1}, \eqref{ansatz a}, \eqref{lte11a} and \eqref{lte13a} implies
\begin{equation}\label{lte14a}\begin{aligned}
\frac{1}{2}\frac{d}{dt}E_{32}(t)\lesssim&\epsilon\bigl(|\eta_t|_\infty+\epsilon^{\frac{1}{2}}|\nabla\eta_t|_4+|\nabla\vv v|_\infty+|\nabla\eta|_\infty\bigr)\\
&\times\bigl(|\vv v_t''|_2^2+\epsilon|\nabla\vv v''|_2^2+\epsilon^2|\vv v''|_2^2\bigr)+|\vv g''|_2|\vv v_t''|_2.
\end{aligned}\end{equation}

\smallskip

Now, we define $E_3(t)\underset{\text{def}}=E_{31}(t)+E_{32}(t)$. Then \eqref{lte8a} and \eqref{lte12a} yields
\begin{equation}\label{lte16a}\begin{aligned}
E_3(t)&\sim|\nabla\eta''|_2^2+\epsilon|\nabla^2\eta''|_2^2+\epsilon^2|\nabla^3\eta''|_2^2+|\eta_t''|_2^2+\epsilon|\nabla\eta_t''|_2^2\\
&\quad
+|\nabla\vv v''|_2^2+\epsilon|\nabla^2\vv v''|_2^2+|\vv v_t''|_2^2+O(\epsilon|\vv v''|_2^2).
\end{aligned}\end{equation}

Combining estimates \eqref{lte10a} and \eqref{lte14a},  using \eqref{lte16a}, we obtain
\begin{equation}\label{third order estimate a}\begin{aligned}
\frac{1}{2}&\frac{d}{dt}E_3(t)\lesssim
\epsilon\bigl(|\eta_t|_\infty+\epsilon^{\frac{1}{2}}|\nabla\eta_t|_4+|\nabla\eta|_\infty+|\nabla^2\eta|_3+|\nabla\vv v|_\infty\bigr)E_3(t)\\
&\quad +|f''|_2|\eta_t''|_2+\epsilon|\nabla f''|_2|\nabla\eta_t''|_2+|\vv g''|_2|\vv v_t''|_2.
\end{aligned}\end{equation}

\smallskip

 {\it Step 3.3. Estimates for \eqref{quasilinear 1a} and\eqref{quasilinear 2a}.} Since \eqref{quasilinear 1a} and \eqref{quasilinear 2a} have the same form as \eqref{quasilinear 2aa}, we have similar estimates as \eqref{third order estimate a} only with $(\eta'',\vv v'')$ being replaced  by $(\eta,\vv v)$ and $(\eta',\vv v')$ respectively.

 \smallskip

In order to get the total estimate for system \eqref{quasilinear 1a}, \eqref{quasilinear 2a} and \eqref{quasilinear 2aa}, we have to estimate the source terms $|f|_2+\epsilon^{\frac{1}{2}}|\nabla f|_2+|\vv g|_2$, $|f'|_2+\epsilon^{\frac{1}{2}}|\nabla f'|_2+|\vv g'|_2$ and $|f''|_2+\epsilon^{\frac{1}{2}}|\nabla f''|_2+|\vv g''|_2$. Thanks to the expressions of $f$, $\vv g$, $f'$, $\vv g'$  and $f''$, $\vv g''$, using \eqref{ansatz a1} and \eqref{ansatz a}, after  tedious but elementary calculations, we obtain
\begin{equation}\label{lte22a}\begin{aligned}
&|f|_2+\epsilon^{\frac{1}{2}}|\nabla f|_2+|\vv g|_2+|f'|_2+\epsilon^{\frac{1}{2}}|\nabla f'|_2+|\vv g'|_2\\
&\quad
+|f''|_2+\epsilon^{\frac{1}{2}}|\nabla f''|_2+|\vv g''|_2\lesssim\epsilon\mathcal{E}(t),
\end{aligned}\end{equation}
where
\begin{equation}\label{total energy for 2D}\begin{aligned}
\mathcal{E}(t)=&|\eta|_{X^3_{\epsilon^4}}^2+|\eta_t|_{X^2_{\epsilon^3}}^2+|\eta_{tt}|_{X^1_{\epsilon^2}}^2
+|\eta_{ttt}|_{X^0_{\epsilon}}^2
+|\vv v|_{X^3_{\epsilon^3}}^2\\
&\quad+|\vv v_t|_{X^2_{\epsilon^2}}^2+|\vv v_{tt}|_{X^1_{\epsilon}}^2+|\vv v_{ttt}|_2^2.
\end{aligned}\end{equation}

In the process of derivation of \eqref{lte22a}, we  used the fact that $\eta'=\eta_t,\, \eta''=\eta_{tt},\, \vv v'=\vv v_t,\, \vv v''=\vv v_{tt}$ and used the H\"older inequalities, Sobolev inequalities and interpolation inequalities frequently. We shall not show the details here.

{\bf Step 4. The final estimate  on \eqref{reduction form 1 for 2D}.} Before closing the {\it a priori estimates}, we first define the energy functional associated to the quasilinear system \eqref{reduction form 1 for 2D}-\eqref{quasilinear 1a}-\eqref{quasilinear 2a}-\eqref{quasilinear 2aa} as
\begin{equation}\label{functional 1a}
E(t)\underset{\text{def}}=E_0(t)+E_1(t)+E_2(t)+E_3(t),
\end{equation}
and $E_1(t),\, E_2(t)$ are defined in the same way as $E_3(t)$  with $(\eta'',\vv v'')$ being replaced by $(\eta,\vv v)$ and $(\eta',\vv v')$ respectively.
Notice that $\eta'=\eta_t,\, \eta''=\eta_{tt}$ and $\vv v'=\vv v_t,\, \vv v''=\vv v_{tt}$. Then \eqref{lte6a} and \eqref{lte16a} yield
\begin{equation}\label{functional 2a}\begin{aligned}
E(t)\sim&|\eta|_{X^1_{\epsilon^2}}^2+|\eta_t|_{X^1_{\epsilon^2}}^2+|\eta_{tt}|_{X^1_{\epsilon^2}}^2
+|\eta_{ttt}|_{X^0_{\epsilon}}^2
+|\vv v|_{X^1_{\epsilon}}^2\\
&\quad+|\vv v_t|_{X^1_{\epsilon}}^2+|\vv v_{tt}|_{X^1_{\epsilon}}^2+|\vv v_{ttt}|_2^2.
\end{aligned}\end{equation}

With the definitions \eqref{functional 1a} and \eqref{total energy for 2D}, using the interpolation inequality \eqref{int1} and the inequalities that $|u|_{L^\infty(\R^2)}\lesssim|u|_{H^2(\R^2)}$ and $|u|_{L^4(\R^2)}\lesssim|u|_{L^2(\R^2)}^{\frac{1}{2}}|\nabla u|_{L^2(\R^2)}^{\frac{1}{2}}$, the energy estimates \eqref{lowest estimate for 2D}, \eqref{third order estimate a} and \eqref{lte22a} give rise to
\begin{equation}\label{energy estimate 1 for 2D}
\frac{1}{2}\frac{d}{dt}E(t)\lesssim\epsilon\mathcal{E}(t)^{\frac{3}{2}},
\end{equation}
where $\mathcal{E}(t)$ is the total energy functional to \eqref{reduction form 1 for 2D} which is defined  in \eqref{total energy for 2D}.

To finish the proof, we have to show that
\begin{equation}\label{sim for 2D}
\mathcal{E}(t)\sim E(t).
\end{equation}
Indeed, thanks to \eqref{functional 2a} and \eqref{total energy for 2D}, we have
\beno
\mathcal{E}(t)\sim E(t)+|\eta|_{X^3_{\epsilon^4}}^2+|\eta_t|_{X^2_{\epsilon^3}}^2+|\vv v|_{X^3_{\epsilon^3}}^2+|\vv v_t|_{X^2_{\epsilon^2}}^2.
\eeno
Then we only need to show that
\beno
|\eta|_{X^3_{\epsilon^4}}^2+|\eta_t|_{X^2_{\epsilon^3}}^2+|\vv v|_{X^3_{\epsilon^3}}^2+|\vv v_t|_{X^2_{\epsilon^2}}^2\lesssim E(t).
\eeno
That is to say, we shall recover the regularity in space through the  regularity in time.
More precisely, \eqref{reduction form 1 for 2D} yields
\begin{equation}\label{transfer for 2D}
\nabla\cdot\vv v=-\eta_t,\quad (1-\epsilon\Delta)\nabla\eta=-\frac{\vv v_t}{1+\epsilon\eta}-\frac{\epsilon}{1+\epsilon\eta}\nabla\cdot\Bigl(\frac{\vv v}{1+\epsilon\eta}\otimes\vv v\Bigr).
\end{equation}

To control $|\vv v_t|_{X_{\epsilon^2}^2}$, we first have
\beno\begin{aligned}
&\quad|\vv v_t|_{X_{\epsilon^2}^2}^2=|\vv v_t|_{H^2}^2+\epsilon^2|\vv v_t|_{H^{2+2}}^2\\
&\lesssim|\vv v_t|_{H^1}^2+|\nabla(\nabla\cdot\vv v_t)|_2^2+|\nabla(\curl\vv v_t)|_2^2+\epsilon^2|\nabla^3(\nabla\cdot\vv v_t)|_2^2
+\epsilon^2|\nabla^3(\curl\vv v_t)|_2^2.
\end{aligned}\eeno
Since $\curl\bigl(\frac{\vv v_t}{1+\epsilon\eta}\bigr)=0$, similar derivation as \eqref{curl} leads to
\beno
\curl\vv v_t
=\frac{\epsilon(v_t^2\partial_1\eta-v_t^1\partial_2\eta)}{1+\epsilon\eta},
\eeno
and \beno\begin{aligned}
&|\nabla(\curl\vv v_t)|_2\lesssim\epsilon(|\vv v_t|_2+|\nabla\vv v_t|_2)|\eta|_{H^3}
\lesssim\epsilon|\vv v_t|_{H^1},\\
&|\nabla^3(\curl\vv v_t)|_2\lesssim\epsilon^{\frac{1}{2}}(|\vv v_t|_2+|\nabla^2\vv v_t|_2)|\eta|_{X^3_\epsilon}
\lesssim\epsilon^{\frac{1}{2}}|\vv v_t|_{H^2},
\end{aligned}\eeno
Then we have
\beno
|\vv v_t|_{X_{\epsilon^2}^2}^2\lesssim|\vv v_t|_{H^1}^2+|\nabla(\nabla\cdot\vv v_t)|_2^2+\epsilon^2|\nabla^3(\nabla\cdot\vv v_t)|_2^2
+\epsilon^3|\vv v_t|_{H^2}^2,
\eeno
which along with the fact that $\epsilon$ is small enough implies
\beno
|\vv v_t|_{X_{\epsilon^2}^2}^2\lesssim|\vv v_t|_{H^1}^2+|\nabla(\nabla\cdot\vv v_t)|_2^2+\epsilon^2|\nabla^3(\nabla\cdot\vv v_t)|_2^2.
\eeno
Now using the first equation of \eqref{transfer for 2D}, we obtain that
\begin{equation}\label{improve 1 for 2D}
|\vv v_t|_{X_{\epsilon^2}^2}^2\lesssim |\vv v_t|_{H^1}^2+|\nabla\eta_{tt}|_2^2+\epsilon^2|\nabla^3\eta_{tt}|_2^2
\lesssim E(t).
\end{equation}
Similarly, we obtain
\begin{equation}\label{improve 1a for 2D}
|\vv v|_{X_{\epsilon^2}^2}^2\lesssim |\vv v|_{H^1}^2+|\nabla\eta_{t}|_2^2+\epsilon^2|\nabla^3\eta_{t}|_2^2
\lesssim E(t).
\end{equation}

 While the second equation of \eqref{transfer for 2D}, \eqref{ansatz a}, \eqref{improve 1 for 2D} and \eqref{improve 1a for 2D} imply
\begin{equation}\label{improve 2 for 2D}\begin{aligned}
&|\eta_t|_{X_{\epsilon^3}^2}^2\sim|\eta_t|_{H^1}^2+|\nabla\bigl[(1-\epsilon\Delta)\nabla\eta_t\bigr]|_2^2
+\epsilon|\nabla^2\bigl[(1-\epsilon\Delta)\nabla\eta_t\bigr]|_2^2\\
&\lesssim |\eta_t|_{X_{\epsilon^2}^1}^2+|\vv v_{tt}|_{X_{\epsilon}^1}^2+|\vv v_t|_{X_{\epsilon^2}^2}^2+|\vv v|_{X_{\epsilon^2}^2}^2
+|\eta|_{X_{\epsilon^2}^1}^2
\lesssim E(t).
\end{aligned}
\end{equation}

To bound  $|\vv v|_{X_{\epsilon^3}^3}$, we first have
\beno\begin{aligned}
&\quad|\vv v|_{X_{\epsilon^3}^3}^2\sim|\vv v|_{H^2}^2+|\nabla^3\vv v|_2^2+\epsilon^3|\vv v|_{H^{3+3}}^2\\
&\lesssim|\vv v|_{H^2}^2+|\nabla^2(\nabla\cdot\vv v)|_2^2+|\nabla^2(\curl\vv v)|_2^2+\epsilon^3|\nabla^5(\nabla\cdot\vv v)|_2^2
+\epsilon^3|\nabla^5(\curl\vv v)|_2^2.
\end{aligned}\eeno
Similar derivation as \eqref{curl} yields
\beno\begin{aligned}
&|\nabla^2(\curl\vv v)|_2\lesssim\epsilon|\vv v|_{H^2},\\
&\epsilon^{\frac{1}{2}}|\nabla^5(\curl\vv v_t)|_2^2\lesssim|\vv v|_{X^3_{\epsilon^2}}|\eta|_{X^3_{\epsilon^3}}(1+|\eta|_{X^3_{\epsilon}})
\lesssim|\vv v|_{X^3_{\epsilon^2}}.
\end{aligned}\eeno
Then we obtain
\beno\begin{aligned}
&|\vv v|_{X_{\epsilon^3}^3}^2\lesssim|\vv v|_{H^2}^2+|\nabla^2(\nabla\cdot\vv v)|_2^2+\epsilon^3|\nabla^5(\nabla\cdot\vv v)|_2^2
+\epsilon^2|\vv v|_{X^3_{\epsilon^2}},
\end{aligned}\eeno
which gives rise to
\beno
|\vv v|_{X_{\epsilon^3}^3}^2\lesssim|\vv v|_{H^2}^2+|\nabla^2(\nabla\cdot\vv v)|_2^2+\epsilon^3|\nabla^5(\nabla\cdot\vv v)|_2^2.
\eeno
Then using the first equation of \eqref{sim for 2D}, \eqref{improve 2 for 2D} and \eqref{functional 2a}, we obtain
\begin{equation}\label{improve 3 for 2D}
|\vv v|_{X_{\epsilon^3}^3}^2\lesssim E(t).
\end{equation}

For $|\eta|_{X_{\epsilon^4}^3}^2$, similar to the derivation of \eqref{improve 2 for 2D}, by using the second equation of \eqref{sim for 2D},
\eqref{functional 2a}, \eqref{improve 1 for 2D}, \eqref{improve 2 for 2D} and \eqref{improve 3 for 2D}, we finally obtain that
\begin{equation}\label{improve 4 for 2D}\begin{aligned}
&|\eta|_{X_{\epsilon^4}^3}^2\sim|\eta|_{H^1}^2+|\nabla^2\bigl[(1-\epsilon\Delta)\nabla\eta\bigr]|_2^2
+\epsilon^2|\nabla^4\bigl[(1-\epsilon\Delta)\nabla\eta\bigr]|_2^2\lesssim E(t).
\end{aligned}
\end{equation}

 Due to \eqref{sim for 2D} and \eqref{energy estimate 1 for 2D}, we have
\begin{equation}\label{energy estimate 2 for 2D}
\frac{1}{2}\frac{d}{dt}E(t)\lesssim\epsilon E(t)^{\frac{3}{2}}.
\end{equation}

{\bf Step 5. Initial data for the quasilinear system and final estimate.}
In this step, we have to derive the regularity for the initial data to the quasilinear system through the system \eqref{reduction form 1 for 2D} and the regularity for initial data  $(\eta_0,\vv v_0)$. The first equation of \eqref{reduction form 1 for 2D} shows that
\beno
|\eta'|_{t=0}|_{X^2_{\epsilon^3}}=|\eta_t|_{t=0}|_{X^2_{\epsilon^3}}=|\nabla\cdot\vv v_0|_{X^2_{\epsilon^3}}\lesssim|\vv v_0|_{X^3_{\epsilon^3}},
\eeno
while the second equation of \eqref{reduction form 1 for 2D} shows that
\beno\begin{aligned}
&\quad|\vv v'|_{t=0}|_{X^2_{\epsilon^2}}=|\vv v_t|_{t=0}|_{X^2_{\epsilon^2}}\\
&\lesssim|(1+\epsilon\eta_0)(1-\epsilon\Delta)\nabla\eta_0|_{X^2_{\epsilon^2}}
+\epsilon|\nabla\cdot\Bigl(\frac{\vv v_0}{1+\epsilon\eta_0}\otimes\vv v_0\Bigr)|_{X^2_{\epsilon^2}}\\
&\lesssim|\eta_0|_{X^3_{\epsilon^4}}+|\vv v_0|_{X^3_{\epsilon^3}},
\end{aligned}\eeno
where we assume that $|\eta_0|_{X^3_{\epsilon^4}}+|\vv v_0|_{X^3_{\epsilon^3}}\leq C$ and $\epsilon\leq\epsilon_0$ with $\epsilon_0$ small enough.

Similarly, thanks to \eqref{quasilinear 2a}, we can obtain the upper bound of $|\eta'_t|_{t=0}|_{X^1_{\epsilon^2}}+|\vv v'_t|_{t=0}|_{X^1_\epsilon}$ ( or $|\eta_{tt}|_{t=0}|_{X^1_{\epsilon^2}}+|\vv v_{tt}|_{t=0}|_{X^1_\epsilon}$). While by \eqref{quasilinear 2aa}, we can also derive the upper bound for $|\eta''_t|_{t=0}|_{X^0_{\epsilon}}+|\vv v''_t|_{t=0}|_2$ (or $|\eta_{ttt}|_{t=0}|_{X^0_{\epsilon}}+|\vv v_{ttt}|_{t=0}|_2$).
Then we finally obtain that
\begin{equation}\label{regularity for initial data for 2D}
E(0)\sim\mathcal{E}(0)\lesssim|\eta_0|_{X^3_{\epsilon^4}}^2+|\vv v_0|_{X^3_{\epsilon^3}}^2.
\end{equation}

\medskip

{\bf Step 6. Existence and uniqueness. } The estimates  \eqref{energy estimate 2 for 2D} and \eqref{regularity for initial data for 2D} are crucial to prove  the existence of  $T>0$ independent of $\epsilon$ such that \eqref{reduction form 1 for 2D} has a unique solution $(\eta,\vv v)$ on a  time interval $[0,T/\epsilon]$ with initial data $(\eta_0,\vv v_0)\in X^3_{\epsilon^4}\times X^3_{\epsilon^3}$  satisfying moreover, by \eqref{energy estimate 2 for 2D} and \eqref{sim for 2D} the estimate
\begin{equation}\label{total energy estimate for 2D}
\sup_{t\in [0,T/\epsilon]}\mathcal{E}(t)\lesssim|\eta_0|_{X^3_{\epsilon^4}}^2+|\vv v_0|_{X^3_{\epsilon^3}}^2.
\end{equation}

The proof of the existence and uniqueness is postponed to Section 5.

\medskip

Notice that $\vv v=(1+\epsilon\eta)\vv u$. Then we have  obtained the long time estimate of solutions to  the original Boussinesq system \eqref{abcd}-\eqref{rme2} with $a=b=d=0, c=-1$ together to  the energy estimate \eqref{lte2a}.
\end{proof}

\section{Existence proof of Theorems \ref{long time existence for case c=-1} and \ref{long time existence for case c=-1 2D}}
In this section, we shall complete the proof of existence and uniqueness of solutions to the transformed systems \eqref{reduction form 1} and \eqref{reduction form 1 for 2D} so that we could complete the proofs to Theorems \ref{long time existence for case c=-1} and \ref{long time existence for case c=-1 2D}. In order to construct the approximate solutions to \eqref{reduction form 1} and \eqref{reduction form 1 for 2D}, we introduce the mollifier operator $\mathcal{J}_\delta$ as follows (see \cite{Lannes2}):
\beno
\widehat{\mathcal{J}_\delta f}(\xi)=\varphi(\delta\xi)\hat{f}(\xi),\quad\forall\xi\in\R^d,\quad\forall f\in L^2(\R^d),
\eeno
where $\varphi\in C_0^\infty(\R^d)$ and $\varphi(0)=1$. Then using Fourier transform, we  obtain the following properties for $\mathcal{J}_\delta$:
\begin{lemma}\label{lemma for mollifier}
For any $s,s'\in\R$ and $1\leq p\leq \infty$, there hold:

(i) $|\mathcal{J}_\delta f|_{H^{s'}}\leq C_{s,s',\delta}|f|_{H^{s}}$;

(ii)$|\mathcal{J}_\delta f|_p\leq C|f|_p$;

(iii) $|\mathcal{J}_\delta f-f|_{H^{s-1}}\leq C\delta|f|_{H^{s}}$;

 (iv) $|\mathcal{J}_\delta f-f|_{H^{s}}\rightarrow 0$ as $\delta\rightarrow 0$;

 (v) $[\mathcal{J}_\delta,a] f|_{H^{s}}\leq C|a|_{H^{t_0+1}}|f|_{H^{s-1}}$, for any $t_0\geq\frac{d}{2}$ and $-t_0<s\leq t_0+1$;\\
%
where $C$ is an universal constant independent of $\delta$ and $C_{s,s',\delta}$ is a constant depending on $s,s',\delta$.
\end{lemma}
\begin{proof}
The statements (i), (iii) and (iv) are verified directly by Fourier analysis. For (ii), denoting by $\breve{\varphi}(\cdot)$ is the inverse Fourier transform of $\varphi$. Then we have
\beno
\mathcal{J}_\delta f=\delta^{-d}\breve{\varphi}(\frac{\cdot}{d})*f.
\eeno
Notice that $\delta^{-d}|\breve{\varphi}(\frac{\cdot}{d})|_1\leq C$. Then (ii) follows by Young inequality.

The statement (v) is a consequence of Theorems 3 and 6 in \cite{Lannes1}. Indeed, since $\mathcal{J}_\delta$ is a zeroth order Fourier multiplier,  by
\cite{Lannes1}, we have
\beno
[\mathcal{J}_\delta,a] f|_{H^{s}}\lesssim C(\varphi(\delta\cdot))|a|_{H^{t_0+1}}|f|_{H^s},
\eeno
where
\beno
C(\varphi(\delta\cdot))=\sup_{|\beta|\leq2+d+[\frac{d}{2}]}\sup_{|\xi|\geq\frac{1}{4}}\langle\xi\rangle^{|\beta|}|\partial_\xi^\beta\varphi(\delta\xi)|
+\sup_{|\xi|\leq 1}|\varphi(\delta\xi)|\leq C.
\eeno
%
Thus, the lemma is proved.
\end{proof}

 We only give the details of the existence proof to Theorem \ref{long time existence for case c=-1}}. The existence proof of Theorem \ref{long time existence for case c=-1 2D} follows a similar line.

Now, we divide the proof into several steps.

{\bf Step 1. Construction of the approximate solutions to \eqref{reduction form 1}.} We construct an approximate solution sequence $\{(\eta^\delta,\, v^\delta)\}_{\delta>0}$ satisfying the following regularizing system
\begin{equation}\label{approximate system}
\left\{\begin{aligned}
&\eta_t^\delta+\mathcal{J}_\delta v^\delta_x=0 \\
&v_t^\delta+(1+\epsilon\mathcal{J}_\delta\eta^\delta)(1-\epsilon\partial_x^2)\mathcal{J}_\delta\eta^\delta_x
    +\epsilon\mathcal{J}_\delta^2\Bigl(\frac{|\mathcal{J}_\delta^2 v^\delta|^2}{1+\epsilon\mathcal{J}_\delta\eta^\delta}\Bigr)_x=0,
 \end{aligned}\right.
    \end{equation}
associated with the initial data $\eta^\delta|_{t=0}=\eta_0$ and $v^\delta|_{t=0}=v_0$.

Denoting by $V^\delta=(\eta^\delta,\ v^\delta)$, then \eqref{approximate system} can be reduced to the following ODE in the Banach space $H^{m+1}\times H^m$ with $m\geq 0$:
\begin{equation}\label{ODE form for approximate system}
\frac{d}{dt}V^\delta(t)=F_\delta(V^\delta),\quad V^\delta(0)=V^\delta_0\underset{\text{def}}{=}(\eta_0,v_0),
\end{equation}
where $F_\delta(V^\delta)=(F_\delta^1(V^\delta),F_\delta^2(V^\delta))$ with
\beno\begin{aligned}
&F_\delta^1(V^\delta)=-\mathcal{J}_\delta v^\delta_x,\\
&F_\delta^2(V^\delta)=-(1+\epsilon\mathcal{J}_\delta\eta^\delta)(1-\epsilon\partial_x^2)\mathcal{J}_\delta\eta^\delta_x
    -\epsilon\mathcal{J}_\delta^2\Bigl(\frac{|\mathcal{J}_\delta^2 v^\delta|^2}{1+\epsilon\mathcal{J}_\delta\eta^\delta}\Bigr)_x.
\end{aligned}\eeno

For any $V^\delta_1$, $V^\delta_2$, by virtue of the properties of $\mathcal{J}_\delta$ in Lemma \ref{lemma for mollifier}, we have
\beno\begin{aligned}
|F_\delta^1(V^\delta_1)-F_\delta^1(V^\delta_2)|_{H^{m+1}}=|\mathcal{J}_\delta\partial_x(v^\delta_1-v^\delta_2)|_{H^{m+1}}\leq C_{\delta,m}
|v^\delta_1-v^\delta_2|_{H^m}.
\end{aligned}\eeno
Similarly, by Lemma \ref{lemma for mollifier} and the product estimate, we have
\beno\begin{aligned}
&|F_\delta^2(V^\delta_1)-F_\delta^2(V^\delta_2)|_{H^m}\leq C_{\delta,m}(|\mathcal{J}_\delta \eta^\delta_1|_{H^m},|\mathcal{J}_\delta \eta^\delta_2|_{H^m}, |\mathcal{J}_\delta v^\delta_1|_{H^m}, |\mathcal{J}_\delta v^\delta_2|_{H^m})\\
&\qquad\times\bigl(|\eta^\delta_1-\eta^\delta_2|_{H^m}+|v^\delta_1-v^\delta_2|_{H^m}\bigr)\\
&\leq C_{\delta,m}(|V_1^\delta|_2,|V_2^\delta|_2)|V_1^\delta-V_2^\delta|_{H^{m+1}\times H^m},
\end{aligned}\eeno
where $C_{\delta,m}(\lambda_1,\lambda_2,\cdots)$ is a constant depending on $\delta, m$ and $\lambda_1, \lambda_2, \cdots$. Then we have
\beno
|F_\delta(V^\delta_1)-F_\delta(V^\delta_2)|_{H^{m+1}\times H^m}\leq C_{\delta,m}(|V_1^\delta|_2,|V_2^\delta|_2)|V_1^\delta-V_2^\delta|_{H^{m+1}\times H^m}
\eeno
so that $F_\delta(\cdot)$ is locally Lipschitz continuous on any open set
\beno
\mathcal{O}_M=\{V\in H^{m+1}\times H^m(\R)\,|\, |V|_{H^{m+1}\times H^m}\leq M\}.
\eeno
Thus, Picard (Cauchy-Lipschitz) existence theorem implies that, given any initial data $V_0\in H^{m+1}\times H^m(\R)$, there exists a unique solution $V^\delta\in C^1([0,T_\delta); \mathcal{O}_M\cap (H^{m+1}\times H^m))$ for some $T_\delta>0$, with any integer $m\geq 0$.

Going back to the regularizing system \eqref{approximate system}, since $V^\delta=(\eta^\delta,v^\delta)\in  C^1([0,T_\delta); H^{m+1}\times H^m)$, by virtue of the properties to $\mathcal{J}_\delta$, we have
\beno
\partial_tV^\delta\in C^1([0,T_\delta); H^{m+1}\times H^m),
\eeno
which implies
\beno
V^\delta\in C^2([0,T_\delta); H^{m+1}\times H^m).
\eeno
Moreover, we could obtain that
\beno
V^\delta\in C^k([0,T_\delta); H^{m+1}\times H^m),\quad\text{for any}\quad k\in\N.
\eeno
Thus, we could apply $\partial_t$ many times to \eqref{approximate system}.

\medskip

{\bf Step 2. Uniform energy estimates on the approximate solutions on some time interval $[0, T/\epsilon)$.} In this step, we shall prove that there exists a uniform existence time interval $[0, T/\epsilon)$ with $T$ being independent of $\delta$ and $\epsilon$. To do so, we have to derive the uniform energy estimates for the approximate solutions $V^\delta$.

{\it Step 2.1. The reduction equations.}
Motivated by the a priori energy estimates for $\eqref{reduction form 1}$, we  apply $\partial_t$ to \eqref{approximate system}. Similar derivation as \eqref{quasilinear 1}, we obtain
\begin{equation}\label{quasilinear for app}\left\{\begin{aligned}
&\eta^\delta_{tt}-\mathcal{J}_\delta\partial_x\bigl((1+\epsilon\mathcal{J}_\delta\eta^\delta)\partial_x\mathcal{J}_\delta\eta^\delta\bigr)
+\epsilon\mathcal{J}_\delta\partial_x\bigl((1+\epsilon\mathcal{J}_\delta\eta^\delta)\partial_x^3\mathcal{J}_\delta\eta^\delta\bigr)\\
&\quad+2\epsilon\mathcal{J}_\delta^2\Bigl(\frac{\mathcal{J}_\delta^2v^\delta}{1+\epsilon\mathcal{J}_\delta\eta^\delta}
\cdot\partial_x\mathcal{J}_\delta^2\eta^\delta_t\Bigr)=f^\delta,\\
&\frac{1}{1+\epsilon\mathcal{J}_\delta\eta^\delta}v^\delta_{tt}-\partial_x^2\mathcal{J}_\delta^2v^\delta+\epsilon \partial_x^4\mathcal{J}_\delta^2v^\delta\\
&\quad+\frac{2\epsilon}{1+\epsilon\mathcal{J}_\delta\eta^\delta}\mathcal{J}_\delta^2\Bigl(
\mathcal{J}_\delta^2v^\delta\cdot\partial_x\mathcal{J}_\delta^2\bigl(\frac{v^\delta_t}{1+\epsilon\mathcal{J}_\delta\eta^\delta}
\bigr)\Bigr)=g^\delta,
\end{aligned}\right.\end{equation}
where
\beno\begin{aligned}
f^\delta&\underset{\text{def}}=\epsilon\mathcal{J}_\delta^3\Bigl[\Bigl(\frac{|\mathcal{J}_\delta^2 v^\delta|^2}{1+\epsilon\mathcal{J}_\delta\eta^\delta}\Bigr)_{xx}-\frac{2\mathcal{J}_\delta^2v^\delta}{1+\epsilon\mathcal{J}_\delta\eta^\delta}
\cdot\mathcal{J}_\delta^2v^\delta_{xx}\Bigr]\\
&\quad-2\epsilon\mathcal{J}_\delta^2\bigl([\mathcal{J}_\delta,\frac{\mathcal{J}_\delta^2v^\delta}{1+\epsilon\mathcal{J}_\delta\eta^\delta}]
\partial_x\mathcal{J}_\delta\eta^\delta_t\bigr),
\end{aligned}\eeno
\beno\begin{aligned}
g^\delta&\underset{\text{def}}=-\frac{\epsilon\mathcal{J}_\delta\eta^\delta_t}{1+\epsilon\mathcal{J}_\delta\eta^\delta}
(1-\epsilon\partial_x^2)\partial_x\mathcal{J}_\delta\eta^\delta-\frac{\epsilon}{1+\epsilon\mathcal{J}_\delta\eta^\delta}\mathcal{J}_\delta^2\partial_x
\Bigl(|\mathcal{J}_\delta^2v^\delta|^2\bigl(\frac{1}{1+\epsilon\mathcal{J}_\delta\eta^\delta}\bigr)_t\Bigr)\\
&\quad-\frac{2\epsilon}{1+\epsilon\mathcal{J}_\delta\eta^\delta}\mathcal{J}_\delta^2
\Bigl(\frac{\mathcal{J}_\delta^2v^\delta_x}{1+\epsilon\mathcal{J}_\delta\eta^\delta}
\cdot\mathcal{J}_\delta^2v^\delta_t-\mathcal{J}_\delta^2v^\delta\cdot\mathcal{J}_\delta^2v^\delta_t
\bigl(\frac{1}{1+\epsilon\mathcal{J}_\delta\eta^\delta}\bigr)_x\Bigr)\\
&\quad+\frac{2\epsilon}{1+\epsilon\mathcal{J}_\delta\eta^\delta}
\mathcal{J}_\delta^2\Bigl(\mathcal{J}_\delta^2v^\delta\cdot[\partial_x\mathcal{J}_\delta^2,\frac{1}{1+\epsilon\mathcal{J}_\delta\eta^\delta}]
v^\delta_t\Bigr).
\end{aligned}\eeno

Similarly, applying $\partial_t$ to \eqref{quasilinear for app}, denoting by $\eta^{'\delta}=\partial_t\eta^{\delta},\, v^{'\delta}=\partial_t v^{\delta}$, we obtain
\begin{equation}\label{quasilinear for app 1}\left\{\begin{aligned}
&\eta'^\delta_{tt}-\mathcal{J}_\delta\partial_x\bigl((1+\epsilon\mathcal{J}_\delta\eta^\delta)\partial_x\mathcal{J}_\delta\eta'^\delta\bigr)
+\epsilon\mathcal{J}_\delta\partial_x\bigl((1+\epsilon\mathcal{J}_\delta\eta^\delta)\partial_x^3\mathcal{J}_\delta\eta'^\delta\bigr)\\
&\quad+2\epsilon\mathcal{J}_\delta^2\Bigl(\frac{\mathcal{J}_\delta^2v^\delta}{1+\epsilon\mathcal{J}_\delta\eta^\delta}
\cdot\partial_x\mathcal{J}_\delta^2\eta'^\delta_t\Bigr)=f'^\delta,\\
&\frac{1}{1+\epsilon\mathcal{J}_\delta\eta^\delta}v'^\delta_{tt}-\partial_x^2\mathcal{J}_\delta^2v'^\delta+\epsilon \partial_x^4\mathcal{J}_\delta^2v'^\delta\\
&\quad+\frac{2\epsilon}{1+\epsilon\mathcal{J}_\delta\eta^\delta}\mathcal{J}_\delta^2\Bigl(
\mathcal{J}_\delta^2v^\delta\cdot\partial_x\mathcal{J}_\delta^2\bigl(\frac{v'^\delta_t}{1+\epsilon\mathcal{J}_\delta\eta^\delta}
\bigr)\Bigr)=g'^\delta,
\end{aligned}\right.\end{equation}
where
\beno\begin{aligned}
f'^\delta&\underset{\text{def}}=\partial_tf^\delta+\epsilon\mathcal{J}_\delta\partial_x\bigl(\mathcal{J}_\delta\eta^\delta_t
\cdot\partial_x\mathcal{J}_\delta\eta^\delta\bigr)-\epsilon^2\mathcal{J}_\delta\partial_x\bigl(\mathcal{J}_\delta\eta^\delta_t
\cdot\partial_x^3\mathcal{J}_\delta\eta^\delta\bigr)\\
&\quad-2\epsilon\mathcal{J}_\delta^2\Bigl(\partial_t\bigl(\frac{\mathcal{J}_\delta^2v^\delta}
{1+\epsilon\mathcal{J}_\delta\eta^\delta}\bigr)
\cdot\partial_x\mathcal{J}_\delta^2\eta^\delta_t\Bigr),
\end{aligned}\eeno
\beno\begin{aligned}
g'^\delta&\underset{\text{def}}=\partial_tg^\delta-\partial_t\bigl(\frac{1}{1+\epsilon\mathcal{J}_\delta\eta^\delta}\bigr)v^\delta_{tt}
-2\epsilon\partial_t\bigl(\frac{1}{1+\epsilon\mathcal{J}_\delta\eta^\delta}\bigr)\mathcal{J}_\delta^2
\Bigl(
\mathcal{J}_\delta^2v^\delta\cdot\partial_x\mathcal{J}_\delta^2\bigl(\frac{v^\delta_t}{1+\epsilon\mathcal{J}_\delta\eta^\delta}
\bigr)\Bigr)\\
&\quad
-\frac{2\epsilon}{1+\epsilon\mathcal{J}_\delta\eta^\delta}\mathcal{J}_\delta^2
\Bigl(
\mathcal{J}_\delta^2v^\delta_t\cdot\partial_x\mathcal{J}_\delta^2\bigl(\frac{v^\delta_t}{1+\epsilon\mathcal{J}_\delta\eta^\delta}
\bigr)\Bigr)\\
&\quad+\frac{2\epsilon^2}{1+\epsilon\mathcal{J}_\delta\eta^\delta}\mathcal{J}_\delta^2
\Bigl(
\mathcal{J}_\delta^2v^\delta\cdot\partial_x\mathcal{J}_\delta^2\bigl(\frac{\mathcal{J}_\delta\eta^\delta_tv^\delta_t}
{(1+\epsilon\mathcal{J}_\delta\eta^\delta)^2}
\bigr)\Bigr).
\end{aligned}\eeno

{\it Step 2.2. Definitions of the energy functionals.} In this step, we always assume that
\begin{equation}\label{ansatz for app}
1+\epsilon\mathcal{J}_\delta\eta^\delta>H>0.
\end{equation}
This assumption is a consequence of the initial assumption $1+\epsilon\eta_0>H>0$ together with the smallness of $\epsilon$ and the following uniform energy estimates.

In order to derive the uniform energy estimates for approximate solutions, similar to the a priori energy estimates, we first introduce the energy functionals
$E^\delta(t)$ and $\mathcal{E}^\delta(t)$ in the similar way as $E(t)$ and $\mathcal{E}(t)$ in \eqref{functional 1} and \eqref{functional 3}. We define
\beno\begin{aligned}
E^\delta(t)&=E_0^\delta(t)+E_1^\delta(t)+E_2^\delta(t)\\
&=E_0^\delta(t)+\bigl(E_{11}^\delta(t)+E_{12}^\delta(t)\bigr)+\bigl(E_{21}^\delta(t)+E_{22}^\delta(t)\bigr)
\end{aligned}\eeno
with
\beno\begin{aligned}
&E_0^\delta(t)=\sum_{k=0}^2 E_{0k}^\delta\underset{\text{def}}{=}\sum_{k=0}^2\Bigl(|\partial^k\eta^\delta|_2^2+\epsilon|\partial^k\eta^\delta_x|_2^2
+(\frac{1}{1+\epsilon\mathcal{J}_\delta\eta^\delta}\partial^kv^\delta
\,|\,\partial^kv^\delta)_2\Bigr),
\end{aligned}\eeno
\beno\begin{aligned}
&E_{11}^\delta(t)=|\eta^\delta_t|_2^2+\epsilon|\eta^\delta_{tx}|_2^2
+((1+\epsilon\mathcal{J}_\delta\eta^\delta)\mathcal{J}_\delta\eta^\delta_x\,|\,\mathcal{J}_\delta\eta^\delta_x)_2\\
&\qquad+2\epsilon((1+\epsilon\mathcal{J}_\delta\eta^\delta)\mathcal{J}_\delta\eta^\delta_{xx}\,|\,\mathcal{J}_\delta\eta^\delta_{xx})_2
+\epsilon^2((1+\epsilon\mathcal{J}_\delta\eta^\delta)\mathcal{J}_\delta\eta^\delta_{xxx}\,|\,\mathcal{J}_\delta\eta^\delta_{xxx})_2,\\
&E_{12}^\delta(t)=(\frac{v^\delta_t}{1+\epsilon\mathcal{J}_\delta\eta^\delta}\,|\,v^\delta_t)_2+|\mathcal{J}_\delta v^\delta_x|_2^2
+\epsilon|\mathcal{J}_\delta v^\delta_{xx}|_2^2,
\end{aligned}\eeno
\beno\begin{aligned}
&E_{21}^\delta(t)=|\eta^\delta_{tt}|_2^2+\epsilon|\eta^\delta_{ttx}|_2^2
+((1+\epsilon\mathcal{J}_\delta\eta^\delta)\mathcal{J}_\delta\eta^\delta_{tx}\,|\,\mathcal{J}_\delta\eta^\delta_{tx})_2\\
&\qquad+2\epsilon((1+\epsilon\mathcal{J}_\delta\eta^\delta)\mathcal{J}_\delta\eta^\delta_{txx}\,|\,\mathcal{J}_\delta\eta^\delta_{txx})_2
+\epsilon^2((1+\epsilon\mathcal{J}_\delta\eta^\delta)\mathcal{J}_\delta\eta^\delta_{txxx}\,|\,\mathcal{J}_\delta\eta^\delta_{txxx})_2,\\
&E_{22}^\delta(t)=(\frac{v^\delta_{tt}}{1+\epsilon\mathcal{J}_\delta\eta^\delta}\,|\,v^\delta_{tt})_2+|\mathcal{J}_\delta v^\delta_{tx}|_2^2
+\epsilon|\mathcal{J}_\delta v^\delta_{txx}|_2^2.
\end{aligned}
\eeno
We remark that we used $(\eta^\delta_t,\,v^\delta_t)$ to replace $(\eta'^\delta\,v'^\delta)$ when we defined $E_{21}^\delta(t)$ and $E_{22}^\delta(t)$. Using \eqref{ansatz for app} and the properties of $\mathcal{J}_\delta$ in Lemma \ref{lemma for mollifier}, we have
\beno\begin{aligned}
E^\delta(t)\sim\widetilde{E^\delta}(t)&\underset{\text{def}}{=}|\eta^\delta|_{X^2_\epsilon}^2
+|\eta^\delta_t|_{X^0_\epsilon}^2+|\mathcal{J}_\delta\eta^\delta_{tx}|_{X^1_{\epsilon^2}}^2+|\eta^\delta_{tt}|_{X^0_\epsilon}^2\\
&\qquad
+|v^\delta|_{H^2}^2+|v^\delta_t|_2^2
+|\mathcal{J}_\delta v^\delta_{tx}|_{X^0_{\epsilon}}^2+|v^\delta_{tt}|_2^2.
\end{aligned}\eeno
 We also define the full energy functional as follows
\begin{equation}\label{expression for energy functional}\begin{aligned}
\mathcal{E}^\delta(t)&=\widetilde{E^\delta}(t)+|\eta^\delta_{tx}|_2^2+|v^\delta_{tx}|_2^2+\epsilon^3|\mathcal{J}_\delta^2\eta^\delta_{xxxxx}|_2^2
+\epsilon^2|\mathcal{J}_\delta^2v^\delta_{xxxx}|_2^2,\\
&\sim|\eta^\delta|_{X^2_\epsilon}^2+|\eta^\delta_t|_{H^1}^2
+|\mathcal{J}_\delta\eta^\delta_{tx}|_{X^0_{\epsilon^2}}^2+|\eta^\delta_{tt}|_{X^0_\epsilon}^2
+\epsilon^3|\mathcal{J}_\delta^2\eta^\delta_{xxxxx}|_2^2\\
&\quad
+|v^\delta|_{H^2}^2+|v^\delta_t|_{H^1}^2
+|\mathcal{J}_\delta v^\delta_{tx}|_{X^0_{\epsilon}}^2+|v^\delta_{tt}|_2^2
+\epsilon^2|\mathcal{J}_\delta^2v^\delta_{xxxx}|_2^2.
\end{aligned}\end{equation}
We remark that the mollifier  for the highest order derivatives of $\eta^\delta$ and $v^\delta$ is $\mathcal{J}_\delta^2$ not $\mathcal{J}_\delta$.

Now, we prove
\begin{equation}\label{equivalent}
\mathcal{E}^\delta(t)\sim\widetilde{E^\delta}(t)\sim E^\delta(t).
\end{equation}
To obtain \eqref{equivalent}, we only need to control $|\eta^\delta_{tx}|_2^2$, $|v^\delta_{tx}|_2^2$, $\epsilon^3|\mathcal{J}_\delta^2\eta^\delta_{xxxxx}|_2^2$,
$\epsilon^2|\mathcal{J}_\delta^2v^\delta_{xxxx}|_2^2$ by  $\widetilde{E^\delta}(t)$.

In what follows, we always assume that on the existence time interval
\begin{equation}\label{ansatz b}
\mathcal{E}^\delta(t)\leq C(\mathcal{E}^\delta(0))\mathcal{E}^\delta(0),
\end{equation}
where $C(\mathcal{E}^\delta(0))$ is a constant depending on $\mathcal{E}^\delta(0)$
 and in what follows, we shall use $C(\lambda_1,\lambda_2,\cdots)$ to denote  constants depending on $\lambda_1,\lambda_2,\cdots$.

Firstly, thanks to  \eqref{approximate system}, we have
\beno
|\eta^\delta_{tx}|_2=|\mathcal{J}_\delta v^\delta_{xx}|_2\leq  |v^\delta_{xx}|_2,
\eeno
and
\beno\begin{aligned}
&|v^\delta_{tx}|_2\lesssim|\Bigl((1+\epsilon\mathcal{J}_\delta\eta^\delta)(1-\epsilon\partial_x^2)\mathcal{J}_\delta\eta^\delta_x\Bigr)_x|_2
    +\epsilon|\mathcal{J}_\delta^2\Bigl(\frac{|\mathcal{J}_\delta^2 v^\delta|^2}{1+\epsilon\mathcal{J}_\delta\eta^\delta}\Bigr)_{xx}|_2\\
&\leq C|\mathcal{J}_\delta\eta^\delta_{xx}|_2+C\epsilon|\mathcal{J}_\delta\eta^\delta_{xxxx}|_2
+C\epsilon|\mathcal{J}_\delta\eta^\delta|_\infty\bigl(|\mathcal{J}_\delta\eta^\delta_x|_2+\epsilon|\mathcal{J}_\delta\eta^\delta_{xxxx}|_2\bigr)\\
&\qquad+ C(|\mathcal{J}_\delta^2v^\delta|_{W^{1,\infty}},|\mathcal{J}_\delta\eta^\delta_x|_\infty)\bigl(|\mathcal{J}_\delta^2v^\delta_x|_2
+|\mathcal{J}_\delta^2v^\delta_{xx}|_2+|\mathcal{J}_\delta\eta^\delta_x|_2+|\mathcal{J}_\delta\eta^\delta_{xx}|_2\bigr)\\
&\leq  C(|\mathcal{J}_\delta^2v^\delta|_{W^{1,\infty}},|\mathcal{J}_\delta\eta^\delta|_{W^{1,\infty}})
\bigl(|\eta^\delta|_{X^2_\epsilon}+|v^\delta|_{H^2}+\epsilon|\mathcal{J}_\delta\eta^\delta_{xxxx}|_2\bigr)\\
&\leq  C(|v^\delta|_{H^2},|\eta^\delta|_{H^2})
\bigl(|\eta^\delta|_{X^2_\epsilon}+|v^\delta|_{H^2}+\epsilon|\mathcal{J}_\delta\eta^\delta_{xxxx}|_2\bigr).
\end{aligned}\eeno
For the last term $\epsilon|\mathcal{J}_\delta\eta^\delta_{xxxx}|_2$ in the above inequality, we infer by Plancherel theorem that
\begin{equation}\label{interpolation}
\epsilon|\mathcal{J}_\delta\eta^\delta_{xxxx}|_2
\lesssim\epsilon|\eta^\delta_{xxx}|_2^{\frac{1}{2}}|\mathcal{J}_\delta^2\eta^\delta_{xxxxx}|_2^{\frac{1}{2}}
\lesssim\epsilon^{\frac{1}{2}}|\eta^\delta_{xxx}|_2+\epsilon^{\frac{3}{2}}|\mathcal{J}_\delta^2\eta^\delta_{xxxxx}|_2.
\end{equation}
Then we obtain
\begin{equation}\label{control 7}
|v^\delta_{tx}|_2\leq C(|v^\delta|_{H^2},|\eta^\delta|_{H^2})
\bigl(|\eta^\delta|_{X^2_\epsilon}+|v^\delta|_{H^2}+\epsilon^{\frac{3}{2}}|\mathcal{J}_\delta^2\eta^\delta_{xxxxx}|_2\bigr).
\end{equation}
Thus, we only need to control $\epsilon^3|\mathcal{J}_\delta^2\eta^\delta_{xxxxx}|_2^2
$ and $\epsilon^2|\mathcal{J}_\delta^2v^\delta_{xxxx}|_2^2$ by  $\widetilde{E^\delta}(t)$.

By virtue of \eqref{approximate system}, we have
\beno
\mathcal{J}_\delta v^\delta_x=-\eta^\delta_t,\quad
(1-\epsilon\partial_x^2)\mathcal{J}_\delta\eta^\delta_x=-\frac{v^\delta_t}{1+\epsilon\mathcal{J}_\delta\eta^\delta}
-\frac{\epsilon}{1+\epsilon\mathcal{J}_\delta\eta^\delta}\mathcal{J}_\delta^2
\Bigl(\frac{|\mathcal{J}_\delta^2v^\delta|^2}{1+\epsilon\mathcal{J}_\delta\eta^\delta}\Bigr)_x.
\eeno
Then we have
\begin{equation}\label{control 1}
\epsilon|\mathcal{J}_\delta^2 v^\delta_{xxxx}|_2=\epsilon|\mathcal{J}_\delta\eta^\delta_{txxx}|_2\leq C\bigl(\widetilde{E^\delta}(t)\bigr)^{\frac{1}{2}},
\end{equation}
and
\begin{equation}\label{control 5}\begin{aligned}
&\epsilon^{\frac{3}{2}}|\mathcal{J}_\delta^2\eta^\delta_{xxxxx}|_2\leq\epsilon^{\frac{1}{2}}|\mathcal{J}_\delta^2\eta^\delta_{xxx}|_2
+C\epsilon^{\frac{1}{2}}|\mathcal{J}_\delta\Bigl(\frac{v^\delta_t}{1+\epsilon\mathcal{J}_\delta\eta^\delta}\Bigr)_{xx}|_2\\
&\qquad
+C\epsilon^{\frac{3}{2}}|\mathcal{J}_\delta\Bigl(\frac{1}{1+\epsilon\mathcal{J}_\delta\eta^\delta}\mathcal{J}_\delta^2
\Bigl(\frac{|\mathcal{J}_\delta^2v^\delta|^2}{1+\epsilon\mathcal{J}_\delta\eta^\delta}\Bigr)_x\Bigr)_{xx}|_2\\
&\quad\underset{\text{def}}{=}\epsilon^{\frac{1}{2}}|\mathcal{J}_\delta^2\eta^\delta_{xxx}|_2+C(A_1+ A_2).
\end{aligned}\end{equation}
For $A_1$, we obtain
\beno\begin{aligned}
A_1&\lesssim\epsilon^{\frac{1}{2}}|\mathcal{J}_\delta v^\delta_{txx}|_2+\epsilon^{\frac{1}{2}}|[\partial_x^2\mathcal{J}_\delta,\frac{1}{1+\epsilon\mathcal{J}_\delta\eta^\delta}]v^\delta_t|_2.
\end{aligned}\eeno
Since
\beno
[\partial_x^2\mathcal{J}_\delta, a]f=[\mathcal{J}_\delta, a]\partial_x^2f+\mathcal{J}_\delta\bigl(\partial_x^2af+\partial_xa\partial_xf\bigr),
\eeno
using  Lemma \ref{lemma for mollifier} and H\"older inequality, we obtain
\beno\begin{aligned}
&[\partial_x^2\mathcal{J}_\delta,\frac{1}{1+\epsilon\mathcal{J}_\delta\eta^\delta}]v^\delta_t|_2
\leq C|\frac{\epsilon\mathcal{J}_\delta\eta^\delta}{1+\epsilon\mathcal{J}_\delta\eta^\delta}|_{H^2}|v^\delta_{tx}|_2\\
&\quad+|\bigl(\frac{1}{1+\epsilon\mathcal{J}_\delta\eta^\delta}\bigr)_{xx}|_2|v^\delta_t|_\infty
+|\bigl(\frac{1}{1+\epsilon\mathcal{J}_\delta\eta^\delta}\bigr)_{x}|_\infty|v^\delta_{tx}|_2\\
&\leq C(|\mathcal{J}_\delta\eta^\delta|_\infty)\epsilon|\mathcal{J}_\delta\eta^\delta|_{H^2}(|v^\delta_t|_2+|v^\delta_{tx}|_2).
\end{aligned}\eeno
Then we have
\begin{equation}\label{control 2}
|A_1|\leq C(|\mathcal{J}_\delta\eta^\delta|_{H^2})(|\mathcal{J}_\delta v^\delta_{tx}|_{X^0_\epsilon}+|v^\delta_t|_2+\epsilon^{\frac{3}{2}}|v^\delta_{tx}|_2).
\end{equation}
By virtue of \eqref{control 7} and \eqref{control 2}, we have
\begin{equation}\label{control 3}
|A_1|\leq C(|v^\delta|_{H^2},|\eta^\delta|_{H^2})\widetilde{E^\delta}(t)
+\epsilon^{\frac{3}{2}}C(|v^\delta|_{H^2},|\eta^\delta|_{H^2})\cdot\epsilon^{\frac{3}{2}}|\mathcal{J}_\delta^2\eta^\delta_{xxxxx}|_2.
\end{equation}

For $A_2$, by the product estimates and interpolation estimates \eqref{int1}, we have
\beno\begin{aligned}
|A_2|&\leq \epsilon^{\frac{3}{2}}C(|\mathcal{J}_\delta^2v^\delta|_{W^{1,\infty}},|\mathcal{J}_\delta\eta^\delta|_{W^{1,\infty}})
\bigl(\epsilon|\mathcal{J}_\delta\eta^\delta_x|_{H^2}+|\mathcal{J}_\delta^2v^\delta_x|_{H^2}\bigr)\\
&\leq C(|v^\delta|_{H^2},|\eta^\delta|_{H^2})\bigl(|\eta^\delta|_{X^2_\epsilon}+|v^\delta|_{H^2}+\epsilon|\mathcal{J}_\delta^2v^\delta_{xxxx}|_2\bigr),
\end{aligned}\eeno
which along with \eqref{control 1} implies that
\begin{equation}\label{control 4}
|A_2|\leq C(|v^\delta|_{H^2},|\eta^\delta|_{H^2})\widetilde{E^\delta}(t).
\end{equation}

Thanks to \eqref{control 5}, \eqref{control 3} and \eqref{control 4}, we obtain
\beno
\epsilon^{\frac{3}{2}}|\mathcal{J}_\delta^2\eta^\delta_{xxxxx}|_2\leq C(|v^\delta|_{H^2},|\eta^\delta|_{H^2})\widetilde{E^\delta}(t)
+\epsilon^{\frac{3}{2}}C(|v^\delta|_{H^2},|\eta^\delta|_{H^2})\cdot\epsilon^{\frac{3}{2}}|\mathcal{J}_\delta^2\eta^\delta_{xxxxx}|_2.
\eeno
By virtue of \eqref{ansatz b}, for $\epsilon$ sufficiently small (depending on $\mathcal{E}^\delta(0)$), we have
\begin{equation}\label{control 6}
\epsilon^{\frac{3}{2}}|\mathcal{J}_\delta^2\eta^\delta_{xxxxx}|_2\leq C(\mathcal{E}^\delta(0))\widetilde{E^\delta}(t).
\end{equation}

Thus, combining \eqref{control 1} and \eqref{control 6}, we obtain the equivalence \eqref{equivalent}.

\smallskip

{\it Step 2.3. Uniform energy estimates for $V^\delta=(\eta^\delta,\,v^\delta)$.}

Motivated by the a priori energy estimates \eqref{energy estimate 1} for \eqref{reduction form 1}, we obtain
\begin{equation}\label{estimate for app 1}
\frac{d}{dt} E^\delta(t)\leq C(\mathcal{E}^\delta(t))\epsilon\mathcal{E}^\delta(t)^{\frac{3}{2}},
\end{equation}
where $C(\mathcal{E}^\delta(t))$ is a constant only depending on $\mathcal{E}^\delta(t)$. The derivation of \eqref{estimate for app 1} is a little different from \eqref{energy estimate 1}.

\smallskip

{\it (i) Estimates for $E_0^\delta(t)$.}
We first derive the estimates for $E_0^\delta(t)$. As usual, we have
\beno\begin{aligned}
&\frac{1}{2}\frac{d}{dt}E_{0k}^\delta=(\partial^k\eta^\delta_t\,|\,\partial^k\eta^\delta)_2
+\epsilon(\partial^{k+1}\eta^\delta_{t}\,|\,\partial^{k+1}\eta^\delta)_2\\
&\qquad+(\frac{1}{1+\epsilon\mathcal{J}_\delta\eta^\delta}\partial^kv^\delta_t
\,|\,\partial^kv^\delta)_2+\frac{1}{2}(\bigl(\frac{1}{1+\epsilon\mathcal{J}_\delta\eta^\delta}\bigr)_t\partial^kv^\delta
\,|\,\partial^kv^\delta)_2.
\end{aligned}\eeno
Using the equations in \eqref{approximate system}, we obtain
\beno\begin{aligned}
&\frac{1}{2}\frac{d}{dt}E_{0k}^\delta=([\partial^k,\frac{1}{1+\epsilon\mathcal{J}_\delta\eta^\delta}]v^\delta_t\,|\,\partial^kv^\delta)_2
-\epsilon(\partial^k\Bigl(\frac{1}{1+\epsilon\mathcal{J}_\delta\eta^\delta}\mathcal{J}_\delta^2
\Bigl(\frac{|\mathcal{J}_\delta^2v^\delta|^2}{1+\epsilon\mathcal{J}_\delta\eta^\delta}\Bigr)_x\Bigr)
\,|\,\partial^kv^\delta)_2\\
&\qquad\qquad+\frac{1}{2}(\bigl(\frac{1}{1+\epsilon\mathcal{J}_\delta\eta^\delta}\bigr)_t\partial^kv^\delta
\,|\,\partial^kv^\delta)_2\underset{\text{def}}{=}B_1^k+B_2^k+B_3^k.
\end{aligned}\eeno

For $B_1^k$, we have $B_1^0=0$ and
\beno\begin{aligned}
&|B_1^1|+|B_1^2|\leq\sum_{k=1}^2|[\partial^k,\frac{1}{1+\epsilon\mathcal{J}_\delta\eta^\delta}]v^\delta_t|_2|\partial^kv^\delta|_2\\
&\leq C\epsilon(1+|\mathcal{J}_\delta\eta^\delta_x|_\infty)\bigl(|\mathcal{J}_\delta\eta^\delta_x|_\infty|v^\delta_t|_{H^1}+
|\mathcal{J}_\delta\eta^\delta_x|_{H^1}|v^\delta_t|_\infty\bigr)|v^\delta|_{H^2}\\
&\leq C(|\eta^\delta|_{H^2})\epsilon|\eta^\delta|_{H^2}|v^\delta_t|_{H^1}|v^\delta|_{H^2}.
\end{aligned}\eeno

For $B_3^k$, we have
\beno
\sum_{k=0}^2|B_3^k|\leq C\epsilon|\mathcal{J}_\delta\eta^\delta_t|_\infty|v^\delta|_{H^2}^2\leq C\epsilon|\eta^\delta_t|_{H^1}|v^\delta|_{H^2}^2.
\eeno

For $B_2^k$, we have

\beno\begin{aligned}
&B_2^k=-\epsilon([\partial^k,\frac{1}{1+\epsilon\mathcal{J}_\delta\eta^\delta}]\mathcal{J}_\delta^2
\Bigl(\frac{|\mathcal{J}_\delta^2v^\delta|^2}{1+\epsilon\mathcal{J}_\delta\eta^\delta}\Bigr)_x
\,|\,\partial^kv^\delta)_2\\
&\qquad
-\epsilon(\partial^k
\Bigl(\frac{|\mathcal{J}_\delta^2v^\delta|^2}{1+\epsilon\mathcal{J}_\delta\eta^\delta}\Bigr)_x
\,|\,\mathcal{J}_\delta^2\Bigl(\frac{\partial^kv^\delta}{1+\epsilon\mathcal{J}_\delta\eta^\delta}\Bigr))_2
\underset{\text{def}}{=}B_{21}^k+B_{22}^k
\end{aligned}\eeno
Similar to $B_1^k$, we have $B_{21}^0=0$ and
\beno\begin{aligned}
|B_{21}^1|+|B_{21}^2|&\leq C(|\eta^\delta|_{H^2})\epsilon|\eta^\delta|_{H^2}|\mathcal{J}_\delta^2
\Bigl(\frac{|\mathcal{J}_\delta^2v^\delta|^2}{1+\epsilon\mathcal{J}_\delta\eta^\delta}\Bigr)_x|_{H^1}|v^\delta|_{H^2}\\
&\leq C(|\eta^\delta|_{H^2},|v^\delta|_{H^2})\epsilon\bigl(|\eta^\delta|_{H^2}^2+|v^\delta|_{H^2}^2\bigr)|v^\delta|_{H^2}.
\end{aligned}\eeno
For $B_{22}^k$, we have
\beno\begin{aligned}
B_{22}^k=&\epsilon(\partial^k
\Bigl(\frac{|\mathcal{J}_\delta^2v^\delta|^2}{1+\epsilon\mathcal{J}_\delta\eta^\delta}\Bigr)
\,|\,\partial_x\Bigl([\mathcal{J}_\delta^2,\frac{1}{1+\epsilon\mathcal{J}_\delta\eta^\delta}]\partial^kv^\delta\Bigr))_2\\
&
+\epsilon^2(\partial^k
\Bigl(\frac{|\mathcal{J}_\delta^2v^\delta|^2\mathcal{J}_\delta\eta^\delta_x}{(1+\epsilon\mathcal{J}_\delta\eta^\delta)^2}\Bigr)
\,|\,\frac{\mathcal{J}_\delta^2\partial^kv^\delta}{1+\epsilon\mathcal{J}_\delta\eta^\delta})_2\\
&
-2\epsilon([\partial^k,\frac{\mathcal{J}_\delta^2v^\delta}{1+\epsilon\mathcal{J}_\delta\eta^\delta}]
\mathcal{J}_\delta^2v^\delta_x
\,|\,\frac{\mathcal{J}_\delta^2\partial^kv^\delta}{1+\epsilon\mathcal{J}_\delta\eta^\delta})_2\\
&-2\epsilon(\frac{\mathcal{J}_\delta^2v^\delta}{1+\epsilon\mathcal{J}_\delta\eta^\delta}\cdot\partial_x
\mathcal{J}_\delta^2\partial^kv^\delta
\,|\,\frac{\mathcal{J}_\delta^2\partial^kv^\delta}{1+\epsilon\mathcal{J}_\delta\eta^\delta})_2.
\end{aligned}\eeno
Notice that the last term of the above equality equals
\beno
\epsilon(\partial_x\Bigl(\frac{\mathcal{J}_\delta^2v^\delta}{1+\epsilon\mathcal{J}_\delta\eta^\delta}\Bigr)\cdot
\mathcal{J}_\delta^2\partial^kv^\delta
\,|\,\frac{\mathcal{J}_\delta^2\partial^kv^\delta}{1+\epsilon\mathcal{J}_\delta\eta^\delta})_2
\eeno
Then, proceeding  as for the previous terms, using Lemma \ref{lemma for mollifier} and product estimates, we have
\beno\begin{aligned}
&\sum_{k=0}^2|B_{22}^k|\leq C(|\eta^\delta|_{X^2_\epsilon},|v^\delta|_{H^2})\epsilon\bigl(|\eta^\delta|_{H^2}^2+|v^\delta|_{H^2}^2\bigr)|v^\delta|_{H^2}.
\end{aligned}\eeno

Combining all the above estimates, we obtain
\begin{equation}\label{estimate for app 0}
\frac{d}{dt}E_0^\delta(t)\leq C(|\eta^\delta|_{X^2_\epsilon},|v^\delta|_{H^2})\epsilon\bigl(\mathcal{E}^\delta(t)\bigr)^{\frac{3}{2}}.
\end{equation}

\smallskip

{\it (ii) Estimates for $E_2^\delta(t)$.}
Now, we derive the energy estimates for \eqref{quasilinear for app 1}. For the second equation of \eqref{quasilinear for app 1}, taking $L^2$ inner product with $v'^\delta_t$ yields
\beno\begin{aligned}
&\frac{1}{2}\frac{d}{dt}E_{22}^\delta+\underbrace{2\epsilon\bigl(
\mathcal{J}_\delta^2v^\delta\cdot\partial_x\mathcal{J}_\delta^2\bigl(\frac{v'^\delta_t}{1+\epsilon\mathcal{J}_\delta\eta^\delta}
\bigr)\,|\,\mathcal{J}_\delta^2\bigl(\frac{v'^\delta_t}{1+\epsilon\mathcal{J}_\delta\eta^\delta}
\bigr)\bigr)_2}_{I}\\
&\quad=\frac{1}{2}(\partial_t\bigl(\frac{1}{1+\epsilon\mathcal{J}_\delta\eta^\delta}\bigr)v'^\delta_{t}\,|\,v'^\delta_{t})_2+(g'^\delta\,|\,v'^\delta_t)_2
\end{aligned}\eeno
By integration by parts, we have
\beno
I=-\epsilon\bigl(
\mathcal{J}_\delta^2v^\delta_x\cdot\mathcal{J}_\delta^2\bigl(\frac{v'^\delta_t}{1+\epsilon\mathcal{J}_\delta\eta^\delta}
\bigr)\,|\,\mathcal{J}_\delta^2\bigl(\frac{v'^\delta_t}{1+\epsilon\mathcal{J}_\delta\eta^\delta}
\bigr)\bigr)_2.
\eeno
Then we obtain
\begin{equation}\label{estimate for app 2}\begin{aligned}
\frac{d}{dt}E_{22}^\delta&\leq C\epsilon\bigl(|\mathcal{J}_\delta\eta^\delta_t|_\infty+|\mathcal{J}_\delta^2v^\delta_x|_\infty\bigr)
\bigl(|v'^\delta_t|_2^2+|\mathcal{J}_\delta^2\bigl(\frac{v'^\delta_t}{1+\epsilon\mathcal{J}_\delta\eta^\delta}\bigr)|_2^2\bigr)
+|g'^\delta|_2|v'^\delta_t|_2\\
&\leq C\epsilon\bigl(|\eta^\delta_t|_{H^1}+|v^\delta|_{H^2}\bigr)|v'^\delta_t|_2^2
+|g'^\delta|_2|v'^\delta_t|_2.
\end{aligned}\end{equation}

For the first equation of \eqref{quasilinear for app 1}, taking $L^2$ inner product by $(1-\epsilon\partial_x^2)\eta'^\delta_t$ results
\beno\begin{aligned}
&\frac{1}{2}\frac{d}{dt}E_{21}^\delta+\underbrace{2\epsilon\bigl(\frac{
\mathcal{J}_\delta^2v^\delta}{1+\epsilon\mathcal{J}_\delta\eta^\delta}\cdot\partial_x\mathcal{J}_\delta^2\eta'^\delta_t\,|\,
(1-\epsilon\partial_x^2)\mathcal{J}_\delta^2\eta'^\delta_t\bigr)_2}_{II}\\
&\quad=\frac{\epsilon}{2}(\mathcal{J}_\delta\eta^\delta_t\mathcal{J}_\delta\eta'^\delta_x\,|\,\mathcal{J}_\delta\eta'^\delta_x)_2
+\epsilon^2(\mathcal{J}_\delta\eta^\delta_t\mathcal{J}_\delta\eta'^\delta_{xx}\,|\,\mathcal{J}_\delta\eta'^\delta_{xx})_2
+\epsilon^2(\mathcal{J}_\delta\eta^\delta_{xx}\mathcal{J}_\delta\eta'^\delta_x\,|\,\mathcal{J}_\delta\eta'^\delta_{tx})_2\\
&\qquad
+\frac{\epsilon^3}{2}(\mathcal{J}_\delta\eta^\delta_t\mathcal{J}_\delta\eta'^\delta_{xxx}\,|\,\mathcal{J}_\delta\eta'^\delta_{xxx})_2
+(f'^\delta\,|\,(1-\epsilon\partial_x^2)\eta'^\delta_t)_2.
\end{aligned}\eeno
By integration by parts, we have
\beno\begin{aligned}
II&=-\epsilon\bigl(\partial_x\bigl(\frac{
\mathcal{J}_\delta^2v^\delta}{1+\epsilon\mathcal{J}_\delta\eta^\delta}\bigr)\cdot\mathcal{J}_\delta^2\eta'^\delta_t\,|\,
(1-\epsilon\partial_x^2)\mathcal{J}_\delta^2\eta'^\delta_t\bigr)_2\\
&\quad
+\epsilon^2\bigl(\partial_x\bigl(\frac{
\mathcal{J}_\delta^2v^\delta}{1+\epsilon\mathcal{J}_\delta\eta^\delta}\bigr)\cdot\mathcal{J}_\delta^2\eta'^\delta_{tx}\,|\,
(1-\epsilon\partial_x^2)\mathcal{J}_\delta^2\eta'^\delta_{tx}\bigr)_2\\
&\leq C\epsilon\bigl(|\mathcal{J}_\delta^2v^\delta_x|_{\infty}
+\epsilon|\mathcal{J}_\delta^2v^\delta|_{\infty}|\mathcal{J}_\delta\eta^\delta_x|_{\infty}\bigr)
\bigl(|\mathcal{J}_\delta^2\eta'^\delta_t|_2^2+\epsilon|\mathcal{J}_\delta^2\eta'^\delta_{tx}|_2^2\bigr).
\end{aligned}\eeno
Then we get
\begin{equation}\label{estimate for app 3}\begin{aligned}
&\frac{d}{dt}E_{21}^\delta\leq C\epsilon\bigl(|\mathcal{J}_\delta^2v^\delta_x|_{\infty}
+\epsilon|\mathcal{J}_\delta^2v^\delta|_{\infty}|\mathcal{J}_\delta\eta^\delta_x|_{\infty}
+|\mathcal{J}_\delta\eta^\delta_t|_\infty+\epsilon^{\frac{1}{2}}|\mathcal{J}_\delta\eta'^\delta_{x}|_\infty\bigr)\\
&\quad\times
\bigl(|\mathcal{J}_\delta^2\eta'^\delta_t|_{X^0_\epsilon}^2+|\mathcal{J}_\delta\eta'^\delta_x|_{X^0_{\epsilon^2}}^2
+\epsilon|\mathcal{J}_\delta\eta^\delta_{xx}|_2^2\bigr)
+|f'^\delta|_2|\eta'^\delta_t|_2+\epsilon|f'^\delta_x|_2|\eta'^\delta_{tx}|_2.
\end{aligned}\end{equation}

Due to \eqref{estimate for app 2} and \eqref{estimate for app 3}, noticing that $\eta'^\delta=\eta_t^\delta$, we obtain
\begin{equation}\label{estimate for app 4}\begin{aligned}
\frac{d}{dt}E_2^\delta&\leq C(|\eta^\delta|_{H^2})\epsilon\bigl(|v^\delta|_{H^2}
+|\eta^\delta_t|_{H^1}+|\mathcal{J}_\delta\eta^\delta_{tx}|_{X^0_{\epsilon^2}}\bigr)\\
&\quad\times
\bigl(|\mathcal{J}_\delta^2\eta'^\delta_t|_{X^0_\epsilon}^2+|\mathcal{J}_\delta\eta'^\delta_x|_{X^0_{\epsilon^2}}^2
+|\mathcal{J}_\delta\eta^\delta_x|_{X^0_{\epsilon^2}}^2+|v'^\delta_t|_2^2\bigr)\\
&\quad
+|f'^\delta|_2|\eta'^\delta_t|_2+\epsilon|f'^\delta_x|_2|\eta'^\delta_{tx}|_2+|g'^\delta|_2|v'^\delta_t|_2.
\end{aligned}\end{equation}

Noticing that $\eta'^\delta=\eta^\delta_t$ and $v'^\delta=v^\delta_t$. Then \eqref{estimate for app 4} implies that
\begin{equation}\label{estimate for app 6}\begin{aligned}
\frac{d}{dt}E_2^\delta&\leq C(\mathcal{E}^\delta(t))\epsilon\bigl(\mathcal{E}^\delta(t)\bigr)^{\frac{3}{2}}
+|f'^\delta|_2|\eta^\delta_{tt}|_2+\epsilon|f'^\delta_x|_2|\eta^\delta_{ttx}|_2+|g'^\delta|_2|v^\delta_{tt}|_2.
\end{aligned}\end{equation}

\smallskip

{\it (iii) Estimates for $E_1^\delta$.}
Similarly as $E_2^\delta$, we also obtain
\begin{equation}\label{estimate for app 7a}\begin{aligned}
\frac{d}{dt}E_1^\delta&\leq C(\mathcal{E}^\delta(t))\epsilon\bigl(\mathcal{E}^\delta(t)\bigr)^{\frac{3}{2}}
+|f^\delta|_2|\eta^\delta_{t}|_2+\epsilon|f^\delta_x|_2|\eta^\delta_{tx}|_2+|g^\delta|_2|v^\delta_{t}|_2.
\end{aligned}\end{equation}

\smallskip

{\it (iv) Estimates for the source terms.}
To achieve \eqref{estimate for app 1}, it remains to derive the bound on $f^\delta$, $g^\delta$, $f'^\delta$ and $g'^\delta$. Similarly as in the derivation of \eqref{lte22}, using the properties of $\mathcal{J}_\delta$ in Lemma \ref{lemma for mollifier}, we obtain
\begin{equation}\label{estimate 2}\begin{aligned}
&|f^\delta|_2+\epsilon^{\frac{1}{2}}|f^\delta_x|_2+|g^\delta|_2+|f'^\delta|_2+\epsilon^{\frac{1}{2}}|f'^\delta_x|_2+|g'^\delta|_2\\
&\lesssim  C(\mathcal{E}^\delta(t))\epsilon\bigl(|\eta^\delta|_{X^2_\epsilon}+|v^\delta|_{H^2}
+|\eta^\delta_t|_{H^1}+|v^\delta_t|_{H^1}
+|\mathcal{J}_\delta\eta^\delta_{tx}|_{X^0_{\epsilon^2}}+|\mathcal{J}_\delta^2v^\delta_{tx}|_{X^0_\epsilon}\\
&\quad
+\epsilon|\mathcal{J}_\delta\eta^\delta_{xxxx}|_2+\epsilon^{\frac{1}{2}}|\mathcal{J}_\delta^2v^\delta_{xxx}|_2\bigr)
\bigl(|\eta^\delta|_{X^2_\epsilon}
+\epsilon|\mathcal{J}_\delta\eta^\delta_{xxxx}|_2+\epsilon^{\frac{3}{2}}|\mathcal{J}_\delta^2\eta^\delta_{xxxxx}|_2\\
&\quad+|\eta^\delta_t|_{H^1}+|\mathcal{J}_\delta\eta^\delta_{tx}|_{X^0_{\epsilon^2}}+|\eta^\delta_{tt}|_{X^0_\epsilon}
 +|v^\delta|_{H^2}
+|v^\delta_t|_{H^1}+|\mathcal{J}_\delta^2v^\delta_{tx}|_{X^0_{\epsilon}}
 +|v^\delta_{tt}|_2\bigr).
\end{aligned}\end{equation}
To verify \eqref{estimate 2}, we first need to check the estimates on the terms involving the commutators in $f^\delta$ and $g^\delta$. Using (v) of Lemma \ref{lemma for mollifier} with $t_0=1$, for the commutator term in $f^\delta$, we have
\begin{equation}\label{estimate 3}\begin{aligned}
&|\epsilon\mathcal{J}_\delta^2\bigl([\mathcal{J}_\delta,\frac{\mathcal{J}_\delta^2v^\delta}{1+\epsilon\mathcal{J}_\delta\eta^\delta}]
\partial_x\mathcal{J}_\delta\eta^\delta_t\bigr)|_{X^0_\epsilon}\leq C\epsilon|\frac{\mathcal{J}_\delta^2v^\delta}{1+\epsilon\mathcal{J}_\delta\eta^\delta}|_{H^2}
|\mathcal{J}_\delta\eta^\delta_t|_{X^0_\epsilon}\\
&\leq C(|\eta^\delta|_{H^2},|v^\delta|_{H^2})\epsilon\bigl(|\eta^\delta|_{H^2}+|v^\delta|_{H^2}\bigr)
|\eta^\delta_t|_{H^1}.
\end{aligned}\end{equation}
For the commutator term in $f'^\delta$, we have
\beno\begin{aligned}
&|\epsilon\mathcal{J}_\delta^2\partial_t\bigl([\mathcal{J}_\delta,\frac{\mathcal{J}_\delta^2v^\delta}{1+\epsilon\mathcal{J}_\delta\eta^\delta}]
\partial_x\mathcal{J}_\delta\eta^\delta_t\bigr)|_{X^0_\epsilon}\\
&\leq\epsilon|[\mathcal{J}_\delta,\partial_t\bigl(\frac{\mathcal{J}_\delta^2v^\delta}{1+\epsilon\mathcal{J}_\delta\eta^\delta}\bigr)]
\partial_x\mathcal{J}_\delta\eta^\delta_t\bigr)|_{X^0_\epsilon}
+\epsilon|[\mathcal{J}_\delta,\frac{\mathcal{J}_\delta^2v^\delta}{1+\epsilon\mathcal{J}_\delta\eta^\delta}]
\partial_x\mathcal{J}_\delta\eta^\delta_{tt}\bigr)|_{X^0_\epsilon}\\
&\underset{\text{def}}{=}III_1+III_2.
\end{aligned}\eeno
For $III_1$, by using the product estimates, we have
\beno\begin{aligned}
III_1&\leq C(|\mathcal{J}_\delta\eta^\delta_x|_\infty,|\mathcal{J}_\delta^2v^\delta|_\infty)
\epsilon\bigl(|\mathcal{J}_\delta\eta^\delta_t|_\infty+|\mathcal{J}_\delta^2v^\delta_t|_\infty
+\epsilon^{\frac{1}{2}}|\mathcal{J}_\delta\eta^\delta_{tx}|_\infty\\
&\qquad
+\epsilon^{\frac{1}{2}}|\mathcal{J}_\delta^2v^\delta_{tx}|_\infty\bigr)
|\mathcal{J}_\delta\eta^\delta_{tx}|_{X^0_\epsilon}.
\end{aligned}\eeno
For $III_2$, by similar derivation as \eqref{estimate 3}, we have
\beno
III_2\leq C(|\eta^\delta|_{H^2},|v^\delta|_{H^2})\epsilon\bigl(|\eta^\delta|_{H^2}+|v^\delta|_{H^2}\bigr)
|\mathcal{J}_\delta\eta^\delta_{tt}|_{X^0_\epsilon}.
\eeno
Then by  Sobolev inequality, we have
\begin{equation}\label{estimate 4}\begin{aligned}
&|\epsilon\mathcal{J}_\delta^2\partial_t\bigl([\mathcal{J}_\delta,\frac{\mathcal{J}_\delta^2v^\delta}{1+\epsilon\mathcal{J}_\delta\eta^\delta}]
\partial_x\mathcal{J}_\delta\eta^\delta_t\bigr)|_{X^0_\epsilon}
\leq  C(\mathcal{E}^\delta(t))\epsilon\bigl(|\eta^\delta|_{H^2}+|v^\delta|_{H^2}
\\
&\quad+|\eta^\delta_t|_{H^1}+|v^\delta_t|_{H^1}
+|\mathcal{J}_\delta\eta^\delta_{tx}|_{X^0_\epsilon}
+|\mathcal{J}_\delta^2v^\delta_{tx}|_{X^0_\epsilon}\bigr)
\bigl(|\mathcal{J}_\delta\eta^\delta_{tx}|_{X^0_\epsilon}+|\mathcal{J}_\delta\eta^\delta_{tt}|_{X^0_\epsilon}\bigr).
\end{aligned}\end{equation}

Similarly, for the commutator terms in $g^\delta$ and $g'^\delta$, noticing that
\beno
[\partial_x\mathcal{J}_\delta^2,a]f=\partial_x\mathcal{J}_\delta\bigl([\mathcal{J}_\delta,a]f\bigr)
+\partial_x\bigl([\mathcal{J}_\delta,a]\mathcal{J}_\delta f\bigr)+\partial_x a\mathcal{J}_\delta^2f,
\eeno
 we have
\begin{equation}\label{estimate 5}\begin{aligned}
&|\frac{2\epsilon}{1+\epsilon\mathcal{J}_\delta\eta^\delta}
\mathcal{J}_\delta^2\Bigl(\mathcal{J}_\delta^2v^\delta\cdot[\partial_x\mathcal{J}_\delta^2,\frac{1}{1+\epsilon\mathcal{J}_\delta\eta^\delta}]
v^\delta_t\Bigr)|_2\\
&\quad+|\partial_t\Bigl(\frac{2\epsilon}{1+\epsilon\mathcal{J}_\delta\eta^\delta}
\mathcal{J}_\delta^2\Bigl(\mathcal{J}_\delta^2v^\delta\cdot[\partial_x\mathcal{J}_\delta^2,\frac{1}{1+\epsilon\mathcal{J}_\delta\eta^\delta}]
v^\delta_t\Bigr)\Bigr)|_2\\
&\leq C(\mathcal{E}^\delta(t))\epsilon\bigl(|\eta^\delta|_{H^2}+|\eta^\delta_t|_{H^1}+|v^\delta_t|_{H^1}\bigr)
\bigl(|\mathcal{J}_\delta\eta^\delta_{tx}|_2+|v^\delta_t|_{H^1}+|v^\delta_{tt}|_2\bigr).
\end{aligned}\end{equation}

Another delicate term we have to check is the third terms of $f'^\delta$. Applying $\epsilon^{\frac{1}{2}}\partial_x$ to the third term, we have
\beno\begin{aligned}
&\epsilon^{\frac{5}{2}}\mathcal{J}_\delta\partial_x^2\bigl(\mathcal{J}_\delta\eta^\delta_t
\cdot\partial_x^3\mathcal{J}_\delta\eta^\delta\bigr)=\epsilon^{\frac{5}{2}}\mathcal{J}_\delta\eta^\delta_t
\cdot\partial_x^5\mathcal{J}_\delta^2\eta^\delta\\
&\quad+\epsilon^{\frac{5}{2}}[\mathcal{J}_\delta,\mathcal{J}_\delta\eta^\delta_t]
\cdot\partial_x^5\mathcal{J}_\delta\eta^\delta+\epsilon^{\frac{5}{2}}\mathcal{J}_\delta\bigl(\partial_x^2\mathcal{J}_\delta\eta^\delta_{t}
\cdot\partial_x^3\mathcal{J}_\delta\eta^\delta+2\partial_x\mathcal{J}_\delta\eta^\delta_{t}
\cdot\partial_x^4\mathcal{J}_\delta\eta^\delta\bigr).
\end{aligned}\eeno
Then by virtue of the properties of $\mathcal{J}_\delta$ in Lemma \ref{lemma for mollifier}, we have
\beno\begin{aligned}
&\epsilon^{\frac{5}{2}}|\mathcal{J}_\delta\partial_x^2\bigl(\mathcal{J}_\delta\eta^\delta_t
\cdot\partial_x^3\mathcal{J}_\delta\eta^\delta\bigr)|_2\leq C\epsilon\bigl(|\mathcal{J}_\delta\eta^\delta_t|_\infty+\epsilon|\mathcal{J}_\delta\eta^\delta_{txx}|_\infty
+\epsilon^{\frac{1}{2}}|\mathcal{J}_\delta\eta^\delta_{t}|_{H^2}\bigr)\\
&\quad\times\bigl(\epsilon^{\frac{1}{2}}|\mathcal{J}_\delta\eta^\delta_{xxx}|_2
+\epsilon|\mathcal{J}_\delta\eta^\delta_{xxxx}|_2+\epsilon^{\frac{3}{2}}|\mathcal{J}_\delta^2\eta^\delta_{xxxxx}|_2\bigr)\\
&\leq C\epsilon\bigl(|\eta^\delta_t|_{H^1}+|\mathcal{J}_\delta\eta^\delta_{tx}|_{X^0_{\epsilon^2}}\bigr)
\bigl(|\eta^\delta|_{X^2_\epsilon}
+\epsilon|\mathcal{J}_\delta\eta^\delta_{xxxx}|_2+\epsilon^{\frac{3}{2}}|\mathcal{J}_\delta^2\eta^\delta_{xxxxx}|_2\bigr).
\end{aligned}\eeno

\smallskip

{\it (v) Derivation of the energy estimates \eqref{estimate for app 1}.}
Similar to \eqref{interpolation}, by interpolation inequality, we have
\begin{equation}\label{estimate 6}\begin{aligned}
\epsilon^{\frac{1}{2}}|\mathcal{J}_\delta^2v^\delta_{xxx}|_2
\lesssim\epsilon^{\frac{1}{2}}|\mathcal{J}_\delta^2v^\delta_{xx}|_2^{\frac{1}{2}}|\mathcal{J}_\delta^2v^\delta_{xxxx}|_2^{\frac{1}{2}}
\lesssim|v^\delta_{xx}|_2+\epsilon|\mathcal{J}_\delta^2v^\delta_{xxxx}|_2.
\end{aligned}\end{equation}
With \eqref{interpolation} and \eqref{estimate 6}, we bound the righthand side terms in \eqref{estimate 2} by $C(\mathcal{E}^\delta)\epsilon\mathcal{E}^\delta(t)$. Thus, by virtue of \eqref{estimate for app 0}, \eqref{estimate for app 6}, \eqref{estimate for app 7a} and \eqref{estimate 2}, we achieve the proof of  \eqref{estimate for app 1}. Thanks to \eqref{equivalent}, under the assumptions \eqref{ansatz for app} and \eqref{ansatz b}, we have
\begin{equation}\label{estimate for app 7}
\frac{d}{dt} E^\delta(t)\leq C(\mathcal{E}^\delta(t))\epsilon\mathcal{E}^\delta(t)^{\frac{3}{2}}\leq C(\mathcal{E}^\delta(0))\epsilon E^\delta(t)^{\frac{3}{2}}.
\end{equation}

\smallskip

{\it (vi) Bound of the initial energy $\mathcal{E}^\delta(0)$.}
Proceeding as in the derivation of \eqref{regularity for initial data} in Step 5 of the proof to Theorem \ref{long time existence for case c=-1}, we  obtain
\begin{equation}\label{initial regularity for app}
\mathcal{E}^\delta(0)\leq C\bigl(|\eta^\delta|_{t=0}|_{X^2_{\epsilon^3}}^2+|v^\delta|_{t=0}|_{X^2_{\epsilon^2}}^2\bigr)
\leq C\bigl(|\eta_0|_{X^2_{\epsilon^3}}^2+|v_0|_{X^2_{\epsilon^2}}^2\bigr).
\end{equation}

With \eqref{estimate for app 7} and \eqref{initial regularity for app}, there exists  $T>0$ which  depends on $|\eta_0|_{X^2_{\epsilon^3}}+|v_0|_{X^2_{\epsilon^2}}$, such that
\begin{equation}\label{estimate for app 8}
\sup_{0\leq t\leq T/\epsilon}\mathcal{E}^\delta(t)\leq C\bigl(|\eta_0|_{X^2_{\epsilon^3}}^2+|v_0|_{X^2_{\epsilon^2}}^2\bigr).
\end{equation}

{\bf Step 3. The family of approximate solutions forms a Cauchy sequence in the  lower order space $C([0,T/\epsilon), X^0_\epsilon(\R)\times L^2(\R))$.}
For any $\delta,\delta'\in(0,\epsilon)$, we shall derive  estimates on $(\eta^\delta-\eta^{\delta'},v^\delta-v^{\delta'})$ in $C([0,T/\epsilon), X^0_\epsilon(\R)\times L^2(\R))$. Denoting by
\beno\begin{aligned}
E_0^{(\delta,\delta')}(t)&\underset{\text{def}}{=}|\eta^\delta-\eta^{\delta'}|_2^2+\epsilon|(\eta^\delta-\eta^{\delta'})_x|_2^2
+(\frac{v^\delta-v^{\delta'}}{1+\epsilon\mathcal{J}_\delta\eta^\delta}\,|\,v^\delta-v^{\delta'})_2,\\
&\sim|\eta^\delta-\eta^{\delta'}|_{X^0_\epsilon}^2+|v^\delta-v^{\delta'}|_2^2.
\end{aligned}\eeno

By \eqref{approximate system}, we deduce that
\beno
\frac{1}{2}\frac{d}{dt}E_0^{(\delta,\delta')}(t)=\sum_{i=1}^4 I_i,
\eeno
where
\beno\begin{aligned}
&I_1=-\{\bigl((\mathcal{J}_\delta v^\delta
-\mathcal{J}_{\delta'}v^{\delta'})_x\,|\,(1-\epsilon\partial_x^2)(\eta^\delta-\eta^{\delta'})\bigr)_2\\
&\qquad
+\bigl((1-\epsilon\partial_x^2)(\mathcal{J}_\delta\eta^\delta
-\mathcal{J}_{\delta'}\eta^{\delta'})_x\,|\,v^\delta-v^{\delta'}\bigr)_2\}\\
&I_2=-\bigl((\frac{1}{1+\epsilon\mathcal{J}_\delta\eta^\delta}-\frac{1}{1+\epsilon\mathcal{J}_{\delta'}\eta^{\delta'}})v^{\delta'}_t
\,|\,v^\delta-v^{\delta'}\bigr)_2\\
&I_3=-\frac{1}{2}\bigl(\partial_t\bigl(\frac{1}{1+\epsilon\mathcal{J}_\delta\eta^\delta}\bigr)(v^\delta-v^{\delta'})\,|\,v^\delta-v^{\delta'}\bigr)_2\\
&I_4=\epsilon\bigl(\mathcal{J}_{\delta'}^2\bigl(\frac{|\mathcal{J}_{\delta'}^2v^{\delta'}|^2}{1+\epsilon\mathcal{J}_{\delta'}\eta^{\delta'}}\bigr)_x
-\mathcal{J}_{\delta}^2\bigl(\frac{|\mathcal{J}_{\delta}^2v^{\delta}|^2}{1+\epsilon\mathcal{J}_{\delta}\eta^{\delta}}\bigr)_x
\,|\,v^\delta-v^{\delta'}\bigr)_2.
\end{aligned}\eeno

For $I_1$, integrating by parts, we obtain
\beno
I_1=-\bigl((1-\epsilon\partial_x^2)(\eta^\delta-\eta^{\delta'})\,|\,(\mathcal{J}_\delta-\mathcal{J}_{\delta'})v^{\delta'}_x\bigr)_2
-\bigl((\mathcal{J}_\delta-\mathcal{J}_{\delta'})\eta^{\delta'}_x\,|\,(1-\epsilon\partial_x^2)(v^\delta-v^{\delta'})\bigr)_2,
\eeno
which along with (iii) in Lemma \ref{lemma for mollifier} implies
\beno\begin{aligned}
|I_1|&\leq C\max\{\delta,\delta'\}\bigl(|\eta^\delta-\eta^{\delta'}|_{H^2}|v^{\delta'}_{x}|_{H^1}
+|\eta^{\delta'}_{x}|_{H^1}|v^\delta-v^{\delta'}|_{H^2}\bigr)\\
&\leq C\bigl(\mathcal{E}^\delta(t)+\mathcal{E}^{\delta'}(t)\bigr)\max\{\delta,\delta'\}\leq CM\max\{\delta,\delta'\},
\end{aligned}\eeno
where $M\underset{\text{def}}{=}C\bigl(|\eta_0|_{X^2_{\epsilon^3}}^2+|v_0|_{X^2_{\epsilon^2}}^2\bigr)$ the uniform bound for $\mathcal{E}^\delta(t)$ in \eqref{estimate for app 8}.

For $I_2$, we have
\beno\begin{aligned}
|I_2|&=\epsilon|\bigl(\frac{(\mathcal{J}_\delta-\mathcal{J}_{\delta'})\eta^\delta+\mathcal{J}_{\delta'}(\eta^\delta-\eta^{\delta'})}
{(1+\epsilon\mathcal{J}_\delta\eta^\delta)(1+\epsilon\mathcal{J}_{\delta'}\eta^{\delta'})}v^{\delta'}_t
\,|\,v^\delta-v^{\delta'}\bigr)_2|\\
&\leq C\epsilon\bigl(|(\mathcal{J}_\delta-\mathcal{J}_{\delta'})\eta^\delta|_2+|\mathcal{J}_{\delta'}(\eta^\delta-\eta^{\delta'})|_2\bigr)
|v^{\delta'}_t|_\infty|v^\delta-v^{\delta'}|_2\\
&\leq C\epsilon\bigl(\max\{\delta,\delta'\}|\eta^\delta|_{H^1}+|\eta^\delta-\eta^{\delta'}|_2\bigr)|v^{\delta'}_t|_{H^1}|v^\delta-v^{\delta'}|_2,
\end{aligned}\eeno
which implies
\beno\begin{aligned}
|I_2|\leq CM^{\frac{3}{2}}\max\{\delta,\delta'\}+C\epsilon M^{\frac{1}{2}}E_0^{(\delta,\delta')}(t).
\end{aligned}\eeno

Similarly, for $I_3$, we have
\beno\begin{aligned}
|I_3|\leq C\epsilon M^{\frac{1}{2}}E_0^{(\delta,\delta')}(t),
\end{aligned}\eeno
while for $I_4$, we have
\beno\begin{aligned}
|I_4|\leq C(M)\max\{\delta,\delta'\}+C(M)\epsilon E_0^{(\delta,\delta')}(t).
\end{aligned}\eeno

Thus, we have
\begin{equation}\label{estimate for app 9}
\frac{d}{dt}E_0^{(\delta,\delta')}(t)\leq C(M)\max\{\delta,\delta'\}+C(M) E_0^{(\delta,\delta')}(t).
\end{equation}
Noticing that $E_0^{(\delta,\delta')}(0)=0$, applying Gronwall inequality to \eqref{estimate for app 9} yields
\begin{equation}\label{estimate for app 10}
E_0^{(\delta,\delta')}(t)\leq C e^{C(M)t}\max\{\delta,\delta'\},
\end{equation}
which implies that
the approximate solutions $\{V^\delta=(\eta^\delta,v^\delta)\}_{\delta>0}$ form a Cauchy sequence in $C([0,T/\epsilon), X^0_\epsilon(\R)\times L^2(\R))$.

\medskip

{\bf Step 4. The limit of the approximate solutions solves \eqref{reduction form 1}.}
On the one hand, \eqref{estimate for app 10} shows that there exists a unique $V=(\eta,v)\in C([0,T/\epsilon), X^0_\epsilon(\R)\times L^2(\R))$ such that
when $\delta\rightarrow 0$,
\begin{equation}\label{limit 1}
V^\delta\rightarrow V\quad\text{in}\quad C([0,T/\epsilon), X^0_\epsilon\times L^2),
\end{equation}
and
\begin{equation}\label{limit 2}
\sup_{0\leq t\leq T/\epsilon}|V^\delta-V|_{X^0_\epsilon\times L^2}\leq C_{M,T,\epsilon}\delta,
\end{equation}
where $C_{M,T,\epsilon}$ is a constant depending on $M,\ T, \ \epsilon$.

On the other hand, by Banach-Alaoglu theorem, the uniform estimate \eqref{estimate for app 8} and (iv) of Lemma \ref{lemma for mollifier} imply that there exists a subsequence $\{V^{\delta_j}\}_{j\in\N}$ such that when $j\rightarrow\infty$,
\begin{equation}\label{limit 3}
V^{\delta_j}\rightharpoonup V,\quad\text{weakly},
\end{equation}
and the energy $\mathcal{E}(t)$ for $V$ has the same bound as in \eqref{estimate for app 8}. Moreover,
\begin{equation}\label{limit 4}\begin{aligned}
&V\in L^\infty([0,T/\epsilon);X^2_{\epsilon^3}\times X^2_{\epsilon^2})\cap Lip([0,T/\epsilon);X^1_{\epsilon^2}\times X^1_{\epsilon})\\
&V_t\in L^\infty([0,T/\epsilon);X^1_{\epsilon^2}\times X^1_{\epsilon})\cap Lip([0,T/\epsilon);X^0_{\epsilon}\times L^2).
\end{aligned}\end{equation}

Thanks to the interpolation inequality \eqref{int1}, we have for any $s\in(0,2]$,
\begin{equation}\label{limit 5}\begin{aligned}
\sup_{0\leq t\leq T/\epsilon}|V^\delta-V|_{X^{2-s}_{\epsilon^{3-s}}\times X^{2-s}_{\epsilon^{2-s}}}
&\leq C\sup_{0\leq t\leq T/\epsilon}|V^\delta-V|_{X^0_{\epsilon}\times L^2}^{\frac{s}{2}}|V^\delta-V|_{X^2_{\epsilon^3}\times X^2_{\epsilon^2}}^{1-\frac{s}{2}}\\
&\leq C_{M,T,\epsilon}\delta^{\frac{s}{2}},
\end{aligned}\end{equation}
where we used \eqref{limit 2}, \eqref{estimate for app 8} and the obtained result $\sup_{0\leq t\leq T/\epsilon}\mathcal{E}(t)\leq CM$.
By \eqref{limit 5}, we obtain that when $\delta\rightarrow 0$,
\begin{equation}\label{limit 6}
V^\delta\rightarrow V, \quad\text{in}\quad C([0,T/\epsilon);X^{2-s}_{\epsilon^{3-s}}\times X^{2-s}_{\epsilon^{2-s}}).
\end{equation}
If $s\in(0,\frac{1}{4})$, the embedding theorem shows that
\beno
C([0,T/\epsilon);X^{2-s}_{\epsilon^{3-s}}\times X^{2-s}_{\epsilon^{2-s}})\hookrightarrow C([0,T/\epsilon);C^4(\R)\times C^3(\R)).
\eeno
Then as $\delta\rightarrow 0$,
\beno
V^\delta\rightarrow V, \quad\text{in}\quad C([0,T/\epsilon);C^4(\R)\times C^3(\R)).
\eeno

Similarly, we could verify that as $\delta\rightarrow 0$,
\beno
V^\delta_t\rightarrow V_t, \quad\text{in}\quad C([0,T/\epsilon);C^2(\R)\times C^1(\R)).
\eeno

Thus, taking $\delta\rightarrow 0$ in \eqref{approximate system}, we obtain that $V=(\eta, v)$ satisfies \eqref{reduction form 1} in the classic sense.

\medskip

{\bf Step 5. Continuity in time of the solutions.} Firstly, by virtue of \eqref{limit 6}, we have
\begin{equation}\label{cont1}
\eta^\delta\rightarrow\eta\quad\text{in}\quad C([0,T/\epsilon);L^\infty(\R)).
\end{equation}
Thanks to \eqref{estimate for app 7} and \eqref{estimate for app 8}, we obtain that
\beno
\frac{d}{dt} E^\delta(t)\leq C(M),
\eeno
which implies that
\beno
E^\delta(t)-E^\delta(0)\leq C(M)t.
\eeno
Taking $\delta\rightarrow 0$ yields
\beno
E(t)-E^0\leq\limsup_{\delta\rightarrow 0}(E^\delta(t)-E^\delta(0))\leq C(M)t,
\eeno
where $E^0=E^\delta(0)|_{\delta=0}$ is determined only by $(\eta_0,v_0)$.
Then we have
\begin{equation}\label{cont2}
\limsup_{t\rightarrow 0}E(t)\leq E^0.
\end{equation}

On the other hand, thanks to \eqref{estimate for app 8} and \eqref{quasilinear for app 1}, we have
\beno
V^\delta_{ttt}\in L^\infty([0,T/\epsilon);\dot{H}^{-1}\times\dot{H}^{-1}),
\eeno
which along with \eqref{limit 4} implies
\begin{equation}\label{cont3}\begin{aligned}
&V\in C_w([0,T/\epsilon);X^2_{\epsilon^3}\times X^2_{\epsilon^2}),\quad V_t\in C_w([0,T/\epsilon);X^1_{\epsilon^2}\times X^1_{\epsilon}),\\
&V_{tt}\in C_w([0,T/\epsilon);X^0_{\epsilon}\times L^2).
\end{aligned}\end{equation}
Due to \eqref{cont1} and \eqref{cont3}, we have
\beno
 E^0\leq \liminf_{t\rightarrow 0}E(t),
\eeno
which along with \eqref{cont2} implies that
\begin{equation}\label{cont4}
\lim_{t\rightarrow 0}E(t)= E^0.
\end{equation}
Then by \eqref{cont1}, \eqref{cont3}, the definition of $E(t)$ and the arguments in Step 2.2 for the higher order derivatives in $x$, we have $V$ are strongly continuous in time $t=0$ in the corresponding functional spaces.

Consider $T_0\in (0,T/\epsilon)$ and the solution $V(\cdot,T_0)=(\eta(\cdot,T_0), v(\cdot,T_0))$. For fixed time $T_0$, $V^{T_0}\underset{\text{def}}{=}V(\cdot,T_0)\in X^2_{\epsilon^3}\times X^2_{\epsilon^2}$ and by \eqref{estimate for app 7} and \eqref{estimate for app 8}, there exists a constant $c_0$ which depends on $M\underset{\text{def}}{=}|\eta_0|_{X^2_{\epsilon^3}}^2+|v_0|_{X^2_{\epsilon^2}}^2$ such that
\begin{equation}\label{cont5}
\bigl(E(T_0)\bigr) \leq \frac{\bigl(E^0\bigr)^{\frac{1}{2}}}{1-c_0\epsilon T_0M^{\frac{1}{2}}}.
\end{equation}
Now we use $V^{T_0}$ as an initial data and construct a forward and backward in time solutions as in the above steps by solving the approximates system \eqref{approximate system}. We obtain the approximate solutions $V^\delta_{T_0}(\cdot,t)$ which also satisfy \eqref{estimate for app 7} and the limit $\widetilde{V}$ of $V^\delta_{T_0}(\cdot,t)$ also solves \eqref{reduction form 1} on some time interval $[T_0-T',T_0+T']$. By the uniqueness of the solutions, $\widetilde{V}$ must coincide with $V$ on the  time interval $[T_0-T',T_0+T']$. Similar to \eqref{cont5}, by using \eqref{estimate for app 8}, there exists a constant $c_0'$ which depends on $M\underset{\text{def}}{=}|\eta_0|_{X^2_{\epsilon^3}}^2+|v_0|_{X^2_{\epsilon^2}}^2$ such that for any $t\in[T_0-T',T_0+T']$,
\beno
\bigl(E(t)\bigr)^{\frac{1}{2}}\leq \frac{\bigl(E(T_0)\bigr)^{\frac{1}{2}}}{1-c_0'\epsilon(t-T_0)M^{\frac{1}{2}}}\leq\frac{\bigl(E^0\bigr)^{\frac{1}{2}}}{(1-c_0'\epsilon(t-T_0)M^{\frac{1}{2}})(1-c_0\epsilon T_0M^{\frac{1}{2}})}.
\eeno
Then we obtain the following restriction for $T'$
\beno
0< T'\leq\frac{1}{2c_0'M^{\frac{1}{2}}}.
\eeno
Following the same argument as the continuity in time $t=0$, we obtain that the solution is continuous in time $t=T_0$ in corresponding functional spaces.

Thus, we obtain that
\begin{equation}\label{cont6}\begin{aligned}
&V\in C([0,T/\epsilon);X^2_{\epsilon^3}\times X^2_{\epsilon^2}),\quad V_t\in C([0,T/\epsilon);X^1_{\epsilon^2}\times X^1_{\epsilon}),\\
&V_{tt}\in C([0,T/\epsilon);X^0_{\epsilon}\times L^2).
\end{aligned}\end{equation}

\medskip

Combining Steps 1 to 5, the existence proof of Theorem \ref{long time existence for case c=-1} is completed.

\section{Possible extensions}
\subsection{A fifth order Boussinesq system}
When the expansion with respect to $\epsilon$ is performed to the next order, one obtains a class of fifth order Boussinesq systems (see \cite{BCS1}). Those models should lead to an error estimate of order $O(\epsilon^3t)$ instead of $O(\epsilon^2t)$ for the usual Boussinesq systems. A rigorous proof of this fact requires in particular to establish that the fifth order Boussinesq systems are well-posed on "long" time scales and thus come the issue of long time existence for those systems. One expects of course a lifespan of at least order $1/\epsilon, $ the question, (as for the usual Boussinesq systems), being to see whether or not the dispersive terms allow to enlarge this lifespan. Due to the large number of equivalent (to the sense of consistency) systems, we will focus on a particular case (BBM-type) which is shown to be locally well-posed in \cite{BCS2}.

We first recall the fifth order Boussinesq system under study (one could obtain the following system from that stated in \cite{BCS1} by  scaling):
\begin{equation}\label{5rme1}\left\{\begin{aligned}
&(1- b\epsilon\partial_x^2+b_1\epsilon^2\partial_x^4)\eta_t+(1+a\epsilon\partial_x^2+a_1\epsilon^2\partial_x^4)u_x+\epsilon(1-b\epsilon\partial_x^2)(\eta u)_x\\
&\qquad+(a+b-\frac{1}{3})\epsilon^2(\eta u_{xx})_x=0,\\
&(1-d\epsilon\partial_x^2+d_1\epsilon^2\partial_x^4)u_t+(1+c\epsilon\partial_x^2+c_1\epsilon^2\partial_x^4)\eta_x+\epsilon(1+c\epsilon\partial_x^2)(u u_x)\\
&\qquad+\epsilon^2(\eta\eta_{xx})_x-(c+d-1)\epsilon^2u_xu_{xx}-(c+d)\epsilon^2uu_{xxx}=0.
\end{aligned}\right.\end{equation}

As an example  we shall deal with the "BBM-type" case:
\beno\begin{aligned}
&b\geq0,\quad b_1>0,\quad a<0,\quad a_1=0,\\
&d\geq0,\quad d_1>0,\quad c<0,\quad c_1=0.
\end{aligned}\eeno

We now state the existence  result in the BBM-type case.

\begin{theorem}\label{5rmT}
Let $s>\frac{3}{2}$ .
Assume that $(\eta_0,u_0)\in X^s_{\epsilon^3}(\R)$  satisfy the (non-cavitation) condition
\begin{equation}\label{5rme2}
1+\epsilon\eta_0\geq H>0,\quad H\in(0,1),
\end{equation}
Then there exists a constant $\tilde{c}_0$  such that for any $\epsilon\leq\epsilon_0=\frac{1-H}{\tilde{c}_0(| \eta_0|_{X^s_{\epsilon^3}}+| u_0|_{X^s_{\epsilon^3}})}$, there
exists $T>0$ independent of $\epsilon$, such that \eqref{5rme1}(the BBM-type case) with the initial data $(\eta_0,u_0)$ has a unique solution $(\eta,u)$ with
$(\eta, u)\in C([0,T/\epsilon];X^s_{\epsilon^3}(\R))$. Moreover,
\begin{equation}\label{5rme3}
\max_{t\in[0,T/\epsilon]} (|\eta|_{X^s_{\epsilon^3}}+|u|_{X^s_{\epsilon^3}})\leq \tilde c (|\eta_0|_{X^s_{\epsilon^3}}+| u_0|_{X^s_{\epsilon^3}}).
\end{equation}
\end{theorem}
\begin{proof}
The proof of the theorem is similar to that in the previous subsection and we only sketch it. Denoting by $U=(\eta,u)^T$, \eqref{5rme1}
(the BBM-type case) is equivalent to the following condensed system
\begin{equation}\label{5rme4}
M_0(\partial_x) U_t+M(U,\partial_x)U=0,
\end{equation}
where
\beno
M_0(\partial_x)=\text{diag}\bigl(1- b\epsilon\partial_x^2+b_1\epsilon^2\partial_x^4,1-d\epsilon\partial_x^2+d_1\epsilon^2\partial_x^4\bigr),
\eeno
\beno
M(U,\partial_x)=(m_{ij})_{i,j=1,2}\quad\text{with}
\eeno
\beno\begin{aligned}
&m_{11}=\epsilon(1-b\epsilon\partial_x^2)(u\partial_x)+(a+b-\frac{1}{3})\epsilon^2u_{xx}\partial_x,\\
&m_{12}=(1+a\epsilon\partial_x^2)\partial_x+\epsilon(1-b\epsilon\partial_x^2)(\eta\partial_x)+(a+b-\frac{1}{3})\epsilon^2\eta\partial_x^3,\\
&m_{21}=(1+c\epsilon\partial_x^2)\partial_x+\epsilon^2\partial_x(\eta\partial_x^2),\\
&m_{22}=\epsilon(1+c\epsilon\partial_x^2)(u\partial_x)-(c+d-1)\epsilon^2u_{xx}\partial_x-(c+d)\epsilon^2u\partial_x^3.
\end{aligned}\eeno

We could search the symmetrizer of both $M_0$ and $M$ as follows:
 \beno
 S=\text{diag}\bigl(1+c\epsilon\partial_x^2+\epsilon^2\eta\partial_x^2,1+\epsilon\eta+a\epsilon\partial_x^2+(a-\frac13)\epsilon^2\eta\partial_x^2\bigr).
 \eeno
We define the energy functional associated to \eqref{5rme4} as
\beno
E_s(U)=(M_0\Lambda^sU\,|\, S\Lambda^s U)_2.
\eeno
It is easy to check that under the assumption \eqref{5rme2} and the assumption that
\begin{equation}\label{5rme5}
\epsilon|\eta|_\infty+\epsilon|\partial_x\eta|_\infty\ll1,
\end{equation}
there holds
\begin{equation}\label{5rme0}
E_s(U)\sim|\eta|_{X^s_{\epsilon^3}}^2+|u|_{X^s_{\epsilon^3}}^2.
\end{equation}

As usual, we get by a standard energy estimate that
\begin{equation}\label{5rme6}\begin{aligned}
&\frac{d}{dt}E_s(U)=(M_0\Lambda^sU_t\,|\,(S+S^*)\Lambda^sU)_2+(\Lambda^sU\,|\,[M_0,S]\Lambda^sU_t)_2\\
&\qquad+(M_0\Lambda^sU\,|\,\partial_tS\Lambda^sU)_2\\
&=-(M(U,\partial_x)\Lambda^sU\,|\,(S+S^*)\Lambda^sU)_2-([\Lambda^s,M(U,\partial_x)]U\,|\,(S+S^*)\Lambda^sU)_2\\
&\qquad+(\Lambda^sU\,|\,[M_0,S]\Lambda^sU_t)_2+(M_0\Lambda^sU\,|\,\partial_tS\Lambda^sU)_2\\
&\underset{\text{def}}= I+II+III+IV.
\end{aligned}\end{equation}

{\bf Estimate for $I$.} We first compute
\beno\begin{aligned}
&\quad-(M(U,\partial_x)\Lambda^sU\,|\,S\Lambda^sU)_2\\
&=-(m_{11}\Lambda^s\eta\,|\,(1+c\epsilon\partial_x^2+\epsilon^2\eta\partial_x^2)\Lambda^s\eta)_2\\
&\quad-\{(m_{12}\Lambda^su\,|\,(1+c\epsilon\partial_x^2+\epsilon^2\eta\partial_x^2)\Lambda^s\eta)_2\\
&\quad+(m_{21}\Lambda^s\eta\,|\,(1+\epsilon\eta+a\epsilon\partial_x^2+(a-\frac13)\epsilon^2\eta\partial_x^2)\Lambda^su)_2\}\\
&\quad-(m_{22}\Lambda^su\,|\,(1+\epsilon\eta+a\epsilon\partial_x^2+(a-\frac13)\epsilon^2\eta\partial_x^2)\Lambda^su)_2\\
&\underset{\text{def}}= I_1+I_2+I_3.
\end{aligned}\eeno

For $I_1$, integrating by parts yields
\beno\begin{aligned}
&I_1=\frac{\epsilon}{2}(\partial_x\bigl(u+(a+b-\frac13)\epsilon u_{xx}\bigr)\Lambda^s\eta\,|\,\Lambda^s\eta)_2\\
&\quad+\frac{\epsilon^2}{2}(\partial_x\bigl((u+(a+b-\frac13)\epsilon u_{xx})(c+\epsilon\eta)\bigr)\partial_x\Lambda^s\eta\,|\,\partial_x\Lambda^s\eta)_2\\
&\quad-\frac{b\epsilon^2}{2}(\partial_xu\partial_x\Lambda^s\eta\,|\,\partial_x\Lambda^s\eta)_2
-\frac{b\epsilon^3}{2}(\partial_x\bigl((c+\epsilon\eta)u\bigr)\partial_x^2\Lambda^s\eta\,|\,\partial_x^2\Lambda^s\eta)_2\\
&\quad+b\epsilon^3(\partial_x^2u\partial_x\Lambda^s\eta\,|\,(c+\epsilon\eta)\partial_x^2\Lambda^s\eta)_2
+2b\epsilon^3(\partial_xu\partial_x^2\Lambda^s\eta\,|\,(c+\epsilon\eta)\partial_x^2\Lambda^s\eta)_2
\end{aligned}\eeno
which along with the interpolation inequality \eqref{int1} and the assumption \eqref{5rme5} implies that
\begin{equation}\label{5rme7}
|I_1|\lesssim\epsilon|u|_{X^s_{\epsilon^3}}|\eta|_{X^s_{\epsilon^3}}^2.
\end{equation}

For $I_2$, integrating by parts gives rise to
\beno\begin{aligned}
&I_2=(\epsilon\partial_x\eta\Lambda^su+(a-\frac13)\epsilon^2\partial_x\eta\partial_x^2\Lambda^su\,|\,
(1+c\epsilon\partial_x^2+\epsilon^2\eta\partial_x^2)\Lambda^s\eta)_2\\
&\qquad+b\epsilon^2(\partial_x^2\eta\partial_x\Lambda^su+2\partial_x\eta\partial_x^2\Lambda^su\,|\,
(1+c\epsilon\partial_x^2+\epsilon^2\eta\partial_x^2)\Lambda^s\eta)_2,
\end{aligned}\eeno
which along with \eqref{int1} and \eqref{5rme5} implies that
\begin{equation}\label{5rme8}
|I_2|\lesssim\epsilon|u|_{X^s_{\epsilon^3}}|\eta|_{X^s_{\epsilon^3}}^2.
\end{equation}

For $I_3$, we could derive  as for $I_1$ that
\begin{equation}\label{5rme9}
|I_3|\lesssim\epsilon|u|_{X^s_{\epsilon^3}}^3.
\end{equation}

Thanks to \eqref{5rme7}, \eqref{5rme8} and \eqref{5rme9}, we obtain
\beno
|(M(U,\partial_x)\Lambda^sU\,|\,S\Lambda^sU)_2|\lesssim\epsilon|u|_{X^s_{\epsilon^3}}(|\eta|_{X^s_{\epsilon^3}}^2
+|u|_{X^s_{\epsilon^3}}^2).
\eeno
The same estimate holds for term $(M(U,\partial_x)\Lambda^sU\,|\,S^*\Lambda^sU)_2$. Then we obtain
\begin{equation}\label{5rme10}
|I|\lesssim\epsilon|u|_{X^s_{\epsilon^3}}(|\eta|_{X^s_{\epsilon^3}}^2
+|u|_{X^s_{\epsilon^3}}^2).
\end{equation}

{\bf Estimate for $II$.} We first calculate that
\beno\begin{aligned}
|II|&\lesssim|[\Lambda^s,M(U,\partial_x)]U|_2|(S+S^*)\Lambda^sU|_2\\
&\lesssim|[\Lambda^s,M(U,\partial_x)]U|_2(|\eta|_{X^s_{\epsilon^3}}+|u|_{X^s_{\epsilon^3}})
\end{aligned}\eeno
provided that
\begin{equation}\label{5rme11}
\epsilon|\eta|_{X^s_{\epsilon^3}}<1.
\end{equation}
Thanks to the expression of $M(U,\partial_x)$, \eqref{int1} and Lemma \ref{L1}, we get that
\beno\begin{aligned}
&|[\Lambda^s,M(U,\partial_x)]U|_2\lesssim\epsilon|(1-b\epsilon\partial_x^2)\bigl([\Lambda^s,u]\partial_x\eta\bigr)|_2
+\epsilon^2|[\Lambda^s,u_{xx}]\partial_x\eta|_2\\
&\qquad+\epsilon^2|\partial_x\bigl([\Lambda^s,\eta]\partial_x^2\eta\bigr)|_2
+\epsilon|(1-b\epsilon\partial_x^2)\bigl([\Lambda^s,\eta]\partial_xu\bigr)|_2
+\epsilon^2|[\Lambda^s,\eta]\partial_x^3u|_2\\
&\qquad+\epsilon|(1+c\epsilon\partial_x^2)\bigl([\Lambda^s,u]\partial_xu\bigr)|_2+\epsilon^2|[\Lambda^s,u_{xx}]\partial_xu|_2
+\epsilon^2|[\Lambda^s,u]\partial_x^3u|_2\\
&\lesssim\epsilon(|\eta|_{X^s_{\epsilon^3}}^2+|u|_{X^s_{\epsilon^3}}^2).
\end{aligned}\eeno
Then we obtain that
\begin{equation}\label{5rme16}
|II|\lesssim\epsilon(|\eta|_{X^s_{\epsilon^3}}^2+|u|_{X^s_{\epsilon^3}}^2)^{\frac{3}{2}}.
\end{equation}

{\bf Estimate for $III$.} Using the expressions of $M_0$ and $S$, we have
\beno\begin{aligned}
&III=\epsilon^2(\Lambda^s\eta\,|\,(1- b\epsilon\partial_x^2+b_1\epsilon^2\partial_x^4)(\eta\partial_x^2\Lambda^s\eta_t)-\eta\partial_x^2
(1- b\epsilon\partial_x^2+b_1\epsilon^2\partial_x^4)\Lambda^s\eta_t)_2\\
&\quad+\epsilon(\Lambda^su\,|\,(1- d\epsilon\partial_x^2+d_1\epsilon^2\partial_x^4)(\eta\Lambda^su_t)-\eta
(1- d\epsilon\partial_x^2+d_1\epsilon^2\partial_x^4)\Lambda^su_t)_2\\
&\quad+(a-\frac13)\epsilon^2(\Lambda^su\,|\,(1- d\epsilon\partial_x^2+d_1\epsilon^2\partial_x^4)(\eta\partial_x^2\Lambda^su_t)-\eta\partial_x^2
(1- d\epsilon\partial_x^2+d_1\epsilon^2\partial_x^4)\Lambda^su_t)_2\\
&\underset{\text{def}}= III_1+III_2+III_3.
\end{aligned}\eeno

Now we rewrite $III_1,\ III_2,\ III_3$ as follows:
\beno\begin{aligned}
&III_1=-\epsilon^2([- b\epsilon\partial_x^2+b_1\epsilon^2\partial_x^4,\ \eta]\Lambda^s\eta\,|\, \partial_x^2\Lambda^s\eta_t)_2\\
&III_2=\epsilon(\Lambda^s u\,|\, [- d\epsilon\partial_x^2+d_1\epsilon^2\partial_x^4,\ \eta]\Lambda^s u_t)_2\\
&III_3=-(a-\frac{1}{3})\epsilon^2( [- d\epsilon\partial_x^2+d_1\epsilon^2\partial_x^4,\ \eta]\Lambda^s u\,|\,\partial_x^2\Lambda^s u_t)_2
\end{aligned}\eeno
which along with \eqref{int1} and Lemma \ref{L1}, we obtain
\begin{equation}\label{5rme14}
|III|\lesssim\epsilon(|\eta|_{X^s_{\epsilon^3}}^2+|u|_{X^s_{\epsilon^3}}^2)(|\eta_t|_{X^{s-1}_{\epsilon^4}}+|u_t|_{X^{s-1}_{\epsilon^4}}).
\end{equation}

{\bf Estimate for $IV$.} Using the expressions of $M_0$ and $S$, we get that
\beno\begin{aligned}
IV&=\epsilon^2((1- b\epsilon\partial_x^2+b_1\epsilon^2\partial_x^4)\Lambda^s\eta\,|\,\eta_t\partial_x^2\Lambda^s\eta)_2\\
&+((1- d\epsilon\partial_x^2+d_1\epsilon^2\partial_x^4)\Lambda^su\,|\,\epsilon\eta_t\Lambda^su+(a-\frac13)\epsilon^2\eta_t\partial_x^2\Lambda^su)_2,
\end{aligned}\eeno
which along with \eqref{int1} implies that
\begin{equation}\label{5rme15}\begin{aligned}
|IV|&\lesssim\epsilon(|\eta_t|_\infty+\epsilon^{\frac12}|\partial_x\eta_t|_\infty)(|\eta|_{X^s_{\epsilon^3}}^2+|u|_{X^s_{\epsilon^3}}^2),\\
&\lesssim\epsilon|\eta_t|_{X^{s-1}_{\epsilon^4}}(|\eta|_{X^s_{\epsilon^3}}^2+|u|_{X^s_{\epsilon^3}}^2).
\end{aligned}\end{equation}

Thanks to \eqref{5rme6}, \eqref{5rme10}, \eqref{5rme16}, \eqref{5rme14} and \eqref{5rme15}, we get that
\beno
\frac{d}{dt}E_s(U)\lesssim\epsilon(|\eta|_{X^s_{\epsilon^3}}+|u|_{X^s_{\epsilon^3}}+|\eta_t|_{X^{s-1}_{\epsilon^4}}
+|u_t|_{X^{s-1}_{\epsilon^4}})
(|\eta|_{X^s_{\epsilon^3}}^2+|u|_{X^s_{\epsilon^3}}^2).
\eeno
Go back to \eqref{5rme1}, we get that
\beno
|\eta_t|_{X^{s-1}_{\epsilon^4}}+|u_t|_{X^{s-1}_{\epsilon^4}}
\lesssim(|\eta|_{X^s_{\epsilon^3}}+|u|_{X^s_{\epsilon^3}})(1+\epsilon|\eta|_{X^s_{\epsilon^3}}
+\epsilon|u|_{X^s_{\epsilon^3}})\lesssim|\eta|_{X^s_{\epsilon^3}}+|u|_{X^s_{\epsilon^3}},
\eeno
provided that
\begin{equation}\label{5rme17}
\epsilon|\eta|_{X^s_{\epsilon^3}}<1.
\end{equation}

Then we get that
\beno
\frac{d}{dt}E_s(U)\lesssim\epsilon(|\eta|_{X^s_{\epsilon^3}}+|u|_{X^s_{\epsilon^3}})
(|\eta|_{X^s_{\epsilon^3}}^2+|u|_{X^s_{\epsilon^3}}^2).
\eeno
which along with \eqref{5rme0} implies that
\begin{equation}\label{5rme18}
\frac{d}{dt}E_s(U)\lesssim\epsilon E_s(U)^{\frac32}.
\end{equation}

Thus, using similar arguments as in the previous subsections, we can deduce from \eqref{5rme18} that  there exists $T>0$ independent of $\epsilon$ such that \eqref{5rme1}(the BBM-type case) has a unique solution on time interval $[0,T/\epsilon]$ with initial data $(\eta_0,u_0)$. Moreover \eqref{5rme3} holds . Theorem \ref{5rmT}
is proved.
\end{proof}

\subsection{A Boussinesq-Full dispersion system for internal waves}
A systematic derivation of asymptotic internal waves models describing waves at the interface of a two-fluids system with a rigid top is given in \cite{BLS}. We will consider here a specific regime leading to a Boussinesq-Full dispersion system for which the  long time existence result is still open. We recall first the relevant  parameters. The index $1$ stands for the upper layer and $2$ for the lower one.

$\gamma= \frac{\rho_1}{\rho_2}<1$ is the ratio of densities, $\delta=\frac{d_1}{d_2}$ the ratio of typical heights of the layers, $\lambda$ a typical wavelength, $a$ a typical amplitude of the wave.

We denote $\epsilon=\frac{a}{d_1}, \mu=\frac{d_1^2}{\lambda}, \epsilon_2=\frac{a}{d_2}=\epsilon \delta, \mu_2=\frac{d_2^2}{\lambda^2}=\frac{\mu}{\delta^2}.$

We consider here the regime where $\mu\sim \epsilon \ll 1, (\mu=\epsilon \ll1$ from now on) and $\mu_2\sim 1.$

It is shown in \cite{BLS}    that in this regime (for which one also has
$\delta^2\sim\epsilon$ and thus $\epsilon_2\sim \epsilon^{3/2}\ll 1$), and in absence of surface tension, the
two-layers system  is consistent with the
\emph{three-parameter family} of Boussinesq/FD systems
\begin{equation}\label{eqB-FD}
         \left\lbrace
                  \begin{array}{l}
                   (1-\mu b\Delta)\partial_t \zeta+
                  \frac{1}{\gamma}\nabla\cdot\big((1-\epsilon\zeta){\bf v}_\beta\big)\vspace{1mm}\\
     \indent-\frac{\sqrt{\mu}}{\gamma^2}|D|\coth(\sqrt{\mu_2}|D|)\nabla\cdot{\bf v}_\beta
+\frac{\mu}{\gamma}\Big(a-\frac{1}{\gamma^2}\coth^2(\sqrt{\mu_2}|D|)\Big)\Delta\nabla\cdot{\bf
v}_\beta = 0\vspace{1mm}\\
                (1-\mu d\Delta)\dt\mathbf{v}_\beta +(1-\gamma)\nabla\zeta-\frac{\epsilon}{2\gamma}\nabla\vert{\bf v}_\beta\vert^2
    +\mu c(1-\gamma)\Delta\nabla\zeta=0,
            \end{array}\right.
\end{equation}
where ${\bf v}_\beta=(1-\mu\beta \Delta)^{-1}{\bf v}$ and the
constants $a$, $b$, $c$ and $d$ are defined as
    $$
    a=\frac{1}{3}(1-\alpha_1-3\beta),\qquad
    b=\frac{1}{3}\alpha_1,\qquad
    c=\beta\alpha_2,\qquad
    d=\beta(1-\alpha_2),
    $$
    with $\alpha_1\geq 0$, $\beta\geq 0$ and $\alpha_2\leq 1$. Note that $a+b+c+d=\frac{1}{3}.$

It is easily checked that \eqref{eqB-FD} is linearly well posed when

$$a\leq 0, c\leq 0, b\geq 0, d\geq 0.$$

The local well-posedness of the Cauchy problem for  \eqref{eqB-FD} was considered in \cite{CTA2} in the following cases

\vspace{0.3cm}
 \begin{itemize}
\item[(1)]  $b>0,d>0, a\leq 0, c<0$;
\item[(2)]   $b>0 ,d>0, a\leq 0, c=0$;
\item[(3)]  $b=0, d>0, a\leq 0, c=0$;
\item[(4)] $b=0, d>0, a\leq 0, c<0$;
\item[(5)]  $b>0, d=0, a\leq 0, c=0$.
\end{itemize}

On the other hand, we do not know of any   {\it long time existence} results for \eqref{eqB-FD} that is existence on time scales of order $1/\epsilon.$ This issue will be considered in a subsequent paper \cite{SX3}.

\subsection{A full dispersion Boussinesq system.}
One obtains a {\it full dispersion system} when in the Boussinesq regime by keeping the original dispersion of the water waves system (see \cite{La1}, \cite{DKM}, and \cite{AMP} where interesting numerical simulations of the propagation of solitary waves are performed). \footnote{As noticed in \cite{AMP} the use of nonlocal models for shallow water waves is also suggested in \cite{Za}.}

They read, setting $\mathcal T_\epsilon =\frac{\tanh \sqrt \epsilon |D|}{\sqrt \epsilon |D|}, \;D=-i\nabla :$
\begin{equation}
    \label{FD1d}
    \left\lbrace
    \begin{array}{l}
    \eta_t+\mathcal T_\epsilon u_x+\epsilon (\eta u)_x=0 \\
     u_t+ \eta_x+\epsilon uu_x=0,
 \end{array}\right.
    \end{equation}
when $d=1$ and
\begin{equation}
    \label{FD2d}
    \left\lbrace
    \begin{array}{l}
    \eta_t+\mathcal T_\epsilon \nabla\cdot \vv u+\epsilon \nabla \cdot(\eta \vv u)=0 \\
    \vv u_t+\nabla \eta+\frac{\epsilon}{2}\nabla |\vv u|^2=0,
 \end{array}\right.
    \end{equation}
when $d=2.$

    Taking  the limit $\sqrt \epsilon |\xi| \to 0$ in $\mathcal T_\epsilon,$  \eqref {FD1d} reduces formally to

    \begin{equation}
    \label{Kaup}
    \left\lbrace
    \begin{array}{l}
    \eta_t+u_x+\frac{\epsilon}{3} u_{xxx}+\epsilon (\eta u)_x=0 \\
     u_t+ \eta_x+\epsilon uu_x=0,
 \end{array}\right.
    \end{equation}
 while in the two-dimensional case,  \eqref {FD2d} reduces in the same limit to


\begin{equation}
    \label{FD2dbis}
    \left\lbrace
    \begin{array}{l}
    \eta_t+\nabla \cdot {\bf u}+\frac{\epsilon}{3}\Delta \nabla\cdot {\bf u}+\epsilon \nabla \cdot(\eta \vv u)=0 \\
    \vv u_t+\nabla \eta+\frac{\epsilon}{2}\nabla |\vv u|^2=0,
 \end{array}\right.
    \end{equation}
    that is to the (linearly ill-posed) system one gets first by expanding to first order the Dirichlet to Neumann operator with respect to $\epsilon$ in the full water wave system (see \cite{La1}).

 System  \eqref{Kaup}   is also known in the Inverse Scattering community as   the Kaup system (see \cite{Ka, Ku}). It is completely integrable though linearly ill-posed since the eigenvalues of the dispersion matrix are  $\pm i\xi(1-\frac{\epsilon}{3}\xi^2)^{1/2}.$
    The Boussinesq system \eqref{FD1d} can therefore be seen as a (well-posed) regularization of the Kaup system. Whether or not it is completely integrable is an open question.

The full dispersion Boussinesq systems have  the following Hamiltonian structure
\begin{displaymath}
\partial_t\begin{pmatrix} \eta\\{\vv u}
\end{pmatrix}+J\text{grad}\;H_\epsilon(\eta,{\vv u})=0
\end{displaymath}
where
\begin{displaymath}
J=\begin{pmatrix} 0 & \partial_{x} & \partial_{y} \\
                 \partial_{x} & 0 & 0 \\
                 \partial_{y} & 0 & 0 \end{pmatrix},
\end{displaymath}
\begin{displaymath}
H_{\epsilon}(U)=\frac12 \int_{\mathbb R^2}\big(|\mathcal T_\epsilon^{1/2} {\vv u}|^2+\eta^2+\epsilon
\eta|{\vv u}|^2\big)dxdy,
\end{displaymath}
\begin{displaymath}
U=\begin{pmatrix} \eta  \\
                {\vv u}
                  \end{pmatrix},
\end{displaymath}
when $d=2$ and
\begin{displaymath}
\partial_t\begin{pmatrix} \eta\\u
\end{pmatrix}+J\text{grad}\;H_\epsilon(\eta,u)=0
\end{displaymath}
where
\begin{displaymath}
J=\begin{pmatrix} 0&\partial_x\\
\partial_x&0
\end{pmatrix}
\end{displaymath}
and
$$H_\epsilon(\eta,u)=\frac{1}{2}\int_\R (|\mathcal T_\epsilon^{1/2}u|^2+\eta^2+\epsilon u^2\eta)dx,$$
when $d=1.$

Note that  the full dispersion Boussinesq system \eqref{FD1d} can be viewed  as the two-way propagation counterpart of   the Whitham equation (see \cite{W} and \cite{La1} for a rigorous derivation):
\begin{equation}\label{Whit}
\eta_t+\left(\mathcal T_\epsilon \right)^{1/2} u_x+\epsilon uu_x=0
\end{equation}
which displays a very rich dynamics (see \cite{EGW, KS} and the references therein).

When surface tension is taken into account, one should replace the operator $\mathcal T_\epsilon$ by $\mathcal P_\epsilon= (I+\beta\epsilon|D|^2)^{1/2}\left(\frac{\tanh(\sqrt \epsilon |D|)}{\sqrt \epsilon |D|}\right)$ where the parameter $\beta>0$ measures surface tension (see \cite{La1}), yielding a more dispersive full dispersion Boussinesq system. When $\beta>\frac{1}{3},$ this full dispersion Boussinesq system yields, taking  the limit  $\sqrt \epsilon |\xi| \to 0$ in $\mathcal P_\epsilon,$ Boussinesq systems of the class
$a<0, b=c=d=0$ for which long time well-posedness is established in Theorem 4.5.

Again we refer to a future work \cite{SX3} for the study of the Cauchy problem associated to \eqref{FD1d}, \eqref{FD2d}.
\section{Concluding remarks}

1. So far we have encounter only two possibilities for the lifespan $T_\epsilon$ of solutions to Boussinesq systems. Either $T_\epsilon=+\infty,$ for a few one-dimensional systems, or $T_\epsilon=O(1/\epsilon)$ for essentially all the admissible (linearly well-posed) systems. One may ask whether another possibility might occur. In view of  what happens in the scalar case (the {\it fractionary KdV equation}, see \cite{LPS2}) one could conjecture that there is no other possibility, at least in the one-dimensional case and when the natural no cavitation condition is imposed on the initial data. Note that no general criteria preventing blow-up in finite time seem to be known for Boussinesq systems except in the one-dimensional "BBM/BBM" system ($a=c=0, b>0, d>0$) for which it is proven in \cite{AABCW} that a uniform control on $|1+\epsilon \eta(\cdot,t)|_\infty$ prevents finite time blow-up.

2. Coming back to \eqref{abcd2}, we remark that all long time existence results in the present paper hold true for \eqref{abcd2} if one fixes $\mu>0$ and let $\epsilon$ tends to $0$.

\vspace{0.3cm}
\noindent {\bf Acknowledgments.}  The first Author acknowledges support from the ANR project GEODISP.
He thanks Vincent Duch\^{e}ne for fruitful discussions.
The third author is partially
supported  by NSF of China under grant 11201455 and by innovation grant from National Center for Mathematics and Interdisciplinary Sciences.
The three  authors would like to appreciate the hospitality and financial support from Morningside Center of Mathematics, CAS.

\bibliographystyle{amsplain}

\end{document}